\documentclass[12pt]{article}
\usepackage{amssymb,amsmath,amsthm,mathrsfs}
\usepackage[usenames]{color}
\usepackage{hyperref}
\usepackage{tikz}
\usepackage{tikz-cd}
\usepackage[all]{xy}
\usepackage[mathscr]{euscript}
\usepackage{enumerate}
\usepackage{subcaption}
\usepackage[shortlabels]{enumitem}

\newtheorem{theorem}{Theorem}[subsection]
\newtheorem{proposition}[theorem]{Proposition}
\newtheorem{corollary}[theorem]{Corollary}
\newtheorem{lemma}[theorem]{Lemma}
\newtheorem*{lemma*}{Lemma}

\theoremstyle{definition}
\newtheorem{definition}[theorem]{Definition}
\newtheorem{notation}[theorem]{Notation}

\newtheorem{example}[theorem]{Example}
\newtheorem{conjecture}[theorem]{Conjecture}

\theoremstyle{remark}
\newtheorem{remark}[theorem]{Remark}

\DeclareMathOperator{\Aut}{Aut}

\DeclareMathOperator{\Mat}{Mat}
\DeclareMathOperator{\Span}{Span}
\DeclareMathOperator{\GL}{GL}
\DeclareMathOperator{\SL}{SL}
\DeclareMathOperator{\SO}{SO}
\DeclareMathOperator{\Spin}{Spin}
\DeclareMathOperator{\Cliff}{Cliff}

\DeclareMathOperator{\Gal}{Gal}
\DeclareMathOperator{\val}{val}
\DeclareMathOperator{\Cl}{Cl}
\DeclareMathOperator{\Spec}{Spec}
\DeclareMathOperator{\supp}{supp}
\DeclareMathOperator{\gcl}{gcl}
\DeclareMathOperator{\wid}{wid}

\DeclareMathOperator{\Sym}{Sym}

\renewcommand{\O}{\operatorname{O}}

\title{Conjugacy width in uniform higher rank arithmetic groups}
\author{Nir Avni and Chen Meiri}
\begin{document}

\maketitle

\begin{abstract} We study widths of conjugacy classes in anisotropic higher rank $S$-arithmetic groups of orthogonal type. Assuming the GRH, we prove that many such groups have bounded conjugacy width. For example, this holds if the degree is greater or equal to $ 17$ and $S$ contains a non-archimedean place. To the best of our knowledge, this is the first boundedness result proved for anisotropic groups.

The proof uses ideas from the Congruence Subgroup Problem. In particular, we define and compute a non standard version of the metaplectic kernel. Conversely, we prove that a quantitative bound on the width of conjugacy classes implies the CSP.

The machinery we develop can also be used for other width questions. For example, in \cite{AM25} we prove, unconditional on GRH, new cases of bounded generation of arithmetic groups.
\end{abstract}

\section{Introduction}


\subsection{Main Result}

Let $\Gamma$ be a group. For a subset $X \subseteq \Gamma$, we say that the width of $X$ is $n$ if every element in the subgroup generated by $X$ is a product of at most $n$ elements of $X \cup X ^{-1}$ and $n$ is the minimal such natural number. We say that the width of $X$ is infinite if no such $n$ exist. The width of $X$ in $\Gamma$ is denoted by $\wid_\Gamma(X)$. 

Group theorists have been studying width for a long time. Of special interest are determining the widths of verbal sets (see Definition \ref{def:w}), conjugacy classes, and generating sets. For example: \begin{itemize} 
\item A classical theorem is that the width of the set of reflections in an orthogonal group over a field is equal to the degree.
\item A conjecture of Thompson \cite[9.24]{KhMa20} states that every finite simple group has a non-trivial conjugacy class whose width is 2.
\item A theorem of Borel \cite{Bor83} implies that the width of any verbal subset in an algebraic group over an algebraically closed field is at most 2.
\item  Roman'kov \cite{Rom} proved that the width of any verbal subset in a finitely generated nilpotent group is finite.
\item In \cite{NiSe07a} and \cite{NiSe07b}, Nikolov and Segal studied widths of verbal subsets in general finite groups. 
For example, they proved that if $G$ is a finite group and $w$ is a power word or the commutator word, then the width of $w(G)$ is bounded by a function of the number of generators of $G$. This result is the key step in proving Serre's conjecture that a finite index subgroup of a finitely generated profinite group is open. 
\item Larsen and Shalev \cite{LaSh09} studied widths of verbal subsets in finite simple groups. They proved that, for every word $w$, $\wid_G(w(G))\le 2$ if $G$ is a finite simple group of sufficiently large size.
\item In \cite{Jai08}, Jaikin-Zapirain determined which  words have finite width in free pro-$p$ groups.
\end{itemize}

In this paper, we study width in $S$-arithmetic groups. As is well known, there is a sharp contrast between arithmetic groups of $S$-rank one and arithmetic groups of higher $S$-rank. The following results suggest that this contrast is present also in the context of width:
\begin{itemize}
\item Bestvina, Bromberg, and Fujiwara show in \cite{BBF19} that the width of any verbal subset in an acylindrically hyperbolic group is either 1 or infinite (see \cite{Rhe68,MyNi14} for the cases of free groups and hyperbolic groups, respectively). In particular, this applies to lattices in rank one Lie groups.
\item Carter and Keller proved in \cite{CK83} that the width of the set of elementary matrices in $\SL_n(\mathbb{Z})$, $n \geq 3$, is finite. It follows that the width of the conjugacy class of the elementary matrix $e_{1,n}$ is finite. From this, one can deduce that the width of any conjugacy class and every verbal subset of $\SL_n(\mathbb{Z})$ is finite, see \cite{AM19}. The same holds for other Chevalley groups by \cite{Tav90}.
\end{itemize} 

For non-Chevalley groups, finite width of the set of elementary matrices was conjecturally replaced by Bounded Generation, i.e., the property \begin{itemize}
\item[(BG)] There exist elements $g_1,\ldots,g_n\in G$ such that $\langle g_1 \rangle \cdots \langle g_n \rangle = G$.
\end{itemize}
For many years, proving BG for non-Chevalley groups was a central open problem in the theory of arithmetic groups, a main motivation being that BG  implies the Congruence Subgroup Problem (see \cite{PlRa93,Lub95}). For other consequences of BG, see, for example, \cite{Sha01}.

Bounded Generation for various isotorpic orthogonal groups of higher rank was proved in \cite{ER06}, but there was no example of a non-isotropic arithmetic group with BG. It came as a surprise that BG fails for non-isotropic groups. In fact, Corvaja--Rapinchuk--Ren--Zannier prove in \cite{CRRZ22} that BG does not hold (even for Chevalley groups) unless some of the $g_i$s are virtually unipotent.

A main idea in this paper is to replace BG by finite width of conjugacy classes. 
\begin{itemize}
\item For higher rank isotropic groups, these properties are closely related. For example, in $\SL_n(\mathbb{Z})$ ($n \geq 3$), BG (i.e., finite width of the set of elementary matrices) is equivalent to finite width of the conjugacy class of $e_{1,n}(1)$. The implication `finite width of elementary matrices implies finite width of the conjugacy class of $e_{1,n}$' is a simple computation; the converse was proved by Suslin in \cite{Sus77}.
\item For anisotropic groups, BG fails. In contrast, we prove that, under the GRH, finite conjugacy width holds in many uniform higher rank arithmetic groups. Moreover, a quantitative version of finite conjugacy width (see Theorem \ref{thm:main}) shares many consequences with BG, including the Congruence Subgroup Property. We stress that all the cases where CSP in higher rank groups is unknown are anisotropic.\\
\end{itemize} 

To state our first result, we use the following notations: For an element $g$ in a group $G$, the generalized conjugacy class of $g$, denoted by $\gcl_G(g)$, is the union of the conjugacy class of $g$, the conjugacy class of $g ^{-1}$, and $\left\{ 1 \right\}$. We say that a group $G$ has finite conjugacy width if the width of every generalized conjugacy class in $G$ is finite. We denote the Witt index of a quadratic space $(V,q)$ by $i_q(V)$. 

\begin{theorem} \label{thm:main} Assume the Generalized Riemann Hypothesis holds. Let $K$ be a number field, let $S$ be a finite set of places of $K$ containing all archimedean places, let $n \geq 17$, let $f_a$ be an anisotropic quadratic form on $K^n$, and let $\Gamma \subseteq \Spin_{f_a}(K)$ be an $S$-arithmetic subgroup. Assume that one of the following holds: \begin{enumerate}
\item $S$ contains a nonarchimedean place.
\item $f_a$ splits over the ring of $S$-adeles, the $S$-congruence kernel of $\Spin_{f_a}$ is trivial, and $S$ contains a place $w_0$ such that $i_{f_a}(K_{w_0}^n) \geq 2$.
\end{enumerate}
Then $\Gamma$ has finite conjugacy width. Moreover, there is a constant $C$ such that, for every $\gamma \in \Gamma$,
\[
\wid_\Gamma (\gcl_\Gamma (\gamma)) \leq C \cdot \max \left\{ \wid_{\Spin_{f_a}(K_v)}\left( \gcl_{\Spin_{f_a}(K_v)} (\gamma) \right) \mid K_v\cong \mathbb{R} \text{ and $f_a$ is definite over $K_v$} \right\}.
\]
\end{theorem} 

We conjecture that Theorem \ref{thm:main} holds unconditionally and more generally:

\begin{conjecture} \label{conj:local.to.global.one.group} Let $K$ be a global field, let $S$ be a finite set of places of $K$ containing all archimedean places, and denote the ring of $S$-integers of $K$ by $O_S$. Let $\mathbf{G}$ be a connected and simply connected semisimple $K$-group of $S$-rank at least two and let $\Gamma$ be a congruence subgroup of $\mathbf{G}(O_S)$. There is a constant $C$ such that, for every $\gamma \in \Gamma $, 
\begin{equation} \label{eq:local.to.global.conj}
\wid_\Gamma \left( \gcl_\Gamma(\gamma) \right) \leq C \cdot \max \left\{ \wid_{\mathbf{G}(K_v)} \left( \gcl_{\mathbf{G}(K_v)}(\gamma) \right) \mid \text{$v|\infty$, and $\mathbf{G}(K_v)$ compact}\right\}.
\end{equation}
\end{conjecture}

Our second result relates finite conjugacy width and the Congruence Subgroup Property:

\begin{theorem}\label{thm:main2} Let $\Gamma$ be as in Conjecture \ref{conj:local.to.global.one.group} and assume that $\mathbf{G}(K_v)$ is noncompact for every nonarchimedean place $v\in S$. Then $\Gamma$ has the Congruence Subgroup Property if and only if its profinite completion $\widehat{\Gamma}$ has finite conjugacy width.
\end{theorem} 

\begin{remark} \label{rem:CSPandFCW} \begin{enumerate}
\item By \cite{AM22}, the right hand side of Inequality \eqref{eq:local.to.global.conj} is finite.
\item Conjecture \ref{conj:local.to.global.one.group} implies the Congruence Subgroup Problem for higher rank groups. We will prove in Proposition \ref{prop:Conj.implies.FCW} that if $\Gamma$ is as in Conjecture \ref{conj:local.to.global.one.group} and Inequality \eqref{eq:local.to.global.conj} holds, then $\widehat{\Gamma}$ has finite conjugacy width. Under the assumptions of Theorem \ref{thm:main2}, $\Gamma$ has the Congruence Subgroup Property.
\end{enumerate} 
\end{remark} 

\begin{remark} If $u$ is a unipotent element, then the map $n\mapsto u^n$ is polynomial. Thus, Bounded Generation with unipotents is a question about solvability of polynomial Diophantine equations. For a semisimple element $g$, the map $n\mapsto g^n$ is an exponential map in $n$, and Bounded Generation with semisimple elements is a question about solvability of exponential Diophantine equations. Exponential Diophantine equations are rarely solvable, and this is a bird-eye reason for the failure of Bounded Generation without unipotents.

On the other hand, for every element $g$, either unipotent or semisimple, the map $h \mapsto h g h ^{-1}$ is a polynomial map in the entries of $h$, so finite conjugacy width is a question about solvability of polynomial Diophantine equations. Thus, finite conjugacy width is closer to Bounded Generation with unipotents, lending support to Conjecture \ref{conj:local.to.global.one.group}.
\end{remark} 

\subsection{Applications}

\subsubsection{Bounded Generation}

While it is believed that all higher rank isotropic groups boundedly generated, this conjecture is only known in few cases (see \cite{CK83} and \cite{Tav90} for Chevalley groups and \cite{ER06} for groups of orthogonal type). The methods we develop in this paper can be used to prove, unconditonal on GRH, new cases of BG; in particular, in \cite{AM25} we treat special linear groups over orders in  division algebras. That paper can also be read as a simple introduction to the techniques we develop in the present paper.

\subsubsection{Verbal width} \label{ssec:verbal}

\begin{definition}\label{def:w}
\begin{enumerate}[(a)]
\item A word $w=u(x_1,\ldots,x_k)$ on $k$ letters in an element of the free group  $F(x_1,\ldots,x_k)$  with basis $x_1,\ldots,x_k$. For example, $x_1x_2^{-2}x_3$ is a word. 
\item Let $\Gamma$ be a group and let $w$ be a non-trivial word on $k$ letters. For every $\bar{g}=(g_1,\ldots,g_k) \in \Gamma^k$ we denote by $w(\bar{g})$ the element of $\Gamma$  obtained by substituting  each $x_i$ appearing in $w$ by $g_i$. For example, if $w=x_1x_2^{-2}x_3$ then $w(\bar{g})=g_1g_2^{-2}g_3$.
\item If $\Gamma$ is a group and $w$ is a word, the verbal subset of $w$ is the set $w(\Gamma):=\{w(\bar{g})^\epsilon \mid \bar{g}\in \Gamma^k,\epsilon=\pm 1\}$. The verbal subgroup of $w$ is the subgroup $\Gamma_w:=\langle w(\Gamma)\rangle$. 
\item The width $\wid_w(\Gamma)$ of $w$ in $\Gamma$ is defined to be $\wid_\Gamma(w(\Gamma))$.
\end{enumerate}
\end{definition}

Theorem \ref{thm:main} gives a new result about verbal width, namely, that for many orthogonal groups $G$, the width of every verbal set in every lattice of $G$ is finite (assuming GRH). More precisely, Theorem \ref{thm:main} implies the following:

\begin{corollary}\label{cor:width} Assume the GRH. Let $K$, $S$, $f_a$, and $\Gamma$ be as in Theorem \ref{thm:main}. There exists a constant $C$  such that, for every word $w$,
\[
\wid_{\Gamma}(w(\Gamma)) \le C \cdot \max \left\{ \wid_{\Spin_{f_a}(K_v)}\big(w(\Spin_{f_a}(K_v))\big) \mid K_v\cong \mathbb{R} \right\}.
\]
\end{corollary} 
Corollary \ref{cor:width} leads to the following conjecture: 
\begin{conjecture}\label{conj:word1} Let $\Gamma$ be as in Conjecture \ref{conj:local.to.global.one.group}. There exists a constant $C$ such that, for every word $w$,
$$\wid_{\Gamma}(w(\Gamma)) \le C \cdot \max \left\{ \wid_{\mathbf{G}(K_v)}\big(w(\mathbf{G}(K_v)\big) \mid K_v \cong \mathbb{R} \right\} $$
\end{conjecture}
\begin{remark}
	In \cite{BuMo99}, Burger and Monod asked if in every higher rank arithmetic group, the commutator word has a finite width. Corollary  \ref{cor:width} is the first solution of this question  (positive or negative) for a uniform higher rank arithmetic group. 
\end{remark}

\begin{remark} \begin{enumerate}
\item The main theorem of \cite{AM19} implies that Conjecture \ref{conj:word1} holds for $\Gamma = \SL_n(\mathbb{Z})$, if $n \geq 3$. 
\item Note the dichotomy between rank 1 groups (where the width is usually infinite by \cite{BBF19}) and higher rank groups (Conjecture \ref{conj:word1}).
\end{enumerate} 
\end{remark} 

\subsubsection{Model Theory of arithmetic groups} 

Theorem \ref{thm:main} also has applications to the model theory of lattices in orthogonal groups. Namely, if a group $\Gamma$ satisfies the conclusion of Theorem \ref{thm:main}, then $\Gamma$ is bi-interpretable with the ring $\mathbb{Z}$.

The general notion of bi-interpretability is explained in \cite{AM22}. It is an equivalence relation on the class of all structures of all first-order languages. Roughly, two structures are bi-interpretable if their categories of definable sets are equivalent.

Bi-interpretation with $\mathbb{Z}$ (as a structure of the language of rings) has several interesting consequences. For example, a corollary of \cite{Khe07} is that if a finitely generated group $\Gamma$ is bi-interpretable with $\mathbb Z$, then there exists a first order sentence in the  language of groups such that $\Gamma$ is the only finitely generated group which satisfies it. This property was called quasi-finitely axiomatizable in \cite{Nie07}. Another consequence is that if $\Gamma$ is a finitely generated group which is bi-interpretable with $\mathbb Z$, then every finitely generated subgroup of $\Gamma$ is  definable.  

We proved in \cite{AM22} that if $\Gamma$ is a higher rank arithmetic group of orthogonal type which satisfies some technical conditions, then $\Gamma/Z(\Gamma)$ is bi-interpretable with $\mathbb Z$. Moreover, $\Gamma$ itself is bi-interpretable with $\mathbb{Z}$ if and only if the power word $x^{|Z(\Gamma)|}$ has finite width in $\Gamma$. By Theorem \ref{thm:main} we get

\begin{corollary} \label{cor:for} Assume the GRH and that $\Gamma$ is as in Theorem \ref{thm:main}. Then $\Gamma$ is bi-interpretable with $\mathbb{Z}$. In particular, $\Gamma$ is quasi-finitely axiomatizable and every finitely generated subgroup of $\Gamma$ is definable. 
\end{corollary}

\begin{remark} Again, note the contrast to rank 1 groups: by \cite{Sel09}, if $\Delta$ is a torsion-free word hyperbolic group, then $\Delta$ and the free product $\Delta * \mathbb{Z}$ satisfy the same first-order sentences. Moreover, the only non-trivial, proper, and definable subgroups of $\Delta$ are cyclic.
\end{remark}

\subsubsection{Stability} A famous question of Ulam \cite[Chapter VI]{Ula60} asks what happens to algebraic notions if we relax their defining properties to only hold approximately, rather than on the nose. Specifying this question to the notion of a homomorphism, the question is whether every almost-homomorphism is close to an actual homomorphism. We use the following definition from \cite{BeCh}:

\begin{definition} For every $n$, let $d_n$ be the normalized Hamming distance on $S_n$:
\[
d_n(\sigma,\tau)=\frac1n \# \left\{ i\in [n] \mid \sigma(i) \neq \tau(i) \right\}.
\]
We say that a group $G$ is flexibly stable if for any $\epsilon >0$ there is $\delta >0$ such that if $f:G \rightarrow S_n$ is any function that satisfies
\begin{equation} \label{eq:almost.hom}
\left( \forall x,y \in G \right) \quad d_n \left( f(xy),f(x) f(y) \right)< \delta
\end{equation}
then there exists a natural number $n \leq N \leq n(1+\epsilon)$ and a homomorphism $F:G \rightarrow S_N$ such that
\begin{equation} \label{eq:close.to.hom}
\left( \forall x \in G \right) \quad d_N \left( f(x),F(x) \right)< \epsilon.
\end{equation}
\end{definition} 
In the definition, \eqref{eq:almost.hom} is the condition that $f$ is almost a homomorphism and \eqref{eq:close.to.hom} is the condition that $f$ is close to a homomorphism. The flexibility pertains to the ability to replace $n$ by $N$ in the definition.

In \cite{BeCh}, Becker and Chapman used Bounded Generation to prove that $\SL_n(\mathbb{Z})$ are flexibly stable. As a corollary of Theorem \ref{thm:main} we show:

\begin{theorem} \label{thm:stability} Assume the GRH and that $\Gamma$ is as in Theorem \ref{thm:main}. Then $\Gamma$ is flexibly stable.
\end{theorem} 

We prove Theorem \ref{thm:stability} in \S\ref{sec:stability}. Conjecture \ref{conj:local.to.global.one.group} will imply that higher-rank lattices are flexibly stable.

\begin{conjecture} \label{conj:stable} Let $\Gamma$ be as in Conjecture \ref{conj:local.to.global.one.group}. Then $\Gamma$ is flexibly stable.
\end{conjecture} 

\subsection{Sketch of the proof: Verbal width and CSP in non-standard models} \label{subsec:sketch}


A key idea of this paper is to relate width in an arithmetic group $\Gamma$ to variants of the Congruence Subgroup Problem for an ultrapower of $\Gamma$.

Since it plays an important role in this paper, we briefly survey the Congruence Subgroup Problem. Let $\mathbf{G}$ be a connected, simply connected, and simple algebraic group over a number field $K$ and let $S$ be a finite set of places of $K$ containing all archimedean ones. Assume that $\mathbf{G}(K_v)$ is not compact, for every finite $v\in S$. Denote the ring of $S$-integers in $K$ by $O_S$ and let $\Gamma = \mathbf{G}(O_S)$ be an arithmetic group. For a non-zero ideal $I$ of $O_S$, the principal congruence subgroup $\mathbf{G}(O_S)[I]$ is the kernel of the homomorphism $\mathbf{G}(O_S) \rightarrow \mathbf{G}(O_S/I)$. We say that $\Gamma$ has the Congruence Subgroup Property if there is a constant $C$ with the following property: if $\Delta \subseteq \Gamma$ is a finite-index subgroup, then there is an ideal $I$ such that $\big[ \mathbf{G}(O_S)[I] : \Delta \cap \mathbf{G}(O_S)[I] \big] \leq C$. The Congruence Subgroup Problem is the conjecture that $\Gamma$ has the Congruence Subgroup Property if the $S$-rank of $\mathbf{G}$ is at least two. It is known that most families of higher rank arithmetic groups have the Congruence Subgroup Property. The main case which is still open is when $\mathbf{G}$ is the norm one elements of a division algebra over $K$.

In the modern proofs, which stem from \cite{BLS64,BMS67}, the Congruence Subgroup Problem is reformulated as follows. Let $\widehat{\mathbf{G}(O_S)}$ be the profinite completion of $\mathbf{G}(O_S)$. By the assumptions on $\mathbf{G}$, the congruence completion of $\mathbf{G}(O_S)$ is $\mathbf{G} \left( \widehat{O_S} \right)$. Since $\widehat{\mathbf{G}(O_S)}$ is compact, the map $\widehat{\mathbf{G}(O_S)} \rightarrow \mathbf{G} \left( \widehat{O_S} \right)$ is onto. Consider the short exact sequence

\begin{equation} \label{eq:SES.classical}
1 \rightarrow C \rightarrow \widehat{\mathbf{G}(O_S)} \rightarrow \mathbf{G} \left( \widehat{O_S} \right) \rightarrow 1.
\end{equation}
The Congruence Subgroup Problem is equivalent to the claim that $C$ is finite.

There are two parts of the proof of the Congruence Subgroup Property for $\mathbf{G}(O_S)$. The first part is to show that the short exact sequence \eqref{eq:SES.classical} is a central extension. The second part is to show that \eqref{eq:SES.classical} has a continuous splitting. By Property (T), the abelization of $\mathbf{G}(O_S)$ is finite, so the finiteness of $C$ follows from the continuous splitting of \eqref{eq:SES.classical}.

The proof of Theorem \ref{thm:main} also starts by reformulating finite conjugacy width in terms of a short exact sequence, now related to an ultrapower of $\Gamma$. As a part of the proof, we show that finite conjugacy width in $\Spin_q(O_S)$ follows from a uniform bound on the width of the commutator word in all congruence subgroups of $\Spin_q(O_S)$. We describe now how to bound these widths. For simplicity, we restrict to the case $O_S=O:=\mathbb{Z} \left[ \frac12 \right]$, $q=f_a:=x_1^2+\cdots+x_{2n}^2$, and only consider the width of the commutator word in $\Spin_{f_a}(O)$. 

Denote $\Gamma=\Spin_{f_a}(O)$. Let $\Gamma^*$ be the ultrapower of $\Gamma$ taken along a non-principal ultrafilter and let $O^*$ be the ultrapower of $O$ along the same ultrafilter. It is easy to see that $\Gamma^*=\Spin_{f_a}(O^*)$. A standard model theoretic argument implies that the commutator word has a finite width in $\Spin_{f_a}(O)$ if and only if the abelianization of $\Gamma^*$ is finite. 

The congruence subgroups $\Gamma ^*[I]:=\Spin_{f_a}(O^*)[I]$ of $\Gamma^*=\Spin_{f_a}(O^*)$, where $I$ ranges over all non-zero ideals of $O^*$, form a basis of identity neighborhoods for a topology on $\Gamma ^*$. We refer to this topology as the congruence topology. Denote the completion of $\Gamma^*$ with the congruence topology by $\overline{\Gamma^*}$. We define a stronger topology on $\Gamma^*$, in which sets of the form $\Gamma ^*[I] \cap (\Gamma^*)'$, where $(\Gamma ^*)'$ is the derived subgroup, form a basis of identity neighborhoods. We call this topology the commgrunece topology. Denote the completion of $\Gamma ^*$ with respect to the commgruence topology by $\widehat{\Gamma^*}$. 

We show that the natural map $\varphi:\widehat{\Gamma ^*} \rightarrow \overline{\Gamma ^*}$ is surjective and that the short exact sequence
\begin{equation} \label{eq:SES.ultrapower}
1 \rightarrow C \rightarrow \widehat{\Gamma ^*} \rightarrow \overline{\Gamma ^*} \rightarrow 1
\end{equation}
is a split topological  central extension with a discrete center. In contrast to the standard case, neither $\widehat{\Gamma ^*}$ nor $\overline{\Gamma ^*}$ is compact, so the surjectivity in \eqref{eq:SES.ultrapower} is not obvious. Also, $\Gamma^*$ does not have Property (T), so the splitting of \eqref{eq:SES.ultrapower} does not imply the finiteness of $C$. However, once the splitting of \eqref{eq:SES.ultrapower} is established, it is easy to see that $C$ is isomorphic to the abelization of $\Gamma ^*$ and  $\widehat{\Gamma^*}\cong C \times \overline{\Gamma^*} $, so the centrality of \eqref{eq:SES.ultrapower} follows.

In general we need to show that $C$ is finite, but, in the case we discuss in the introduction, we will show that $C$ is trivial. Let $K^*$ be the ultrapower of $K$ over the same ultrafilter. It is easy to see that the congruence topology extends to a group topology on $\Spin_{f_a}(K^*)$. We show the same for the commgruence topology. The argument here is more involved than in the proof of the CSP and, in a sense, replaces the centrality argument in the classical case.

At this point, we get a short exact sequence
\begin{equation} \label{eq:SES.ultrapower.K}
1 \rightarrow C \rightarrow \widehat{\Spin_{f_a}(K^*)} \rightarrow \overline{\Spin_{f_a}(K^*)} \rightarrow 1
\end{equation}
with the same kernel $C$ of \eqref{eq:SES.ultrapower}. The advantage of \eqref{eq:SES.ultrapower.K} is that $\Spin_{f_a}(K^*)$ is perfect, which we prove in Corollary \ref{cor:spin*.perfect}. Thus, in order to show that $C$ is trivial, it is enough to prove that \eqref{eq:SES.ultrapower.K} admits a continuous splitting with a dense image.

At this point we know that \eqref{eq:SES.ultrapower.K} has a continuous splitting over the open subgroup $\overline{\Gamma ^*}$ and has an abstract (not necessarily continuous) splitting over the subgroup $\Spin_{f_a}(K^*)$ with a dense image. We call a topological central extension 
\begin{equation} \label{eq:metaplectic.def}
1 \rightarrow D \rightarrow E \rightarrow \overline{\Spin_{f_a}(K^*)} \rightarrow 1
\end{equation}
discrete metaplectic if $D$ is discrete, there is an abstract section of \eqref{eq:metaplectic.def} over $\Spin_{f_a}(K^*)$ whose image is dense, and there is a continuous splitting of \eqref{eq:metaplectic.def} over an open subgroup of $\overline{\Spin_{f_a}(K^*)}$. To finish the proof of Theorem \ref{thm:main}, we show that every discrete metaplectic extension has a continuous splitting with dense image. For the case we consider in this sketch, we can assume that the continuous splitting of \eqref{eq:metaplectic.def} is over all of $\overline{\Gamma^*}$.

Let $\overline{K^*}$ be the congruence completion of $K^*$ and let $f_s$ be the sum of $n$ hyperbolic planes. Since $\overline{\Spin_{f_a}(K^*)}=\Spin_{f_a} \left( \overline{K^*} \right)$ is isomorphic to the split form $\Spin_{f_s} \left( \overline{K^*} \right)$, we can rewrite \eqref{eq:metaplectic.def} as
\begin{equation} \label{eq:metaplectic.split}
1 \rightarrow D \rightarrow E \rightarrow \overline{\Spin_{f_s}(K^*)} \rightarrow 1
\end{equation}
and note that \eqref{eq:metaplectic.split} has a continuous splitting over $\overline{\Spin_{f_s}(O^*)}\cong \overline{\Spin_{f_a}(O^*)}$ and an abstract splitting with dense image over $\overline{\Spin_{f_a}(K^*)}$.

In order to show that \eqref{eq:metaplectic.split} splits, it is enough to show that it has an abstract splitting over $\Spin_{f_s}(K^*)$. Indeed, let $\alpha$ be such a splitting. Since every root element is a commutator, Bounded Generation for $\Spin_{f_s}(O)$ implies that $\Spin_{f_s}(O^*)$ is perfect. Thus, every splitting of \eqref{eq:metaplectic.split} over $\Spin_{f_s}(O^*)$ coincides with $\alpha$. Since there is a continuous splitting over $\overline{\Spin_{f_s}(O^*)}$, it agrees with $\alpha$ over $\Spin_{f_s}(O^*)=\Spin_{f_s}(K^*) \cap \overline{\Spin_{f_s}(O^*)}$. Thus, $\alpha$ is continuous and extends to a splitting of \eqref{eq:metaplectic.split}. In addition, we show that $\Spin_{f_a}(K^*)$ is perfect, so $\alpha$ coincides with the abstract section and hence has a dense image.

Let $h_1,h_2$ be two coroots of $\Spin_{f_s}$ corresponding to two roots of angle $120^\circ$. For every $a_1,a_2\in (K^*)^ \times$, the elements $h_1(a_1),h_2(a_2)$ commute. It follows from \cite[Theorem 8.2]{Mat69} and the perfectness of $\Spin_{f_s}(K^*)$ that the central extension \eqref{eq:metaplectic.split} has a splitting over $\Spin_{f_s}(K^*)$ if some lifts of $h_1(a_1),h_2(a_2)$ to $E$ commute. 

One way to ensure that some lifts of $h_1(a_1),h_2(a_2)$ commute is to find a pair $(g_1,g_2)$ of commuting elements in $\Spin_{f_a}(K^*)$ that is sufficiently close to $(h_1(a_1),h_2(a_2))$. Indeed, let $\mu:\Spin_{f_a}(K^*) \rightarrow E$ be a section. If $g_1,g_2\in \Spin_{f_a}(K^*)$ commute, then the lifts $\mu(g_1),\mu(g_2)$ also commute. If $(h_1(a_1),h_2(a_2))$ is close to $(g_1,g_2)$, then one can choose the lifts of $h_1(a_1),h_2(a_2)$ to be close to $\mu(g_1),\mu(g_2)$, so their commutator is close to the identity. Since $D$ is discrete, if the lifts of $h_1(a_1),h_2(a_2)$ are chosen close enough to $\mu(g_1),\mu(g_2)$, then $h_1(a_1)$ and $h_2(a_2)$ commute.   

In general, we will construct, for every $a_1,a_2$, a sequence $(g_{1,1},g_{2,1}),\ldots,(g_{1,m},g_{2,m})$ of commuting pairs in $\Spin_{f_a}(K^*)$ and lifts $\widetilde{h_1(a_1)},\widetilde{h_2(a_2)}$ such that
\begin{equation} \label{eq:successive.approximation}
\left[ \widetilde{h_1(a_1)},\widetilde{h_2(a_2)}\right] = \left[ \mu(g_{1,1}),\mu(g_{2,1}) \right] \cdots \left[ \mu(g_{1,m}),\mu(g_{2,m}) \right]=1.
\end{equation}
A main difficulty in finding the sequence of commuting pairs, not present in the classical case, is that, a-priori, $D$ may have a divisible factor. It is in this point of the proof that we use the Generalized Riemann Hypothesis. Specifically, we use a variant of Artin's Primitive Root Conjecture from \cite{Lens77} in the argument. We believe that results similar to \cite{Hea86} could give an unconditional result.

We end this section with a weaker version of Conjecture \ref{conj:local.to.global.one.group} which resembles Margulis' Normal Subgroup Theorem.

\begin{conjecture} Let $\Gamma$ be as in Conjecture \ref{conj:local.to.global.one.group}. Fix a non-principal ultrafilter $\mathcal{N}$ on the natural numbers. Let $O^*,K^*$ denote the $\mathcal{N}$-ultrapowers of $O$ and $K$, respectively. 

For every $g\in \mathbf{G}(O^*)$ such that the normal subgroup of $\mathbf{G}(K^*)$ generated by $g$ is $\mathbf{G}(K^*)$, the normal subgroup of $\mathbf{G}(O^*)$ generated by $g$ contains a finite index subgroup of a principal congruence subgroup.
\end{conjecture}

\subsection{Acknowledgements} We thank Danny Neftin, Eugene Plotkin, Pavel Gvozdevsky, Peter Sarnak, and V.N. Venkataramana for their help. N.A. was supported by NSF grant DMS-1902041, C.M. was supported by ISF grants 1226/19 and 872/23 and BSF grant 2020119.

\section{Notations} \label{sec:notations}

For the rest of the paper, we fix a number field $K$, a finite set $S$ of valuations of $K$ containing all archimedean valuations, and an anisotropic quadratic form $f_a$ on $K^n$. We will use the following notation:
\begin{itemize}
\item $S_{def}$ is the set of real valuations $v$ in $S$ such that $f_a$ is definite over $K_v$.
\item $O$ is the ring of $S$-integers of $K$.
\item $\mathcal{I}^K$ is the set of non-zero ideals of $O$.
\item $\mathcal{P} ^K$ is the set of non-zero prime ideals of $O$.
\item For $P\in \mathcal{P} ^K$, $\val_P:K^ \times \rightarrow \mathbb{Z}$ is the $P$-adic valuation. $K_P$ and $O_P$ are the $P$-adic completions of $K$ and $O$, respectively.
\item For $I\in \mathcal{I} ^K$, denote $\mathcal{P} ^K(I):=\left\{ P\in \mathcal{P} ^K \mid I \subseteq P \right\}$.
\item For $\mathcal{Q} \subseteq \mathcal{P} ^K$, denote $\mathcal{I} ^K(\mathcal{Q}):=\left\{ I\in \mathcal{I} ^K \mid \left( \forall P\in \mathcal{P} ^K \smallsetminus \mathcal{Q} \right) I+P=O \right\}$.
\item $\mathbf{O}:=\prod_{P\in \mathcal{P} ^K}O_P$ is the profinite completion of $O$ and $\mathbf{A} := K \cdot \mathbf{O}$ is the ring of finite $S$-adeles.
\end{itemize} 

In addition, we fix a non-principal ultrafilter $\mathcal{U}$ on $\mathbb{N}$ and make the following notations:
\begin{itemize}
\item For a set $X$, $X^*$ is the $\mathcal{U}$-ultrapower of $X$.
\item If $R$ is a ring and $V$ is a scheme over a ring $R$, we will identify $V(R^*)$ with the ultrapower of $V(R)$.
\item If $x_1,x_2,\ldots$ is a sequence of elements of $X$, $[x_n]\in X^*$ is the equivalence class of $(x_n)$. 
\item If $A_1,A_2,\ldots$ is a sequence of subsets of $X$, $[A_n] \subseteq X^*$ is the image of $\prod_n A_n$ in $X^*$.
\item $\leq$ is the linear order on $\mathbb{Z}^*$ which is the ultrapower of the standard order. We extend $\leq$ to $\mathbb{Z}^*\cup \left\{ \infty \right\}$ by declaring $m \leq \infty$, for every $m\in \mathbb{Z} ^*$.
\end{itemize} 
\newpage

\section{Congruence topologies on ultraproducts}

\subsection{Internal sets and ultraproduct topologies}

\begin{definition} Let $X$ be a set and $Y \subseteq X^*$ a subset. \begin{enumerate}
\item $Y$ is {\bf internal} if there is a sequence $A_1,A_2,\ldots \subseteq X$ of subsets of $X$ such that $Y=[A_n]$.
\item $Y$ is {\bf small internal} if there is a sequence $A_1,A_2,\ldots \subseteq X$ of finite subsets of $X$ such that $Y=[A_n]$.
\end{enumerate}  
\end{definition} 

\begin{remark} \label{rem:internal.sets} If $A_1,A_2,\ldots$ and $B_1,B_2,\ldots$ are sequences of subsets of $X$ then $[A_n]=[B_n]$ if and only if the set $\left\{ n \mid A_n=B_n \right\}$ belongs to $\mathcal{U}$. Thus, the map $(A_n) \mapsto [A_n]$ descends to an embedding of $\left( 2^{X} \right)^*$ (the ultrapower of the power set of $X$) into $2^{X^*}$ (the power set of $X^*$). The image of this map is the collection of internal sets.
\end{remark} 

\begin{lemma} \label{lem:internal.saturation} Let $X$ be a set. \begin{enumerate}
\item If $A,B \subseteq X^*$ are internal subsets such that $A \subseteq B$ and $B$ is small internal, then $A$ is small internal.
\item If $A,B \subseteq X^*$ are internal, then $A\cap B, A\cup B, X^* \smallsetminus A$ are internal.
\item The collection of internal subsets of $X^*$ is closed to finite unions, finite intersections, and complements. The collection of small internal subsets of $X^*$ is closed to finite unions and finite intersections.
\item (saturation) If $Y_1,Y_2,\ldots \subseteq X^*$ is a sequence of internal subsets such that the union $\cup_{n=1}^\infty Y_n$ is internal, then there is a natural number $N$ such that $\cup_{n=1}^\infty Y_n = \cup_{n \leq N} Y_n$.
\end{enumerate} 
\end{lemma} 

\begin{proof} \begin{enumerate}
\item By the definition of (finite) internal sets, there are subsets $A_n,B_n \subseteq X$ such that $A=[A_n]$, $B=[B_n]$, and all $B_n$ are finite. By assumption, the set $\left\{ n \mid A_n \subseteq B_n \right\}$ belongs to $\mathcal{U}$. Define
\[
A_n'=\begin{cases} A_n & \text{$A_n \subseteq B_n$} \\ \emptyset & \text{else} \end{cases} .
\]
For every $n$, the set $A_n'$ is finite. Since $\left\{ n \mid A_n=A_n' \right\}$ belongs to $\mathcal{U}$, $A=[A_n]=[A_n']$ is small internal.
\item If $A=[A_n]$ and $B=[B_n]$ are internal, then, arguing similarly to the previous claim, $A\cup B=[A_n \cup B_n]$, $A\cap B=[A_n \cap B_n]$, and $X^* \smallsetminus A=[X \smallsetminus A_n]$.
\item Follows from the previous claim.
\item Assume that $Y_1,Y_2,\ldots$ are internal subsets whose union $Z:=\cup Y_i$ is also internal.  Write $Y_i=[Y_{i,n}]$ and $Z=[Z_n]$. We can assume that $Y_{1,n}=\emptyset$, for every $n$. Replacing $Y_{i,n}$ by $Y_{i,n}\cap Z_n$ (which has no effect on $Y_i$), we can also assume that $Y_{i,n} \subseteq Z_n$, for every $i$ and $n$. Define $f: \mathbb{N} \rightarrow \mathbb{N}$ as follows: \begin{enumerate}
\item If $\cup_{j=1}^n Y_{j,n} \neq Z_n$, let $f(n)=n$.
\item Otherwise, let $f(n)$ be the maximal integer $i$ such that $\cup_{j=1}^i Y_{j,n} \neq Z_n$.
\end{enumerate} 
If $f$ is $\mathcal{U}$-essentially bounded by $N$, i.e., $f(n)<N$ for $\mathcal{U}$-almost every $n$, then $\cup_1^N Y_i=Z$. Otherwise, for every $n$ choose $x_n \in Z_n \smallsetminus \cup_{j=1}^{f(n)} Y_{j,n}$. For every $i$, the set $\left\{ n \mid x_n \in Y_{i,n} \right\} \subseteq \left\{ n \mid f(n) \leq i \right\}$ is $\mathcal{U}$-negligible, so $[x_n]\notin Y_i$. Thus $[x_n]\in Z \smallsetminus \cup_{i=1}^\infty Y_i$, a contradiction.
\end{enumerate} 
\end{proof} 

\begin{remark} The type of argument in the proof of the first claim of Lemma \ref{lem:internal.saturation} is called a coordinate-wise argument. In the rest of the paper, we will omit the details of coordinate-wise arguments that only use well-known facts.
\end{remark}

\begin{definition} Let $X$ be a set and $\tau$ be a topology on $X$. The collection $\tau ^*=\left\{ [O_n] \mid O_n \in \tau \right\}$ is a base for a topology on $X^*$ which we call the ultrapower topology.
\end{definition} 

\begin{lemma} \label{lem:ultraproduct.topology.base} If $\mathcal{B} \subseteq \tau$ is a base for $\tau$ then $\mathcal{B}^*$ is a base for the ultrapower topology.
\end{lemma} 

\begin{proof} Suppose that $U \subseteq X^*$ is open and $x\in U$. By definition there is a sequence $x_n\in X$ such that $x=[x_n]$ and there is a sequence $O_n \in \tau$ such that $x\in [O_n] \subseteq U$. By changing $O_n$ for a $\mathcal{U}$-negligible set of indexes $n$, we can assume that $x_n\in O_n$ for every $n$. Let $B_n\in \mathcal{B}$ be such that $x_n\in B_n \subseteq O_n$. Then $[B_n]\in \mathcal{B} ^*$ is a neighborhood of $x$ that is contained in $U$.
\end{proof} 

\begin{lemma} \label{lem:closure.in.ultraproduct.topology} If $X$ is a topological space and $A_n \subseteq X$, then $\overline{[A_n]}=\left[ \overline{A_n} \right]$.
\end{lemma} 

\begin{proof} A coordinate-wise argument.
\end{proof} 

\subsection{Topological groups---general properties}

We start by quoting several well known results.

\begin{lemma} \label{lemma:con_hom} Let $\Gamma$ and $\Lambda$ be Hausdorff topological groups. Let $A \subseteq \Gamma$ be a dense subgroup and let $U \subseteq \Gamma$ be an open subgroup. Let $\alpha:A \rightarrow \Lambda$ be an abstract homomorphisms and let $\beta:U \rightarrow \Lambda$ be a continuous homomorphism. Assume that $\alpha \restriction_{A\cap U}=\beta \restriction_{A\cap U}$. Then there is a unique continuous homomorphism $\gamma:\Gamma \rightarrow \Lambda$ extending both $\alpha$ and $\beta$.
\end{lemma}

\begin{corollary} \label{lemma:con_isom} Let $\Gamma$ and $\Lambda$ be Hausdorff topological groups. Let $A \le \Gamma$ and $B \le \Lambda$ be dense subgroups and let $\alpha:A \rightarrow B$ be an isomorphism of abstract groups. Let $U \le \Gamma$ and $V \le \Lambda$ be open subgroups and let $\beta:U \rightarrow V$ be an isomorphism of  topological groups. Assume that $\alpha\restriction_{A \cap U}=\beta\restriction_{A \cap U}$. Then there exists a unique continuous isomorphism of topological groups  $\gamma:\Gamma\rightarrow \Lambda$ such that
$\gamma\restriction_A=\alpha$ and $\gamma\restriction_U=\beta$. \end{corollary}

\begin{lemma}\label{lemma:top_base}
	Let $\Gamma$ be a group and let $\mathcal{B}$ be a collection of subgroups of $\Gamma$ with the following properties:
	\begin{enumerate}
		\item\label{item:top_base1} $\cap_{U \in \mathcal{B}}U=1$.
		\item\label{item:top_base2} For every $U_1,U_2 \in \mathcal{B}$ there exists $U_3 \in \mathcal{B}$ such that $U_3 \subseteq U_1\cap U_2$. 
		\item\label{item:top_base3} For every $g \in \Gamma$ and every $U_1 \in \mathcal{B}$ there exists $U_2 \in \mathcal{B}$ such that $gU_2g^{-1} \subseteq U_1$.
	\end{enumerate}
	There exists a unique topology on $\Gamma$, called the topology generated by $\mathcal{B}$, under which $G$ is a Hausdorff topological group and $\mathcal{B}$ is a base of open neighborhoods of the identity. 
\end{lemma}

\begin{definition} We say that a topology on $\Gamma$ is subnormally generated if there is an open subgroup $\Delta \subseteq \Gamma$ and a base of identity neighborhoods consisting of normal subgroups of $\Delta$.
\end{definition} 

Next, we discus completions in topological groups. A standard reference is \cite[\S3]{Bour98}. Recall that a Hausdorff topological group $\Gamma$ has two uniform structures. The left uniform structure is generated by the entourages
\[
\left\{ (x,y)\in \Gamma^2 \mid x ^{-1} y \in U \right\} 
\]
where $U$ ranges over a base of symmetric neighborhoods of 1. The right uniform structure is generated by the entourages
\[
\left\{ (x,y)\in \Gamma^2 \mid x y ^{-1} \in U \right\}.
\]
In general, these uniform structures are different. In such cases, it might happen that the map $x \mapsto x ^{-1}$ is uniformly continuous with respect to, say, the left uniform structure and the completion of $\Gamma$ is not a topological group. However, if the topology of $\Gamma$ is subnormally generated, the left and right uniform structures coincide and the completion of $\Gamma$ is a topological group (see \cite[III \S3 Theorem 1 and Exercise 7]{Bour98}). This will be the case for all groups in this paper.

\begin{notation} For a topological group $(\Gamma,\tau)$, we denote the completion of $\Gamma$ in the topology $\tau$ by $\widehat{\Gamma_\tau}$.
\end{notation}

The following follows directly from \cite[III \S7.3 Corollary 1]{Bour98}

\begin{lemma} \label{lem:Bourbaki.completion.inverse.limit} Let $(\Gamma,\tau)$ be a Hausdorff topological group and let $\mathcal{B}$ be a base of identity neighborhoods consisting of normal subgroups. Then the embedding $\Gamma \rightarrow \lim_{\Delta \in \mathcal{B}} \Gamma / \Delta$ extends to an isomorphism of topological groups between $\widehat{\Gamma_\tau}$ and $\lim_{\Delta\in\mathcal{B}} \Gamma / \Delta$.
\end{lemma} 

\subsection{Internal ideals and sets of ideals}

\begin{lemma} \label{lem:internal.ideals} Let $I \subseteq O^*$ be an ideal. The following are equivalent: \begin{enumerate}
\item $I$ is generated by two elements.
\item $I$ is finitely generated.
\item $I$ is an internal set of $O^*$.
\item There is a sequence $I_1,I_2,\ldots$ of ideals of $O$ such that $I=[I_n]$.
\end{enumerate} 
\end{lemma} 

\begin{proof} $ $ \begin{enumerate}
\item[$1. \implies 2.$] Clear.
\item[$2. \implies 3.$] If $I=a_1O^*+\cdots+a_kO^*$, write $a_i=[a_{i,n}]$ and, for each $n$, let $I_n$ be the ideal generated by $a_{1,n},\ldots,a_{k,n}$. Then $I=[I_n]$.

\item[$3. \implies 4.$] If $I=[A_n]$ then $[A_n+A_n]=I+I=I=[A_n]$ and $[A_n \cdot O]=I \cdot O^*=I=[A_n]$, so $A_n$ is an ideal for $\mathcal{U}$-almost every $n$. Define a sequence of ideals $I_n$ by $I_n=A_n$ if $A_n$ is an ideal and $I_n=O$ if $A_n$ is not an ideal. Then $I=[I_n]$.
\item[$4. \implies 1.$] Each $I_n$ is generated by two elements, say $I_n=a_nO+b_nO$. Then $I$ is generated by the elements $[a_n]$ and $[b_n]$.
\end{enumerate} 
\end{proof} 

Similarly,

\begin{lemma} Let $P \subseteq O^*$ be an internal ideal. Then $P$ is prime if and only if there are prime ideals $P_1,P_2,\ldots \subseteq O^*$ such that $P=[P_n]$.
\end{lemma} 

\begin{definition} Let $\mathcal{I} ^{K^*}$ be the collection of internal non-zero ideals of $O^*$ and let $\mathcal{P}^{K^*}$ be the collection of internal non-zero prime ideals of $O^*$.
\end{definition} 

\begin{remark} If $I\in \mathcal{I} ^{K^*}$, when we write $I=[I_n]$ we will assume that each $I_n$ is a non-zero ideal of $O$. Similarly, if $P\in \mathcal{P} ^{K^*}$, when we write $P=[P_n]$ we will assume that each $P_n$ is a non-zero prime ideal of $O$.
\end{remark} 

\begin{remark} By Remark \ref{rem:internal.sets} and Lemma \ref{lem:internal.ideals}, $\mathcal{I} ^{K^*}$ is identified with the ultraproduct $\left( \mathcal{I}^K \right)^*$ and $\mathcal{P} ^{K^*}$ is identified with the ultraproduct $\left( \mathcal{P} ^K \right) ^*$. In particular, we can talk about internal and small internal subsets of $\mathcal{I} ^{K^*}$ and $\mathcal{P} ^{K^*}$.
\end{remark} 

\begin{definition} For every $I\in \mathcal{I} ^{K^*}$ and every $\mathcal{Q} \subseteq \mathcal{P}^{K^*}$, denote
\[
\mathcal{P} ^{K^*}(I):= \left\{ P\in \mathcal{P} ^{K^*} \mid I \subseteq P \right\}
\]
and
\[
\mathcal{I} ^{K^*}(\mathcal{Q}):=\left\{ I \in \mathcal{I} ^{K^*} \mid \left( \forall P\in \mathcal{P} ^{K^*} \smallsetminus \mathcal{Q} \right) I+P=O^* \right\}.
\]
\end{definition} 

\begin{lemma} \label{lem:small.is.P(I)} A subset $\mathcal{Q} \subseteq \mathcal{P} ^{K^*}$ is small internal if and only if there is $I\in \mathcal{I} ^{K^*}$ such that $\mathcal{Q}=\mathcal{P} ^{K^*}(I)$.
\end{lemma} 

\begin{proof} In one direction, let $I=[I_n]\in \mathcal{I} ^{K^*}$. For every $n$, let $\mathcal{Q}_n:=\left\{ P\in \mathcal{P} ^K \mid I_n \subseteq P \right\}$. Then each $\mathcal{Q}_n$ is finite and $\mathcal{Q}=\mathcal{P} ^{K^*}(I)=[\mathcal{Q}_n]$ is small internal.

In the other direction, assume that $\mathcal{Q}=[\mathcal{Q}_n]$, where each $\mathcal{Q}_n \subseteq \mathcal{P} ^K$ is finite. For every $n$, let $I_n:=\prod_{P\in \mathcal{Q}_n} P$ (and $I_n=O$ if $\mathcal{Q}_n=\emptyset$). Then $\mathcal{Q} = \mathcal{P} ^{K^*}([I_n])$.
\end{proof} 

\subsection{Valuations}

\begin{lemma} \label{lem:val.ideals} $ $ \begin{enumerate}
\item If $P=[P_n]\in \mathcal{P} ^{K^*}$ and $I=[I_n]\in \mathcal{I} ^{K^*}$, then the element $[\val_{P_n}(I_n)]\in \mathbb{N} ^*$ depends only on $P$ and $I$. We denote it by $\val_P(I)$.
\item For every $I,J\in \mathcal{I} ^{K^*}$, $\val_P(I J)=\val_P(I)+\val_P(J)$.
\item For every $I\in \mathcal{I} ^{K^*}$ and $m\in \mathbb{N}^*$, the set $\left\{ P \in \mathcal{P} ^{K^*} \mid \val_P(I)>m \right\}$ is small internal.
\end{enumerate} 
\end{lemma} 

\begin{proof} A coordinate-wise argument.
\end{proof} 

Note that $\val_P(I)=0$ if and only if $I$ is not contained in $P$.

\begin{definition} Let $P=[P_n]\in \mathcal{P} ^{K^*}$ and let $a=[a_n]\in K^*$. By a coordinate-wise argument, the element $\left[ \val_{P_n}(a_n) \right]\in \left( \mathbb{Z} \cup \left\{ \infty \right\} \right)^*=\mathbb{Z}^* \cup \left\{ \infty \right\}$ depends only on $P$ and $a$. We denote it by $\val_P(a)$. 
\end{definition} 

\begin{lemma} $ $ \begin{enumerate}
\item For every $P\in \mathcal{P} ^{K^*}$, the function $\val_P$ is a non-archimedean valuation on $K^*$ with value group $\mathbb{Z} ^*$.
\item If $P\in \mathcal{P} ^{K^*}$ and $a\in O^*$, then $\val_P(a)=\val_P(a O^*)$.
\item For every $a\in K^* \smallsetminus 0$ there is $m\in \mathbb{N} ^*$ such that $-m < \val_Q(a) < m$, for every $Q\in \mathcal{P} ^{K^*}$.
\end{enumerate} 
\end{lemma} 

\begin{proof} A coordinate-wise argument.
\end{proof} 

\begin{proposition}[Strong Approximation Theorem  for $K^*$]\label{prop:strong_approximation_K} Let $k$ be a natural number. Let $m_1,\ldots,m_k\in \mathbb{N} ^*$, $I_1,\ldots,I_k\in \mathcal{I} ^{K^*}$, and $a_1,\ldots,a_k\in K^*$. Assume that $\mathcal{P}^{K^*}(I_i) $ and $\mathcal{P}^{K^*}(I_j)$ are disjoint, for every $1 \le i \ne j \le k$. There exists an element $a \in K^*$ such that:
\begin{enumerate}
	\item For every $1 \le i \le k$ and every $P\in \mathcal{P}^{K^*}(I_i)$,  $\val_P(a-a_i) \ge m_i$.
	\item For every $P \in \cap_{i=1}^k \mathcal{P}(I_i)^c$, $\val_P(a)\ge 0$.
\end{enumerate}
\end{proposition}
\begin{proof} A coordinate-wise argument.
\end{proof}

\subsection{The congruence topology on $\mathbf{A} ^*$}

The next lemma extends the map $\val_P$ from $K^*$ to $\mathbf{A} ^*$.

\begin{lemma}[The valuation map]\label{lemma:val_properties} Let $P \in \mathcal{P}^{K^*}$. 
\begin{enumerate}
\item\label{item:mac1} There exists a function $\val_{P}:\mathbf{A} ^*\rightarrow \mathbb{Z} ^* \cup \{\infty\}$ such that, if $P=[P_n]_n$ and $a=[a_n]_n \in \mathbf{A} ^*$, then $\val_P(a)=[\val_{P_n}(a_n)]_n \in \mathbb{Z} ^* \cup \{\infty\}$.
\item\label{item:mac2} For every $a,b \in \mathbf{A} ^*$, $\val_{P}(a+b) \ge \min(\val_P(a),\val_P(b))$ and $\val_{P}(ab) =\val_P(a)+\val_P(b)$. If $\val_P(a)<\val_{P}(b)$, then $\val_{P}(a+b) = \val_P(a)$.
\item\label{item:mac3} For every $m \in \mathbb{N} ^*$ and  $a \in \mathbf{A} ^*$, $\{P\in \mathcal{P}^{K^*} \mid \val_P(a) <-m\}$ is a small internal subset and  $\{P\in \mathcal{P}^{K^*} \mid \val_P (a) >m \}$ is an internal subset.
\item\label{item:mac4} If $a\in K^* \subseteq \mathbf{A} ^*$, then the two definitions of $\val_P(a)$ agree.
\item\label{item:mac5} If $a\in \mathbf{A} ^*$ then there is $m\in \mathbb{N} ^*$ such that $\val_P(a)>-m$, for every $P\in \mathcal{P} ^{K^*}$.
\end{enumerate}
\end{lemma}

\begin{proof} A coordinate-wise argument. 
\end{proof} 

\begin{lemma}\label{lemma:ideal_pow} For every $m=[m_n] \in \mathbb{N} ^*$, there is a well defined $m$-th power map ${}^m:\mathcal{I}^{K^*}\rightarrow \mathcal{I}^{K^*}$ defined by  $[I_n]_n^m=[I_n^{m_n}]_n$. If $m,m'\in \mathbb{N} ^*$, then
\begin{enumerate}
\item\label{item:ideal_pow1} $I^mI^{m'}=I^{m+m'}$.
\item\label{item:ideal_pow2}  $(I^m)^{m'}=I^{mm'}$.
\item For every $a \in O ^*$ and $P \in \mathcal{P}^{K^*}$, $\val_{P}(a)=\max\{m \in \mathbb{N}^* \mid a \in P^m\}$.
\end{enumerate}
\end{lemma}

\begin{proof}
A coordinate-wise argument. 
\end{proof}

\begin{definition} For $I \in \mathcal{I}^{K^*}$, let $\mathbf{I}:=\{x \in \mathbf{A}^* \mid  (\forall P \in \mathcal{P}^{K^*})\ \val_P(x) \ge \val_P(I)\}$.
\end{definition}

The ring $O^*$ is not Noetherian and ideals in $\mathcal{I}^{K^*}$ are not necessarily products of $\mathbb{N} ^*$-powers of ideals in $\mathcal{P}^{K^*}$. Nevertheless, the following lemma shows that ideals are determined by their valuations.  

\begin{lemma}\label{lemma:equal_ideals_2} Let $I \in \mathcal{I}^{K^*}$.
\begin{enumerate}
	\item  $I=\{x \in O^* \mid (\forall P \in \mathcal{P}^{K^*})\ \val_P(x) \ge \val_P(I) \}$. In particular:
	\begin{itemize} 
		\item $I=\mathbf{I} \cap K^*$.
		\item For every $I_1,I_2 \in \mathcal{I}^{K^*}$, $I_1 \subseteq I_2$ if and only if, for every $P \in \mathcal{P}^{K^*}$, $\val_P(I_1) \leq \val_P(I_2)$. 
	\end{itemize}
	\item $\mathbf{I}$ is an ideal of $\mathbf{O}^*$.
\end{enumerate}
\end{lemma}
\begin{proof} A coordinate-wise argument.
\end{proof}

\begin{lemma}\label{lemmma:congruence_topolgy_ring}  The set $\mathcal{C}:=\{ \mathbf{I} \mid I \in \mathcal{I}^{K^*}\}$ is a base of open neighborhoods of $0$ for a Hausdorff topology on $\mathbf{A} ^*$ under which $\mathbf{A} ^*$ is a  topological ring. Under this topology, for every $I \in \mathcal{I}^{K^*}$,  $\mathbf{I}$ is closed.
\end{lemma}
\begin{proof} $\mathcal{C}$ is closed under intersections and every $\mathbf{I} \in \mathcal{C}$ is an additive subgroup of $\mathbf{A} ^*$. Thus, in order to show the first claim, it is enough to show that if $a \in \mathbf{A} ^*$ and $I \in \mathcal{I}^{K^*}$ then there exists $J \in \mathcal{I}^{K^*}$ such that $a \mathbf{J} \subseteq \mathbf{I}$. By Lemmas \ref{lem:small.is.P(I)} and \ref{lemma:val_properties}, there exists an ideal $J'\in \mathcal{I}^{K^*}$ such that $\mathcal{P}^{K^*}(J')=\{P \in \mathcal{P}^{K^*} \mid \val_P(a)<0\}$. By Lemma \ref{lemma:val_properties}, there exists $m\in \mathbb{N}^*$ such that $\val_P(a)>-m$, for every $P\in \mathcal{P}^{K^*}$. Denote $J=(J')^mI$. For every $x\in \mathbf{J}$, $\val_P(ax) \geq \val_P(a)+m\val_P(J')+\val_P(I) \geq \val_P(I)$, so $a \mathbf{J} \subseteq \mathbf{I}$.

Since $\mathbf{I}$ is an ideal of $\mathbf{A} ^*$, $\mathbf{I}=\mathbf{A} ^* \smallsetminus \bigcup_{a \in \mathbf{A} ^* \setminus \mathbf{I}}(a+\mathbf{I})$ is closed. 
\end{proof}

\begin{definition}\label{def:congruence topology} We call the topology on $\mathbf{A} ^*$ defined in Lemma \ref{lemmma:congruence_topolgy_ring} and the induced subset topologies on $\mathbf{O} ^*$, $O^*$, and $K^*$ the {\bf congruence topologies}. 
\end{definition}

\begin{remark} By Lemma \ref{lem:internal.ideals}, every non-zero ideal of $O^*$ contains an ideal in $\mathcal{I}^{K^*}$. Thus, the congruence topology on $O^*$ and $K^*$ is equal to the ultrapower of the congruence topologies on $O$ and $K$, respectively.
\end{remark}  
  
\begin{lemma} Let $I \in \mathcal{I}^{K^*}$. Then:
\begin{enumerate}
	\item $O^*$ is dense in $\mathbf{O}^*$.
	\item $K^*$ is dense in  $\mathbf{A} ^*$.
	\item $I$ is  dense in $\mathbf{I}$
\end{enumerate}
\end{lemma}
\begin{proof}
The claim follows from the Strong Approximation Theorem for $K^*$ (Proposition \ref{prop:strong_approximation_K}). \end{proof}

By \cite[III. \S5 Proposition 6]{Bour98}, every Hausdorff topological ring has a completion as a topological ring.

\begin{definition} Denote the completion of $\mathbf{A}^*$ with respect to the congruence topology by $\mathbb{K}$ and denote the closure of $\mathbf{O} ^*$ in $\mathbb{K}$ by $\mathbb{O}$. For $I\in \mathcal{I}^{K^*}$, denote the closure of $I$ in $\mathbb{K}$ by $\mathbb{I}$.
\end{definition} 

\begin{remark} \begin{enumerate}
\item Since $K^*$ is dense in $\mathbf{A} ^*$,  $\mathbb{K}$ is also the completion of $K^*$ with respect to the congruence topology. 
\item Since $O^*$ is dense in $\mathbf{O} ^*$ , $\mathbb{O}$ is also the closure of $O^*$.  Since a closed subspace of a complete space is complete, $\mathbb{O}$ is the completion of $\mathbf{O} ^*$ and of $O^*$. 
\item $\{\mathbb{I}\mid I \in \mathcal{I}^{K^*}\}$ is a base of open neighborhoods of $0$ for $\mathbb{K}$.
\end{enumerate} 
\end{remark}

\begin{remark} The set $\mathbb{Z} ^* \cup \left\{ \infty \right\}$ has a complete uniform structure given by the base $\left\{ V_n \right\}_{n\in \mathbb{Z} ^*}$, where $V_n=\left\{ (x,y) \mid \text{either $x,y>n$ or $x=y$} \right\}$. For every $P\in \mathcal{P}^{K^*}$, the map $\val_P:\mathbf{A} ^* \rightarrow \mathbb{Z} ^* \cup \{\infty\}$ extends continuously to a map $\val_P:\mathbb{K} \rightarrow \mathbb{Z} ^* \cup \{\infty\}$ with similar properties. In particular, for every $I \in \mathcal{I}^{K^*}$, $\mathbb{I}=\{x \in \mathbb{K} \mid (\forall P \in \mathcal{P}^{K^*})\ \val_P(x) \ge \val_P(I)\}$.
\end{remark}

\begin{lemma}\label{lemma:isom_to_inverse} The natural embedding of $i:O^* \rightarrow \underset{\longleftarrow}\lim (O^*/I)$ has a unique  extension to an isomorphism of the topological rings $\mathbb{O} \cong \underset{\longleftarrow}\lim (O^*/I)$, where the inverse limit is taken over $I\in \mathcal{I}^{K^*}$ and every $O^*/I$ is given the discrete topology. 
\end{lemma}

\begin{proof} This is a special case of Lemma \ref{lem:Bourbaki.completion.inverse.limit}.
\end{proof}

\subsection{The congruence completion of $G(K^*)$}

Let $G$ be an algebraic group and fix an embedding $G\hookrightarrow \Mat_D$ defined over $O$. If $R$ is a Hausdorff topological ring containing $O$, the topology on $R$ induces a topology on $G(R) \subseteq \Mat_D(R)$ in which $G(R)$ is a Hausdorff topological group. Unless stated otherwise, we equip $G(R)$ with this topology. This applies, in particular, to $G(K^*)$, $G(\mathbf{A} ^*)$, and $G(\mathbb{K})$. We refer to these topologies as the congruence topologies.

\begin{definition} Let $I\in \mathcal{I}^{K^*}$. Denote
\begin{enumerate}
\item $\pi_I:G(O^*) \rightarrow G(O^*/I)$, $\pi_{\bf I}:G(\mathbf{O} ^*) \rightarrow G(\mathbf{O} ^*/{\bf I})$ and $\pi_{\mathbb I}:G(\mathbb{K}) \rightarrow G(\mathbb{K}/{\mathbb I})$ the modulo-$I$, modulo-${\bf I}$, and modulo-$\mathbb{I}$ reduction maps, respectively. 
\item Let $G(O^*)[I]:=\ker \pi_I$, $G(\mathbf{O} ^*)[I]:=\ker \pi_{\bf I}$, and $G(\mathbb{K})[I]:=\ker \pi_{\mathbb I}$.
\end{enumerate}
A principal congruence subgroup is a subgroup of $G(K^*)$ of the form $G(O^*)[I]$. A congruence subgroup is a subgroup of $G(K^*)$ that contains a principal congruence subgroup.
\end{definition}

\begin{definition} The collections of principal congruence subgroups of $G(K)$, $G(K^*)$, $G(\mathbf{A}^*)$, and $G(\mathbb{K})$ are denoted by $\mathcal{C}_{G,O}$, $\mathcal{C}_{G,O^*}$, $\mathcal{C}_{G,\mathbf{O} ^*}$, and $\mathcal{C}_{G,\mathbb{O}}$, respectively.
\end{definition} 

By Remark \ref{rem:internal.sets}, $\mathcal{C}_{G,O^*}$ is identified with $\mathcal{C}_{G,O}^*$.

The following lemma is clear.

\begin{lemma}
The collections $\mathcal{C}_{G,O^*}$, $\mathcal{C}_{G,\mathbf{O} ^*}$, and $\mathcal{C}_{G,\mathbb{O}}$ are bases for open neighborhoods of the identity for the congruence topologies of $G(K^*)$, $G(\mathbf{A} ^*)$, and $G(\mathbb{K})$, respectively. 
\end{lemma}

The standard Strong Approximation Theorem states if $G$ is semisimple and 1-connected then $G(K)$ is dense in $G(\mathbf{A})$ or, more precisely, that the embedding of $G(K)$ into $G(\mathbf{A})$ uniquely extends to a topological isomorphism between the congruence completion of $G(K)$ and $G(\mathbf{A})$. Our next goal is to prove the analogous statement: the embedding of $G(K^*)$ in $G(\mathbb{K})$ uniquely extends  to a topological isomorphism from the congruence completion of $G(K^*)$ to $G(\mathbb{K})$. 

\begin{proposition}[Strong approximation in $X(K^*)$]\label{prop:dense} Let $X$ be an affine scheme over $K$. If $X(K)$ is dense in $X(\mathbf{A})$ then $X(K^*)$ is dense in $X(\mathbb{K})$.  
\end{proposition}
\begin{proof} Suppose that $X$ is the common zero locus of the polynomials $p_1,\ldots,p_k\in K[x_1,\ldots,x_m]$ and denote the polynomial map $x \mapsto (p_1(x),\ldots,p_k(x))$ by $f$. A coordinate-wise argument shows that $X(K^*)$ is dense in $X(\mathbf{A} ^*)$, so it is enough to prove that $X(\mathbf{A} ^*)$ is dense in $X(\mathbb{K})$. 

Let $a \in X(\mathbb{K})$ and let $I=[I_n]_n \in \mathcal{I}^{K^*}$. We have to show that there exists $c \in X(\mathbf{A}^*) \cap (a+\mathbb{I}^m)$. Identify $(K^*)^m$ with the ultrapower $(K^m)^*$ and choose a sequence $b_n\in K^m$ such that $[b_n]_n \in a + \mathbb{I}^m$. Denote
\[
Y:=\left\{ n \in \mathbb{N} \mid (\forall J \in \mathcal{I} ^K)\quad f \left( b_n+I_n^m \right) \cap J^k \neq \emptyset \right\}.
\]

We claim that $Y$ is $\mathcal{U}$-large. Otherwise, for $\mathcal{U}$-almost every $n$, there exists $J_n \in \mathcal{I} ^K$ such that $f \left( b_n+I_n^m \right) \cap J_n^k = \emptyset$. Letting $J=[J_n]_n$, it follows that $f([b_n]_n+I^m) \cap J^k=\emptyset$. By Lemma \ref{lemma:equal_ideals_2}, $f([b_n]_n+I^m) \cap \mathbb{J}^k=\emptyset$. Since $\mathbb{J}$ is open, we get that $f([b_n]_n+\mathbb{I} ^m)\cap \mathbb{J} ^k=\emptyset$, a contradiction to $a\in [b_n]_n+\mathbb{I} ^m$.

Choose a decreasing sequence of ideals $J_r\in \mathcal{I} ^K$ such that $\bigcap J_r \mathbf{O}=0$. For every $n\in Y$ and $r\in \mathbb{N}$ there is an element $c_{n,r}\in b_n+I_n^m$ such that $p_1(c_{n,r}),\ldots,p_k(c_{n,r})\in J_r$. Since $b_n+I_n \mathbf{O}^m$ is compact, there is an accumulation point $c_n\in b_n+I_n \mathbf{O}^m$ of $c_{n,1},c_{n,2},\ldots$. By the assumption on $J_r$, $c_n\in X(\mathbf{A})$. Then $c:=[c_n]\in X(\mathbf{A} ^*)$ and $c\in [b_n]_n+\mathbb{I}^m=a+\mathbb{I}^m$.
 \end{proof}

\begin{corollary}[Strong Approximation Theorem for $G(K^*)$] Suppose that $G$ is a semisimple and 1-connected group defined over $K$. The natural embedding of  $G(K^*)$ in $G(\mathbb{K})$ has a unique extension to a topological isomorphism from the congruence completion of $G(K^*)$ to $G(\mathbb{K})$. Similarly,  The natural embedding of  $G(O^*)$ in $G(\mathbb{O})$ has a unique extension to a topological isomorphism from the congruence completion of $G(O^*)$ to $G(\mathbb{O})$.\end{corollary}
\begin{proof} Denote the congruence topology on $G(K^*)$ by $\kappa$. Since $G(O^*)$ is open in $G(K^*)$, Lemma \ref{lem:Bourbaki.completion.inverse.limit} implies that $G(K^*)$ has a completion with respect to $\kappa$. Let $i:G(K^*) \rightarrow G(\mathbb{K})$ be the embedding. $i$ induces an isomorphism between the topological groups $G(K^*)$ and $i(G(K^*))$, so by the universal property of completions, $i$ uniquely extends to an isomorphism of the topological groups $\widehat{G(K^*)}_{\kappa}$ and $G(\mathbb{K})$.  Since $G(O^*)$ is dense in $G(\mathbb{O})$ and $G(\mathbb{O})$ is closed and open, the second statement follows from the first.     
\end{proof}

\begin{remark}
From now on we view $G(\mathbb{O})$ and $G(\mathbb{K})$ as the congruence completions of $G(O^*)$ and $G(K^*)$, respectively. 
\end{remark}

\begin{lemma}\label{lemma:isom_to_inverse_gp} \label{cor:cong_com} The natural embedding of $G(O^*)$ in $\underset{\longleftarrow}\lim\ G(O^*/I)$ uniquely  extends to an isomorphism  of the topological groups  $G(\mathbb{O})$ and $\underset{\longleftarrow}\lim\ G(O^*/I)$ where the inverse limit is taken over $ \mathcal{I}^{J^*}$ and every $G(O^*/I)$ is given the discrete topology. 
\end{lemma}
\begin{proof} This is a special case of Lemma \ref{lem:Bourbaki.completion.inverse.limit}.
\end{proof}

\subsection{The $\mathcal{Q}$-projection maps}

We identify $G(\mathbf{A})$ with $\prod'_{P \in \mathcal{P}^{K}}G(K_P)$ in the natural way. More precisely, elements of $G(\mathbf{A})$ are identified with sequences $(g_P)_{P \in \mathcal{P}^{K}} \in \prod_{P\in \mathcal{P}^K} G(K_P)$ such that $g_P \in G(O_P)$ for all but finitely many $P$.

\begin{definition}
\begin{enumerate}
\item Let $\mathcal{Q} \subseteq \mathcal{P}^{K}$. For every $g=(g_P)_P \in G(\mathbf{A})$, define $g_{\mathcal{Q}}=(g_{\mathcal{Q},P})_P \in G(\mathbf{A})$ by:
\[
g_{\mathcal{Q},P}:=\begin{cases} g_P & P \in \mathcal{Q} \\ 1 & P \notin \mathcal{Q} \end{cases}.
\]
\item For every internal subset $\mathcal{Q}=[\mathcal{Q}_n]_n \subseteq \mathcal{P}^{K^*}$ and every $g=[g_n]_n \in G(\mathbf{A}^*)$, denote $g_\mathcal{Q}=[(g_n)_{\mathcal{Q}_n}]_n$. This definition does not depend on the choice of representatives.
 \item If $\mathcal{Q} \subseteq \mathcal{P}^{K^*}$ is internal, the $\mathcal{Q}$-projection map is the map $\pi_{\mathcal{Q}}:G(\mathbf{A}^*)\rightarrow G(\mathbf{A}^*)$ defined by $\pi_\mathcal{Q}(g)=g_\mathcal{Q}$. We denote the image of $\pi_{\mathcal{Q}}$ by $G(\mathbf{A}^*)_\mathcal{Q}$.
\end{enumerate}	
\end{definition}

\begin{lemma}
Let $\mathcal{Q} \subseteq \mathcal{P}^{K^*}$ be an internal subset.
\begin{enumerate}
\item $\pi_\mathcal{Q}$ is a continuous homomorphism.
\item $G(\mathbf{A}^*)_{\mathcal{Q}}=\ker(\pi_{\mathcal{Q}^c})$, where $\mathcal{Q}^c:=\mathcal{P}^{K^*}\smallsetminus \mathcal{Q}$. 
\item The map $(\pi_{\mathcal{Q}} \times \pi_{\mathcal{Q} ^c}):G(\mathbf{A}^*) \rightarrow G(\mathbf{A}^*)_\mathcal{Q} \times G(\mathbf{A}^*)_{\mathcal{Q}^c}$ is an isomorphism of topological groups.
\end{enumerate}
\end{lemma}
\begin{proof} A coordinate-wise argument. 
\end{proof}

\section{Width and non-standard metaplectic extensions}

\subsection{Discrete metaplectic extensions}

Let $\Gamma$ be a topological group and $X \subseteq \Gamma$ be a subgroup. In this subsection we consider (in a special case) the extension of the topology obtained by adding $X$ as an open set. We first find a sufficient condition on this topology to be a group topology and then study the relation between the completions of $\Gamma$ with respect to these two topologies.

\begin{definition} Let $(\Gamma,\tau)$ be a topological group and let $X \subseteq \Gamma$ be a subgroup. Define the topology $\tau(X)$ to be the topology generated by left translates of sets of the form $X \cap \Lambda$, where $\Lambda$ ranges over the $\tau$-open subgroups of $\Gamma$.
\end{definition} 

\begin{definition} \label{def:good.subgroup.topology} Let $(\Gamma,\tau)$ be a topological group and let $X \subseteq \Gamma$ be a subgroup. We say that $X$ is {\bf good} if there is a $\tau$-open subgroup $\Delta_X$ such that \begin{enumerate}
\item $\Delta_X \cap X$ is a dense normal subgroup of $\Delta_X$.
\item For every $g\in \Gamma$ there is a $\tau$-open subgroup $\Delta'$ with $g[\Delta_X,\Delta']g ^{-1} \subseteq X$.
\end{enumerate}
\end{definition} 

\begin{theorem} \label{thm:good.implies.topological.group} Assume that $X$ is good and that $\Gamma$ is normally generated by any of its $\tau$-open subgroups. Then $\tau(X)$ is a group topology on $\Gamma$.
\end{theorem} 

\begin{proof} We verify the conditions of Lemma \ref{lemma:top_base}. The first two conditions follow from the corresponding claims about $\tau$. The third condition is that, for every $g\in \Gamma$ and every $\tau$-open subgroup $\Lambda$, there is a $\tau$-open subgroup $\Lambda'$ such that 
\begin{equation} \label{eq:good.implies.topological.group} 
g \left( \Lambda' \cap X \right) g ^{-1} \subseteq \Lambda \cap X.
\end{equation}
Let $\Delta_X$ be as in Definition \ref{def:good.subgroup.topology}. \\

\noindent {\bf Case 1:} $g$ can be conjugated to $\Delta_X$, i.e. $g=f h f ^{-1}$ with $f\in \Gamma$ and $h\in \Delta_X$. By Definition \ref{def:good.subgroup.topology}, there is a $\tau$-open subgroup $\Delta'$ such that $f[\Delta_X,\Delta']f ^{-1} \subseteq X$. We claim that \eqref{eq:good.implies.topological.group} holds if $\Lambda'$ is any $\tau$-open subgroup for which $g \Lambda' g ^{-1} \subseteq \Lambda$ and $f ^{-1} \Lambda' f \subseteq \Delta'$. Indeed, if $x\in \Lambda' \cap X$ then $g x g ^{-1} \in \Lambda$ and
\[
g x g ^{-1} = f h f ^{-1} x f h ^{-1} f ^{-1} = f\underbrace{[h,f ^{-1} x f]}_{\in [\Delta_X,\Delta']} f ^{-1} x \in X \cdot x = X.
\]

\noindent {\bf Case 2:} $g$ is general. By assumption, $g$ can be written as $g=g_1 \cdots g_n$, where $g_i$ can be conjugated into $\Delta_X$. By Case 1, there is a sequence $\Lambda_0,\ldots,\Lambda_n$ of $\tau$-open subgroups such that $\Lambda_0=\Lambda$ and such that $g_k \left( \Lambda_{k} \cap X \right) g_k ^{-1} \subseteq \Lambda_{k-1} \cap X$, for $k=1,\ldots,n$. Then \eqref{eq:good.implies.topological.group} holds for $\Lambda'=\Lambda_n$
\end{proof} 

\begin{notation} Let $(\Gamma,\tau)$ be a topological group and let $X \subseteq \Gamma$ be a good subgroup. By the universal property of completions, the identity map extends to a continuous homomorphism $\widehat{\Gamma_{\tau(X)}} \rightarrow \widehat{\Gamma_\tau}$. Denote this homomorphism by $\rho_{\tau,X}$.
\end{notation} 

\begin{lemma} \label{lem:ses.compact} Let $(\Gamma,\tau)$ be a topological group that has a base of identity neighborhoods consisting of normal subgroups and let $X \triangleleft \Gamma$ be a dense normal subgroup such that $[\Gamma , \Gamma] \subseteq X$. Then $X$ is good, $\ker \rho_{\tau,X}$ is a discrete subgroup of $\widehat{\Gamma_{\tau,X}}$ isomorphic to $\frac{\Gamma}{X}$, and the short exact sequence
\[
1 \rightarrow \ker \rho_{\tau,X} \rightarrow \widehat{\Gamma_{\tau(X)}} \stackrel{\rho_{\tau,X}}{\rightarrow} \widehat{\Gamma_\tau} \rightarrow 1
\]
is a topological central extension that has a continuous splitting.
\end{lemma} 

\begin{proof} Clearly, $X$ is good. If $\Lambda \subseteq \Gamma$ is a $\tau$-open and normal subgroup then $X \cdot \Lambda = \Gamma$. It follows that the map $\frac{\Gamma}{\Lambda \cap X} \rightarrow \frac{\Gamma}{\Lambda} \times \frac{\Gamma}{X}$ is onto and, thus, an isomorphism of discrete groups. If $\Lambda' \subseteq \Lambda$ is a smaller $\tau$-open normal subgroup then the diagram
\[
\begin{tikzcd}
\frac{\Gamma}{X \cap \Lambda'} \arrow[d] \arrow[r] & \frac{\Gamma}{ X } \times \frac{\Gamma}{\Lambda'} \arrow[d] \\ \frac{\Gamma}{ X \cap \Lambda} \arrow[r] & \frac{\Gamma}{ X } \times \frac{\Gamma}{\Lambda}
\end{tikzcd}
\]
commutes. By Lemma \ref{lem:Bourbaki.completion.inverse.limit}, we get an isomorphism of topological groups $i:\widehat{\Gamma_{\tau(X)}} \rightarrow \frac{\Gamma}{X} \times \widehat{\Gamma_\tau}$ where $\frac{\Gamma}{X}$ is given the discrete topology. Moreover, the composition of $i$ with the projection to the second coordinate is $\rho_{\tau,X}$. Since $[\Gamma , \Gamma] \subseteq X$, the group $\frac{\Gamma}{X}$ is abelian and the lemma follows.
\end{proof} 

\begin{lemma} \label{lem:ses.general} Let $(\Gamma,\tau)$ be a topological group and let $X \subseteq \Gamma$ be a subgroup. Assume that \begin{enumerate}
\item $X$ is good.
\item $\tau$ is subnormally generated.
\item $\Gamma$ is normally generated by every $\tau$-open subgroup of $\Gamma$.
\end{enumerate}
Then the short exact sequence
\begin{equation} \label{eq:ses.general}
1 \rightarrow \ker \rho_{\tau,X} \rightarrow \widehat{\Gamma_{\tau(X)}} \rightarrow \widehat{\Gamma_\tau} \rightarrow 1
\end{equation}
is a topological central extension and the following hold:
\begin{enumerate}
\item $\ker \rho_{\tau,X}$ is discrete.
\item $\rho$ has a section over $\Gamma$.
\item $\rho$ has a continuous section over a $\tau$-open subgroup of $\widehat{\Gamma_\tau}$.
\item There is a $\tau$-open subgroup $\Lambda$ such that $\ker \rho_{\tau,X}\cong \frac{\Lambda}{\Lambda \cap X}$.
\end{enumerate} 
\end{lemma} 

\begin{proof} Let $\Delta_X \subseteq \Gamma$ be as in Definition \ref{def:good.subgroup.topology}. By shrinking $\Delta_X$, we can assume that $\tau$ has a basis of normal subgroups of $\Delta_X$ and that $[\Delta_X,\Delta_X] \subseteq X$. Let $\Delta=\Delta_X$ and $Y=\Delta_X \cap X$. Applying Lemma \ref{lem:ses.compact} to $\Delta$ and $Y$, we get that $\ker \rho_{\tau,Y}$ is discrete in $\widehat{\Delta_{\tau,Y}}$ and the short exact sequence 
\[
1 \rightarrow \ker \rho_{\tau,Y} \rightarrow \widehat{\Delta_{\tau,Y}} \rightarrow \widehat{\Delta_\tau} \rightarrow 1
\]
is a split topological central extension.

The inclusion $\Delta_X \hookrightarrow \Gamma$ extends to inclusions $\widehat{\Delta_\tau} \hookrightarrow \widehat{\Gamma_\tau}$ and $\widehat{\Delta_{\tau,Y}} \hookrightarrow \widehat{\Gamma_{\tau,X}}$. Since $\Delta$ is open in both $\tau$ and $\tau(X)$, the subgroups $\widehat{\Delta_\tau}$ and $\widehat{\Delta_{\tau,Y}}$ are open in $\widehat{\Gamma_\tau}$ and $\widehat{\Gamma_{\tau,X}}$, respectively. It follows that $\rho_{\tau,X} ^{-1} \left( \widehat{\Delta_\tau} \right)=\widehat{\Delta_{\tau,Y}}$ and $\rho_{\tau,X}\restriction_{\widehat{\Delta_{\tau,Y}}}=\rho_{\tau,Y}$. In particular, $\ker \rho_{\tau,X}=\ker \rho_{\tau,Y}$ is discrete and the short exact sequence \eqref{eq:ses.general} is a topological extension. 

It remains to prove that $\ker \rho_{\tau,X}$ is central. The centralizer $C_{\Gamma} \left( \ker \rho_{\tau,X} \right)$ of $\ker \rho_{\tau,X}$ in $\Gamma$ is normal and contains $\Delta$. By the assumptions, $C_{\Gamma} \left( \ker \rho_{\tau,X} \right)=\Gamma$, so the centralizer $C_{\widehat{\Gamma_{\tau,X}}} \left( \ker \rho_{\tau,X} \right)$ is dense. Since centralizers are closed, $C_{\widehat{\Gamma_{\tau,X}}} \left( \ker \rho_{\tau,X} \right)=\widehat{\Gamma_{\tau,X}}$ and $\ker \rho_{\tau,X}$ is central.
\end{proof} 

We summarize the properties above with the following

\begin{definition} Let $\Gamma$ be a group and $\tau$ be a group topology on $\Gamma$ that admits a completion. A discrete $(\Gamma,\tau)$-metaplectic extension is a topological central extension
\[
1 \rightarrow C \rightarrow E \stackrel{\rho}{\rightarrow} \widehat{\Gamma}_\tau \rightarrow 1
\]
such that \begin{enumerate}
\item $C$ is discrete.
\item There is a (possibly discontinuous) splitting of $\rho$ over $\Gamma$ such that $\rho(\Gamma)$ is dense in $E$.
\item There is a continuous splitting of $\rho$ over an open subgroup of $\widehat{\Gamma}_\tau$.
\end{enumerate} 
When $\tau$ is clear from the context, we omit it and call the extension $\Gamma$-metaplectic.
\end{definition} 

\subsection{Sifters and width}

\begin{definition} \label{def:sifter} Let $\Gamma$ be a group and let $\mathcal{B}$ be a base of identity neighborhoods for a group topology on $\Gamma$. A $\mathcal{B}$-sifter on $\Gamma$ is an assignment $\Delta \mapsto X_\Delta$, taking a subgroup $\Delta \in \mathcal{B}$ to a symmetric subset $X_\Delta \subseteq \Delta$ that contains $1$, such that the following holds: for every $\Delta_1,\Delta_2 \in \mathcal{B}$ and every $g\in \Gamma$,
\[
g \Delta_1 g ^{-1} \subseteq \Delta_2 \implies g X_{\Delta_1} g ^{-1} \subseteq X_{\Delta_2}.
\]
\end{definition} 

Equivalently, a $\mathcal{B}$-sifter is an assignment $\Delta \mapsto X_\Delta$ taking a subgroup $\Delta$ that is conjugate to some subgroup in $\mathcal{B}$ to a subset $X_\Delta \subseteq \Delta$ such that the following two conditions hold: \begin{enumerate}
\item $X$ is monotone: $\Delta_1 \subseteq \Delta_2 \implies X_{\Delta_1} \subseteq X_{\Delta_2}$.
\item $X$ is conjugation invariant: $g X_\Delta g ^{-1} = X_{g \Delta g ^{-1}}$
\end{enumerate} 

\begin{example} Let $\Gamma,\mathcal{B}$ be as in Definition \ref{def:sifter}. \begin{enumerate}
\item The assignment $\Delta \mapsto [\Delta,\Delta]$ is a $\mathcal{B}$-sifter.
\item More generally, if $w$ is a word then the assignment $\Delta \mapsto w(\Delta) \cup w(\Delta) ^{-1}$ is a $\mathcal{B}$-sifter.
\item If $\Gamma \subseteq \GL_n$ is a linear group, the assignment $\Delta \mapsto \left\{ g\in \Delta \mid \text{$g$ is unipotent} \right\}$ is a $\mathcal{B}$-sifter.
\end{enumerate}
\end{example}

\begin{example} Let $X$ be a $\mathcal{B}$-sifter on $\Gamma$. \begin{enumerate}
\item For every natural number $n$, the assignment 
\[
X^n_\Delta := \left\{ x_1 \cdots x_n \mid x_i \in X_\Delta \right\}
\]
is a $\mathcal{B}$-sifter. 
\item The assignment $\langle X \rangle_\Delta := \langle X_\Delta \rangle$ is a $\mathcal{B}$-sifter.
\item The assignment $X^*_{[\Delta_n]} := \left[X_{\Delta_n} \right]$ is a $\mathcal{B} ^*$-sifter on $\Gamma ^*$ (with the ultraproduct topology---see Lemma \ref{lem:ultraproduct.topology.base}).
\end{enumerate}
\end{example} 

\begin{definition} \label{def:thick.sifter} Let $X$ be a $\mathcal{B}$-sifter on a topological group $\Gamma$. We say that $X$ is {\bf thick} if the following conditions hold: \begin{enumerate}
\item For every $\Delta \in \mathcal{B}$, the closure of $X_\Delta$ has non-empty interior.
\item There is $\Delta_1\in \mathcal{B}$ such that for every $\Delta_2 \in \mathcal{B}$ there is $\Delta_3\in \mathcal{B}$ with $[\Delta_1,\Delta_3] \subseteq X_{\Delta_2}$.
\end{enumerate} 
\end{definition} 

\begin{lemma} \label{lem:thick.implies.thick} Let $X$ be a thick $\mathcal{B}$-sifter on a topological group $\Gamma$. Then \begin{enumerate}
\item For every $n$, $X^n$ is thick.
\item $\langle X \rangle$ is thick.
\item For every $\Delta\in \mathcal{B}$, the set $\langle X_\Delta \rangle$ is good in the sense of Definition \ref{def:good.subgroup.topology}.
\item $X^*$ is thick.
\end{enumerate} 
\end{lemma} 

\begin{proof} \begin{enumerate}
\item Clear.
\item Clear.
\item By definition, $X_\Delta$ is normal in $\Delta$ so $\langle X_\Delta \rangle$ is a normal subgroup of $\Delta$. In addition, the closure of $\langle X_\Delta \rangle$ contains an open subgroup $\Lambda$. By shrinking $\Lambda$ we can assume that $\Lambda \subseteq \Delta$, so $\langle X_\Delta \rangle \cap \Lambda$ is a dense normal subgroup of $\Lambda$.

Let $\Delta_1$ be as in Definition \ref{def:thick.sifter}. By shrinking $\Lambda$, we can assume that $\Lambda \subseteq \Delta_1$. We show that the second condition of Definition \ref{def:good.subgroup.topology} holds for $\Delta_X=\Lambda$. Let $x\in \Gamma$. By Definition \ref{def:thick.sifter} there is an open subgroup $\Delta_3$ such that $\left[ \Delta_1 , \Delta_3 \right] \subseteq X_{x ^{-1} \Lambda x}=x ^{-1} X_\Lambda x$. Thus,
\[
x \left[ \Lambda, \Delta_3\right] x ^{-1} \subseteq x \left[ \Delta_1 , \Delta_3 \right] x ^{-1} \subseteq X_\Lambda \subseteq X_\Delta,
\]
as required.

\item The claim follows from Lemma \ref{lem:closure.in.ultraproduct.topology} and a coordinate-wise argument.
\end{enumerate} 
\end{proof} 


We now specialize to the congruence topology.

\begin{notation} If $G \subseteq \GL_n$ is an algebraic subgroup defined over $K$, we denote the collection of principal congruence subgroups of $G(K)$ by $\mathcal{C}_{G,O}$.
\end{notation} 

By Lemma \ref{lem:ultraproduct.topology.base}, the ultrapower $\mathcal{C}_{G,O}^*$ is a base of neighborhoods of identity to the ultraproduct of the congruence topology on $G(K^*)$.


\begin{theorem} \label{thm:width.to.ultrapower} Let $G$ be a semisimple group defined over $K$. Assume that the $S$-congruence kernel of $G$ is finite. Let $X$ be a thick $\mathcal{C}_{G,O}$-sifter on $G(K)$. The following conditions are equivalent: \begin{enumerate}
\item There is a constant $A$ such that $\wid(X_\Delta) < A$ for every $\Delta \in \mathcal{C}_{G,O}$.
\item There is a constant $B$ such that for every $\Delta \in \mathcal{C}_{G,O}$ there is $\Lambda \in \mathcal{C}_{G,O}$ satisfying $\left[ \Lambda : \Lambda \cap X_{\Delta}^B \right] < B$.
\item For every $\Delta \in \mathcal{C}_{G,O}^*$ there is $\Lambda \in \mathcal{C}_{G,O}^*$ such that $\left[ \Lambda : \Lambda \cap \langle X_{\Delta}^* \rangle \right] < \infty$.
\end{enumerate} 
\end{theorem} 

For the proof, we will use the following:

\begin{definition}\label{def:separating} Let $\Gamma$ be a group and let $X$ be a symmetric subset of $\Gamma$ which contains the identity. 
\begin{enumerate}
\item A {\bf $(\Gamma,X)$-separating set} is a subset $T$ of $\Gamma$ which contains the identity element and such that for every distinct  $t_1,t_2 \in T$, $t_1X \cap t_2X=\emptyset$.
\item A  $(\Gamma,X)$-separating set is called maximal if it is not properly contained in any other $(\Gamma,X)$-separating set. We denote the minimal size of a maximal $(\Gamma,X)$-separating set by $[\Gamma : X]$.
\item A {\bf $(\Gamma,X)$-covering set} is a subset $T$ of $\Gamma$ which contains the identity element  and such that $XT=\Gamma$. Note that is $T$ is a maximal $(\Gamma,X)$-separating set then $T$ is a $(\Gamma,X^2)$-covering  set \end{enumerate}
\end{definition} 

\begin{lemma} \label{lem:expansion} Let $H \triangleleft G$ be groups, let $X \subseteq G$ be a symmetric subset such that $\langle X \rangle = X \cdot H$. Then $H \subseteq X^{2 \cdot 6^{[H : H \cap X]}}$.
\end{lemma} 

\begin{proof} Let $T$ be a maximal $(H,H\cap X)$-separating set and denote the size of $T$ by $t$. We can assume that $T \ne \{1 \}$. Denote $Y:=X^2$. Then $Y \supseteq X$, $\langle Y\rangle=Y \cdot H$ and 
$T$ is a $(H,Y\cap H)$-covering set. By induction on the size of a covering set, it is enough to prove that there exists a proper subset $T'$ of $T$ which is a $(H,Y^6\cap H)$-covering set. 
We divide the proof into cases.
\begin{enumerate}
	\item If $Y^2\supseteq H$ then $T'=\{1\}$ is an $(H,Y^2 \cap H)$-covering set. 
		\item Assume that $(Y^4 \cap H) \setminus Y \ne \emptyset$ and choose 	$g \in (Y^4 \cap H) \setminus Y$. Since $T$ is an a $(H,Y\cap H)$-covering set and $g \notin Y$, there exists $1 \ne t \in T$ such that $g \in tY$. Thus, $t \in Y^5$, $tY\subseteq Y^6$ and $T'=T \setminus \{t\}$ is an $(H,Y^6\cap H)$-covering set.

\item Assume that  $Y \cap H \ne H$.  If $Y^2\supseteq H$ we are back in case (1) so we can assume that $Y^2\not\supseteq H$.  Since $\langle Y\rangle=YH $, $Y^3\supsetneq Y^2$ and there are $g_1 \in Y$ and $g_2 \in H$ such that $g_1g_2\in Y^3\setminus Y^2$. Thus, $g_2 \in (Y^4\cap H) \setminus Y$ and we are back at step (2). 	
\end{enumerate}
\end{proof}

\begin{proof}[Proof of Theorem \ref{thm:width.to.ultrapower}] Denote the size of the $S$-congruence kernel of $G(K)$ by $C$. \begin{enumerate}
\item[$1. \implies 2.$] We show that the claim holds with $B=\max \left\{ A,C+1 \right\}$. Given $\Delta$, the normal subgroup $\langle X_{\Delta} \rangle$ has finite index in $\Delta$, so there is a principal congruence subgroup $\Lambda$ satisfying $\left[ \Lambda : \Lambda \cap \langle X_\Delta \rangle \right] \leq C$. Thus,
\[
\left[ \Lambda : \Lambda \cap X_\Delta ^B \right] = \left[ \Lambda : \Lambda \cap \langle X_\Delta \rangle \right] \leq C <B
\]
\item[$2. \implies 1.$] We show the claim for $A=2B \cdot 6^B+1$. Given $\Delta\in \mathcal{C}_{G,O}$, let $\Lambda\in \mathcal{C}_{G,O}$ be such that $\left[ \Lambda : \Lambda \cap X_\Delta ^B \right] < B$. 

We claim that the conditions of Lemma \ref{lem:expansion} hold for $H:=\Lambda \cap \langle X_\Delta \rangle$, $G:=\langle X_\Delta \rangle$, and $X:=X_\Delta^B$. Since $H \subseteq \langle X \rangle$, it is enough to prove that $X \cdot H$ is closed to left multiplications by $X$. Let $x_1,x_2\in X$ and let $h\in H$. Since the closure of $X$ is a subgroup, there is $x_3\in X$ and $h'\in \Lambda$ such that $x_1x_2=x_3h'$. This implies that $h'\in H$, so $x_1x_2h=x_3h'h\in X \cdot H$.

We now show that $X_\Delta^{A+1}=X_\Delta ^A$. Let $g\in X_\Delta ^{A+1}$. Since the closure of $X_\Delta$ is a group, $g=xh$, where $x\in X_\Delta$ and $h\in \Lambda$. Since $\left[ H : H \cap X \right] \leq \left[ \Lambda : \Lambda \cap X_\Delta^B \right] < B$, Lemma \ref{lem:expansion} implies that $h\in \Lambda \cap \langle X_\Delta \rangle = H \subseteq X^{2 \cdot 6^B}=X_\Delta^{A-1}$. This implies that $g\in X_\Delta ^A$.

\item[$2. \implies 3.$] A coordinate-wise argument shows that for every $\Delta \in \mathcal{C}_{G,O}^*$ there is $\Lambda \in \mathcal{C}_{G,O}^*$ such that $\left[ \Lambda : \Lambda \cap X_{\Delta}^B \right] < \infty$, which is a stronger statement than 3.

\item[$3. \implies 2.$] Arguing by contradiction, suppose that for every $n\in \mathbb{N}$ there is $\Delta_n\in \mathcal{C}_{G,O}$ for which $\left[ \Lambda : \Lambda \cap X_{\Delta_n}^n \right] > n$, for every $\Lambda \in \mathcal{C}_{G,O}$. Let $\Delta=[\Delta_n]\in \mathcal{C}_{G,O}^*$ and let $\Lambda=[\Lambda_n]\in \mathcal{C}_{G,O}^*$ be such that $\left[ \Lambda : \Lambda \cap \langle X_\Delta ^* \rangle \right] < \infty$. 

Let $N:=\left[ \Lambda : \Lambda \cap \langle X_\Delta ^* \rangle \right]$ and let $g_1,\ldots,g_N$ be coset representatives of $\Lambda \cap \langle X^*_\Delta \rangle$ in $\Lambda$. Since
\[
\Lambda=\bigsqcup_{i=1}^N g_i \left( \Lambda \cap \langle X^*_{\Delta} \rangle \right) = \bigcup_{n=1}^\infty \left( \bigsqcup_{i=1}^N g_i \left( \Lambda \cap (X^*_{\Delta})^n \right) \right),
\]
Lemma \ref{lem:internal.saturation} implies that there is a natural number $M$ such that 
\[
\Lambda=\bigsqcup_{i=1}^N g_i \left( \Lambda \cap \left( X^*_{\Delta} \right) ^M \right).
\]
Writing $g_i=[g_{i,n}]$, we get that
\[
\Lambda_n=\bigsqcup_{i=1}^N g_{i,n} \left( \Lambda_n \cap X_{\Delta_n}^M \right) 
\]
for $\mathcal{U}$-almost all natural numbers $n$. For those natural numbers, $\left[ \Lambda_n : \Lambda_n \cap X_{\Delta_n}^M \right] \leq N$, a contradiction.
\end{enumerate} 
\end{proof} 

\begin{definition} Let $G$ be a semisimple group over $K$ and let $\Gamma \subseteq G(K)$ be an $S$-arithmetic subgroup. We say that $\Gamma$ has bounded commutator width if there is a constant $C$ such that $\wid_{\Gamma}\left( \left[ \Gamma[I] , \Gamma[I] \right] \right) < C$, for every $I\in \mathcal{I} ^K$.
\end{definition} 

\begin{theorem} \label{thm:commutator.width.finite.metaplectic} Let $G$ be a semisimple group over $K$. Assume that \begin{enumerate}
\item The $S$-congruence kernel of $G$ is finite.
\item There is $n$ such that the $\mathcal{C}_{G,O}$-sifter $X_\Delta := \left[ \Delta , \Delta \right] ^n$ is thick.
\item Every congruence subgroup of $G(K^*)$ normally generates.
\item Every discrete metaplectic extension of $G(K^*)$ has finite kernel.
\end{enumerate} 
Then $G(O)$ has bounded commutator width.
\end{theorem} 

\begin{proof} By \ref{thm:width.to.ultrapower} it is enough to prove that for every $\Delta \in \mathcal{C}_{G,O}^*$ there is $\Lambda \in \mathcal{C}_{G,O}^*$ for which $\left[ \Lambda : \Lambda \cap \langle X^*_\Delta \rangle \right]$. 

Let $\tau$ be the congruence topology on $G(K)$ and let $\Delta \in \mathcal{C}_{G,O}^*$. By Lemma \ref{lem:thick.implies.thick} the $\mathcal{C}_{G,O}^*$-sifter $\langle X^* \rangle$ is thick and by the same lemma, the subgroup $Y:=\langle X^*_{\Delta} \rangle$ is good. By Lemma \ref{lem:ses.general} and the third assumption, the short exact sequence
\[
1 \rightarrow \ker \rho_{\tau ^*,Y} \rightarrow \widehat{\Spin_{f_a}(K^*)}_{\tau^*(Y)} \rightarrow \widehat{\Spin_{f_a}(K^*)}_{\tau ^*} \rightarrow 1
\]
is a discrete $\Spin_{f_a}(K^*)$-metaplectic extension. By the assumption, $\ker \rho_{\tau ^*,Y}$ is finite, but by Lemma \ref{lem:ses.general}, there is a subgroup $\Lambda \in \mathcal{C}_{G,O}^*$ such that $\ker \rho_{\tau ^*,Y}=\frac{\Lambda}{\Lambda \cap \langle X^*_\Delta \rangle}$.
\end{proof} 

\section{Approximation in tori} \label{sec:approximation.tori}

\subsection{The tori $T_t$} \label{subsec:T_t.standard}

\begin{definition} Recall that $S_{def}$ is the set of real places $v$ of $K$ for which $G(K_v)$ is compact. Define $O_+:=\left\{ x\in O \mid \left( \forall v\in S_{def} \right) \, v(x)>0 \right\}$.
\end{definition} 
By our assumption, $S_{def}\neq \emptyset$, so if $t\in O_+$ then $-t$ is not a square in $K$.

\begin{definition} Let $\mathbb{G}_m=\Spec \left( O[x,y]/(xy-1) \right)$ be the split 1-dimensional torus over $O$. Given a prime $Q\in \mathcal{P} ^K$ and an element $(x,x ^{-1})\in \mathbb{G}_m(K_Q)$, we denote $\val_Q \left( (x,x ^{-1}) \right) := \min \left\{ \val_Q(x),\val_Q(x ^{-1}) \right\}$.
\end{definition} 

\begin{definition} For a field extension $F/K$, let $S_F$ be the set of places of $F$ whose restrictions to $K$ belong to $S$ and let $O_F$ be the ring of $S_F$-integers in $F$.
\end{definition} 

\begin{definition} Suppose that $t\in O_+$. \begin{enumerate}
\item Let $\widetilde{T_t}$ be the restriction of scalars of $\mathbb{G}_m \times \Spec \left( O_{K(\sqrt{-t})} \right)$ from $O_{K(\sqrt{-t})}$ to $O$. By definition, $\widetilde{T_t}$ is an algebraic group over $\Spec(O_K)$ and, for every $O$-algebra $R$, $\widetilde{T_t}(R)=\mathbb{G}_m(R \otimes_O O_{K(\sqrt{-t})})$. For the proof of existence of restrictions of scalars and their basic properties, see \cite[\S7.6]{BLR90}.
\item For $\alpha \in \widetilde{T_t}(K)$, denote $\supp(\alpha)=\left\{ Q\in \mathcal{P} ^K \mid \alpha \notin \widetilde{T_t}(O_Q) \right\}$.
\item Let $N_{\widetilde{T_t}/ \mathbb{G}_m}: \widetilde{T_t} \rightarrow \mathbb{G}_m$ be the norm homomorphism defined as follows: for every $O$-algebra $R$, the Galois involution of $K(\sqrt{-t})$ induces an involution $\alpha \mapsto \overline{\alpha}$ of $R \otimes_O O_{K(\sqrt{-t})}$ and $N_{\widetilde{T_t}/ \mathbb{G}_m}(\alpha)=\alpha \overline{\alpha}\in R^ \times \cong \mathbb{G}_m(R)$. If $\alpha=a \otimes 1+b \otimes \sqrt{-t}$, then $N_{\widetilde{T_t}/ \mathbb{G}_m}(\alpha)=a^2+tb^2$.
\item We denote the kernel of $N_{\widetilde{T_t}/ \mathbb{G}_m}$ by $T_t$. It is an anisotropic 1-dimensional torus defined over $O$.
\end{enumerate} 
\end{definition} 

We will use the following properties of $T_t$:

\begin{lemma} \label{lem:properties.T_t} Let $t\in O_+$. \begin{enumerate}
\item The image of $T_t(O)\hookrightarrow \mathbb{G}_m(O_{K(\sqrt{-t})}) \hookrightarrow O_{K(\sqrt{-t})}$ is the set of norm-one elements in $O_{K(\sqrt{-t})}$.
\item $T_t$ is $K$-isomorphic to the zero locus of the polynomial $x^2+ty^2-1$ in $\mathbb{A}_K^2$. In particular, $T_t$ has weak approximation and $T_t(K)=\left\{ x\in K(\sqrt{-t}) \mid N_{K(\sqrt{-t})/K}(x)=1 \right\}$.
\item If $Q\in \mathcal{P}^K$ is ramified or inert in $K(\sqrt{-t})$ then $T_t(K_Q)=T_t(O_Q)$.
\item If $Q\in \mathcal{P}^K$ is split in $K(\sqrt{-t})$ then there is an $O_Q$-isomorphism $\rho_Q :T_t \stackrel{\cong}{\rightarrow} \mathbb{G}_m$. Such an isomorphism is unique up to post-composition with the flip $(x,y)\mapsto (y,x)$ and is determined by a choice of a square root of $-t$ in $K_Q$. If $P$ is a prime of $K(\sqrt{-t})$ lying above $Q$ and $z\in T_t(K)$, then
\[
\val_Q \left( \rho_Q(z) \right) = \min \left\{ \val_P(z) , -\val_P(z) \right\},
\]
where, in the RHS, we identify $T_t(K)$ with the norm-one elements in $K(\sqrt{-t})$.
\end{enumerate} 
\end{lemma} 

\begin{proof} By definition, $\widetilde{T_t}(O)=\mathbb{G}_m(O_{K(\sqrt{-t})})$, so the first claim follows.

If $R$ is a PID containing $O$, then there is an isomorphism of $R$-modules $R \otimes_O O_{K(\sqrt{-t})}\cong R \cdot 1 \oplus R \cdot \alpha$, for some $\alpha$. If $R$ is a field, one can take $\alpha=\sqrt{-t}$. Claims 2,3,4 follow from this.
\end{proof} 

\begin{lemma} \label{lem:T_t.bounded.complexity} There is a constant $C$ such that all the schemes $T_t$ have complexity bounded by $C$.
\end{lemma} 

\begin{proof} Suppose first that $O_K$ is a PID. In this case, $O_{K(\sqrt{-t})}$ is a free $O_K$-module to which we can choose an $O_K$ basis of the form $1,\alpha$. Write $\alpha^2=a+b \alpha$, with $a,b\in O_K$. Writing the equation $(x_1+x_2 \alpha)(y_1+y_2 \alpha)=1$ in coordinates, we get that the restriction of scalars $\widetilde{T_t}$ is isomorphic to the affine scheme of 4-tuples $(x_1,x_2,y_1,y_2)$ satisfying the equations
\[
x_1y_1+ ax_2y_2=1 \quad \text{ and } \quad x_1y_2+y_1x_2+bx_2y_2=0.
\]
Since the norm map $N_{\widetilde{T_t}/ \mathbb{G}_m}$ is a quadratic polynomial in the variables $x_1,x_2,y_1,y_2$, the scheme $T_t$ is the zero locus of three quadratic polynomials in 4 variables.

If $O_K$ is not a PID, choose two disjoint sets $T_1,T_2 \subseteq \mathcal{P} ^K \smallsetminus S$ of representatives of the elements of $\Cl_K$. Let $R_1$ be the ring of $S\cup T_1$-integers in $K$, let $R_2$ be the ring of $S\cup T_2$-integers in $K$, and let $R_3$ be the ring of $S\cup T_1\cup T_2$-integers in $K$. Then $R_1,R_2,R_3$ are PIDs and $T_t$ is a gluing of $T_t \otimes \Spec(R_1)$ and $T_t \otimes \Spec(R_2)$ along $T_t \otimes \Spec(R_3)$, which, by the first part of the proof, are all schemes of bounded complexity.
\end{proof} 

\begin{definition} Let $t\in O_+$. \begin{enumerate}
\item Let $\mathcal{S}(T_t)=\left\{ Q\in \mathcal{P} ^K \mid \text{$Q$ splits in $K(\sqrt{-t})$} \right\}=\left\{ Q\in \mathcal{P} ^K \mid T_t\cong \mathbb{G}_m \text{ over $O_Q$}\right\}$.
\item A trivialization of $T_t$ is a morphism $\rho:T_t \times \Spec(\mathbf{A}) \rightarrow \mathbb{G}_m \times \Spec(\mathbf{A})$ such that $\rho \times \Spec(\mathbf{A}_{\mathcal{S}(T_t)})$ is an isomorphism and $\rho \times \Spec(\mathbf{A}_{\mathcal{P} ^K \smallsetminus \mathcal{S}(T_t)})$ is the trivial homomorphism.
\item If $\rho$ is a trivialization and $Q\in \mathcal{P} ^K$, we denote $\rho_Q: T_t(K_Q) \rightarrow \mathbb{G}_m(K_Q)$ the map induced by $\rho \times \Spec(K_Q)$.
\item If $\rho$ is a trivialization and $Q\in \mathcal{S}(T_t)$, then $\rho_Q$ gives rise to an embedding $K(\sqrt{-t}) \hookrightarrow K_Q$. The preimage of $Q \cdot O_Q$ is a prime in $\mathcal{P} ^{K(\sqrt{-t})}$ lying over $Q$. We denote this prime by $Q^ \rho$.
\end{enumerate} 
\end{definition} 

\subsection{Approximation in $\mathbb{G}_m$}

\begin{lemma} \label{lem:Hensel} Let $P\in \mathcal{P}^K$, let $N\in \mathbb{N}$, and let $a,b\in O_P$. Suppose that $N>2\val_P(2)$ and $a\equiv b^2\text{ (mod $P^{2\val_P(b)+N}$)}$. Then there is an element $c\in O_P$ such that $c\equiv b\text{ (mod $P^{\val_P(b)+N-\val_P(2)}$)}$ and $a=c^2$.
\end{lemma} 

\begin{proof} Let $\pi\in O_P$ be a uniformizer and let $e=\frac{b^2-a}{\pi^{2\val_P(b)+N}}\in O_P$. For $y\in O_P$, 
\[
(b+\pi^{\val_P(b)+N-\val_P(2)}y)^2-a=b^2+2b \pi^{\val_P(b)+N-\val_P(2)}y+\pi ^{2\val_P(b)+2N-2\val_P(2)}y^2-a=
\]
\[
=\pi^{2\val_P(b)+N} \left( e+b'y+\pi^{N-2\val_P(2)}y^2 \right),
\]
where $b'\in O_P^ \times$. By Hensel's lemma, there is $y\in O_P$ such that $e+b'y+\pi^{N-2\val_P(2)}y^2=0$, and we can take $c=b+\pi^{\val_P(b)+N-\val_P(2)}y$.
\end{proof} 

\begin{lemma} \label{lem:a2.approximation.basic} Let $a\in \mathbb{G}_m(\mathbf{O})[2]$ and let $I\in \mathcal{I}^K$. There is an element $t\in O_+$, a trivialization $\rho$ of $T_t$, and an element $\alpha \in T_t(O)$ such that \begin{enumerate}
\item $\mathcal{P}^K(I) \subseteq \mathcal{S}(T_t)$.
\item $\left( a \rho(\alpha) ^{-1} \right)_{\mathcal{P}^K(I)}\in \mathbb{G}_m(\mathbf{O})[I]$.
\item If $\sqrt{-3}\notin K$, then $\sqrt{-3}\notin K(\sqrt{-t})$.
\end{enumerate} 
\end{lemma} 

\begin{proof}  We can assume that $2|I$. Let $N=\max \left\{ \val_Q(I)+\val_Q(2) \mid Q\in \mathcal{P} ^K(I) \right\}$. Changing $a$ by some element in $\mathbb{G}_m(\mathbf{O})[I]$, we can also assume that $a_Q\not\equiv \pm1\text{ (mod $Q^{2\val_Q(I)}$)}$, for every $Q\in \mathcal{P} ^K(I)$. This implies that $a_Q\not\equiv a_Q ^{-1} \text{ (mod $Q^{4\val_Q(I)}$)}$, for every $Q\in \mathcal{P} ^K(I)$. 

Since $\Spin_{f_a}(O)$ is infinite, $S \neq S_{def}$. Dirichlet's unit theorem implies that there is a unit $u\in \mathbb{G}_m(O)$ such that $u_v > 1$, for every $v\in S_{def}$. By assumption, $a+a ^{-1} \in 2 \mathbf{O}$, so there is an element $x\in O$ such that $x\equiv \frac{a+a ^{-1}}{2}\text{ (mod $I^{8N}$)}$. Choose $k\in \mathbb{N}$ such that \begin{enumerate}
\item $t:=u^{2k}-x^2 \in O_+$ (this holds for $k$ large enough).
\item $3t$ is not a square in $O$ (this holds for almost all $k\in 5 \mathbb{N}$ because, by Siegel's theorem, the polynomial $P(X,Y):=3(X^{10}-x^2)-Y^2$ has finitely many solutions).
\item $u^k\equiv 1\text{ (mod $I^{8N}$)}$ (this holds for all $k\in D \mathbb{N}$, for some $D$).
\end{enumerate} 
The element $-t$ is not a square in $K$ because $t\in O_+$. Moreover, if $\sqrt{-3}\in K(\sqrt{-t}) \smallsetminus K$, then $3t$ is a square in $O$, a contradiction. Let $\alpha := \frac{x}{u^k}+\frac{1}{u^k}\sqrt{-t}\in K(\sqrt{-t})$. Since $N_{K(\sqrt{-t})/K}(\alpha)=\frac{x^2}{u^{2k}}+\frac{t}{u^{2k}}=1$, we get that $\alpha \in T_t(O)$. 

For every $Q\in \mathcal{P} ^K(I)$, 
\[
-t=x^2-u^{2k} \equiv \left( \frac{a_Q+a_Q ^{-1}}{2} \right)^2-1 = \left( \frac{a_Q-a_Q ^{-1}}{2} \right) ^2 \quad \text{ (mod $Q^{8\val_Q(I)}$)}.
\]
Since $\val_Q (a_Q-a_Q ^{-1})<2\val_Q(I)+2 \leq 4\val_Q(I)$, Lemma \ref{lem:Hensel} implies that there an element $\xi_Q\in O_Q$ such that $\xi_Q^2=-t$ and $\xi_Q\equiv \frac{a_Q-a_Q ^{-1}}{2}\text{ (mod $Q^{\val_Q(I)}$)}$. Letting $\rho$ be a trivialization of $T_t$ such that, for every $Q\in \mathcal{P} ^K(I)$, $\rho_Q$ corresponds to $\xi_Q$ (see Lemma \ref{lem:properties.T_t}), we get that
\[
\rho_Q(\alpha)=\frac{x}{u^k}+\frac{\xi_Q}{u^k}\equiv x+ \xi_Q\equiv \frac{a_Q+a_Q ^{-1}}{2}+ \frac{a_Q-a_Q ^{-1}}{2} = a_Q \quad \text{ (mod $Q^{\val_Q(I)}$)},
\]
so $a_Q \rho_Q(\alpha) ^{-1} \in \mathbb{G}_m(O_Q)[Q^{\val_Q(I)}]$.
\end{proof} 



Before stating the next result, we recall a few notions from Class Field Theory. Given a number field $L$ and an ideal $I$ of the ring of integers of $L$, the ray class group $\Cl_L^J$ is the quotient
\[
\Cl_L^J:= \left\{ \text{fractional ideals prime to $J$} \right\} / \left\{ \text{principal ideals $(a)$ with $a\equiv1\text{ (mod $J$)}$} \right\}.
\]
The ray class field is the unique abelian extension $R/L$ unramified outside $\mathcal{P} ^L(J)$ for which the Artin map $P \mapsto \left( \frac{R/L}{P} \right)$ descends to an isomorphism $\Cl_L^J \rightarrow \Gal(R/L)$. If $J=1$, the ray class group is the (ordinary) class group and the ray class field is called the Hilbert class field.

\begin{lemma} \label{lem:a1.approximation.basic} Let $a\in \mathbb{G}_m(\mathbf{A})$, let $t\in O_+$, let $\rho$ be a trivialization of $T_t$, and let $J\in \mathcal{I} ^K$. Suppose that $\mathcal{P} ^K(J) \subseteq \mathcal{S}(T_t)$, that $\supp(a)=\left\{Q_1,\ldots,Q_n \right\} \subseteq  \mathcal{S}(T_t) \smallsetminus \mathcal{P} ^K(J)$, and that the fractional ideal $(Q_1^ \rho)^{\val_{Q_1}(a)} \cdots (Q_n^ \rho)^{\val_{Q_n}(a)}$ of $K(\sqrt{-t})$ is trivial in the ray class group $\Cl_{K(\sqrt{-t})}^J$. Then there is an element $\alpha \in T_t(K)$ such that $\rho(\alpha)_{\mathcal{P} ^K(J)}\in \mathbb{G}_m(\mathbf{O})[J]$ and $a \rho(\alpha) ^{-1}\in \mathbb{G}_m(\mathbf{O})$.
\end{lemma} 

\begin{proof} Let $\sigma\in \Gal(K(\sqrt{-t})/K)$ be the non-trivial element. By the assumption, there is an element $\beta \in K(\sqrt{-t})$ such that $\beta \equiv 1\text{ (mod $J$)}$ and $\beta \cdot O_{K(\sqrt{-t})}=(Q_1^ \rho)^{\val_{Q_1}(a)} \cdots (Q_n^ \rho)^{\val_{Q_n}(a)}$. For $P\in \mathcal{P}^{K(\sqrt{-t})}$, 
\[
\val_P(\beta)=\begin{cases} \val_{Q_i}(a) & P=Q_i^\rho \\ 0 & \text{else} \end{cases}.
\]
Define $\gamma=\frac{\beta}{\sigma(\beta)}\in K(\sqrt{-t})^ \times$. Since $N_{K(\sqrt{-t})/K}(\gamma)=1$, $\gamma$ corresponds to an element of $T_t(K)$ which we denote by $\alpha$. Since $\beta \equiv 1\text{ (mod $J$)}$, we get that $\gamma\equiv1\text{ (mod $J$)}$, so $\alpha_{\mathcal{P} ^K(J)}\in T_t(\mathbf{O})[J]$ and the first claim holds.

For $i=1,\ldots,n$, $Q_i$ splits in $K(\sqrt{-t})$ (since $Q_i\in \mathcal{S}(T_t)$), so $\sigma(Q_i^ \rho)\neq Q_i^ \rho$. Thus, for $P\in \mathcal{P}^{K(\sqrt{-t})}$,
\[
\val_P(\gamma)=\begin{cases} \val_{Q_i}(a) & P=Q_i^ \rho \\ -\val_{Q_i}(a) & P=\sigma(Q_i^ \rho) \\ 0 & \text{else} \end{cases}.
\]
For every $Q\in \mathcal{P}^K$, since $\val_Q(a) \leq 0$ and by Lemma \ref{lem:properties.T_t},
\[
\val_Q(\rho(\alpha))=\begin{cases} \val_{Q_i}(a) & Q=Q_i \\ 0 & \text{else} \end{cases},
\]
which implies that $a \rho(\alpha) ^{-1}\in \mathbb{G}_m(\mathbf{O})$.
\end{proof} 

\subsection{Approximation in $\mathbb{G}_m \times \mathbb{G}_m$ I}

\begin{lemma} \label{lem:weak.approx.standard} Let $I\in \mathcal{I}^K$ and let $a_1,a_2\in \mathbb{G}_m(\mathbf{A})$. There is an element $t\in O_+$, a trivialization $\rho$ of $T_t$, and elements $\alpha_1,\alpha_2\in T_t(K)$ such that \begin{enumerate}
\item $\mathcal{P} ^K(I) \subseteq \mathcal{S}(T_t)$ and $\left( a_i \rho(\alpha_i) ^{-1} \right)_{\mathcal{P}^{K}(I)} \in \mathbb{G}_m(\mathbf{O})[I]$, for $i=1,2$.
\item The sets $\mathcal{P} ^{K}(I), \mathcal{P} ^K(tO), \supp(\alpha_1) \smallsetminus \mathcal{P} ^{K}(I),\supp(\alpha_2) \smallsetminus \mathcal{P} ^{K}(I)$ are pairwise disjoint.
\end{enumerate} 
\end{lemma} 

\begin{proof} Let $t\in O_+$ be an element such that $t\equiv -1\text{ (mod I)}$ and let $\rho$ be any trivialization of $T_t$. The congruence condition implies that the torus $T_t$ splits over $\mathbf{A}_{\mathcal{P} ^K(I)}$, i.e., $\mathcal{P} ^K(I) \subseteq \mathcal{S}(T_t)$.

By the congruence condition on $t$, the sets $\mathcal{P} ^K(I)$ and $\mathcal{P} ^K(tO)$ are disjoint. By the weak approximation theorem, there is an element $\alpha_1\in T_t(K)$ such that $\left( a_1 \rho(\alpha_1) ^{-1} \right)_{\mathcal{P}^{K}(I)} \in \mathbb{G}_m(\mathbf{O})[I]$ and such that $( \alpha_1)_{\mathcal{P} ^K(tO)}\in T_t(\mathbf{O})$.

Let $\mathcal{R}_1 :=\supp(\alpha_1) \smallsetminus \mathcal{P}^K(I)$. By construction, $\mathcal{R}_1$ is disjoint from $\mathcal{P} ^K(tO)$ and by Lemma \ref{lem:properties.T_t}, $\mathcal{R}_1 \subseteq \mathcal{S}(T_t)$.

Since $\mathcal{P} ^K(I)$ and $\mathcal{R}_1 \cup \mathcal{P} ^K(tO)$ are disjoint, the weak approximation theorem implies that there is an element $\alpha_2\in T_t(K)$ such that $\left( a_2 \rho(\alpha_2) ^{-1} \right)_{\mathcal{P} ^K(I)}\in \mathbb{G}_m(\mathbf{O})[I]$ and such that $\rho(\alpha_2)_{\mathcal{R}_1 \cup \mathcal{P} ^K(tO)}\in T_t(\mathbf{O})$. By construction, $\supp(\alpha_2) \smallsetminus \mathcal{P} ^K(I)$ is disjoint from $\mathcal{R}_1$ and $\mathcal{P} ^K(tO)$.
\end{proof} 

\begin{lemma} \label{lem:approx.single.support.standard} Let $a_1\in \mathbb{G}_m(\mathbf{A})$, $a_2\in \mathbb{G}_m(\mathbf{O})[2]$, and $I\in \mathcal{I} ^K$. There is a prime $Q\in \mathcal{P} ^K \smallsetminus \mathcal{P} ^K(I)$, an element $t\in O_+$, a trivialization $\rho$ of $T_t$, and elements $\alpha_1\in T_t(K), \alpha_2\in T_t(O)$ such that \begin{enumerate}
\item \label{claim:approx.single.support.standard.a1} $\left( a_1 \rho(\alpha_1) ^{-1} \right)_{\mathcal{P} ^K(I)}\in \mathbb{G}_m(\mathbf{O})$.
\item \label{claim:approx.single.support.standard.a2} $\left( a_2 \rho(\alpha_2) ^{-1} \right)_{\mathcal{P} ^K(I)}\in \mathbb{G}_m(\mathbf{O})[I]$.
\item \label{claim:approx.single.support.standard.supp} $\supp(\alpha_1) \smallsetminus \mathcal{P} ^K(I)=\left\{ Q \right\}$ and $\val_Q(\alpha_1)=-1$.
\end{enumerate} 
\end{lemma} 

\begin{proof} Without loss of generality, we can assume that $\supp(a_1) \subseteq \mathcal{P}^K(I)$. Let $t,\rho,\alpha_2$ be obtained by applying Lemma \ref{lem:a2.approximation.basic} to $a_2,I$. Claim \eqref{claim:approx.single.support.standard.a2} holds by construction. 

Let $H_t$ be the Hilbert class field of $K(\sqrt{-t})$ and let $\sigma \in \Gal(H_t/K(\sqrt{-t}))$ be the element corresponding to the ideal $\prod_{Q\in \mathcal{P}^K(I)} \left( Q^ \rho \right)^{\val_Q(a_1)}$ under the Artin isomorphism $\Gal(H_t/K(\sqrt{-t}))\cong \Cl_{K(\sqrt{-t})}$. By Chebotarev's density theorem, there is an ideal $P\in \mathcal{P} ^{K(\sqrt{-t})}$ relatively prime to $I$ such that $\left( \frac{H_t/K(\sqrt{-t})}{P} \right) = \sigma$. Since the set of ideals of $O_{K(\sqrt{-t})}$ whose intersection with $K$ splits in $K(\sqrt{-t})$ has density 1, we can further require that $P\cap K$ splits in $K(\sqrt{-t})$. Let $Q:= P \cap K$ and note that $Q\in \mathcal{S}(T_t) \smallsetminus \mathcal{P} ^K(I)$.

Let $b\in \mathbb{G}_m(\mathbf{A})$ be such that $b_{\mathcal{P} ^K(I)}=a_{\mathcal{P} ^K(I)}$, such that $\val_Q(b)=1$, and such that $\supp(b) \subseteq \mathcal{P} ^K(I) \cup \left\{ Q \right\}$. It follows that $\prod_{Q\in \supp(b)} \left( Q^ \rho \right)^{\val_Q(b)}$ is principal so Lemma \ref{lem:a1.approximation.basic} implies that there is an element $\alpha_1\in T_t(K)$ such that $b \rho (\alpha_1) ^{-1}\in \mathbb{G}_m(\mathbf{O})$. This implies Claims \eqref{claim:approx.single.support.standard.a1} and \eqref{claim:approx.single.support.standard.supp}.
\end{proof} 

\subsection{Approximation in $\mathbb{G}_m \times \mathbb{G}_m$ II}

\begin{lemma} \label{lem:Artin.ext} Let $K \subseteq L \subseteq M$ be finite Galois extensions of $K$ and let $P\in \mathcal{P}^L$. Assume that $P\cap K$ splits completely in $L$ and is unramified in $M$. Then $\left( \frac{M/L}{P} \right)=\left( \frac{M/K}{P\cap K} \right)\cap \Gal(M/L)$.
\end{lemma} 

\begin{proof} Let $Q\in \mathcal{P}^M$ lie above $P$ and let $\Delta_Q \subseteq \Gal(M/L)$ be the decomposition group of $Q$. Since $P\cap K$ splits completely in $L$, $\left( \frac{M/K}{P\cap K} \right)\cap \Gal(M/L)$ is a single conjugacy class in $\Gal(M/L)$. Since both $\left( \frac{M/L}{P} \right)\cap \Delta_Q$ and $\left( \frac{M/K}{P\cap K} \right)\cap \Gal(M/L)\cap \Delta_Q$ consist of the unique element in $\Gal(M/L)$ that fixes $Q$ and acts on $O_M/Q$ as the map $x \mapsto x^{|O_L/P|}=x^{|O_K/P\cap K|}$, we get $\left( \frac{M/L}{P} \right)=\left( \frac{M/K}{P\cap K} \right)\cap \Gal(M/L)$.
\end{proof} 

For a group $G$ and an integer $N$, denote the set of elements in $G$ that are $N$th powers by $\,^N G$.

\begin{lemma} \label{lem:prod.groups} Let $G_1,\ldots,G_n$ be groups, let $G=\prod_1^n G_i$, let $H \triangleleft G$ be a normal subgroup, and assume that $[G:H]<|G_i|$, for all $i$. Denote $G^*=\prod_1 ^n (G_i \smallsetminus \left\{ 1 \right\})$. Then $(G^*\cap H) \cdot (G^*\cap H) \supseteq \,^{2[G:H]}G$.
\end{lemma}

\begin{proof} For every $A \subseteq \left\{ 1,\ldots,n \right\}$, let $G_A=\prod_{i\in A} G_i$, let $G^*_A=\prod_{i\in A} \left( G_i \smallsetminus \left\{ 1 \right\} \right)$, and let $1_A\in G_A$ be the identity element. If $A,B$ are disjoint, then $G_{A\cup B}=G_A \times G_B$ and $G^*_{A\cup B}=G^*_A \times G^*_B$. 

Assume first that $H=G$. For every $i$, $G^*_{\left\{ i \right\}}\neq \emptyset$, and if $G_{\left\{ i \right\}}\neq C_2$ then $G^*_{\left\{ i \right\}} \cdot G^*_{\left\{ i \right\}}=G_{\left\{ i \right\}}$. Thus $G^* \cdot G^* \supseteq \,^{2}G$.

In the general case, let $A=\left\{ i \mid \,^{[G:H]}G_i=1 \right\}$. By the assumption that $[G:H]<|G_i|$ and a pigeonhole argument, there is $x_A\in G^*_A$ such that $x_A \times 1_{A^c}\in H$. Since $H \supseteq \,^{[G:H]}G \supseteq 1_A \times\,^{[G:H]}G_{A^c}$, we get that $\left\{ x_A,x_A ^{-1} \right\} \times \left( G^*_{A^c} \cap \,^{[G:H]}G_{A^c} \right) \subseteq G^* \cap H$. Using the special case of the claim for the group $\,^{[G:H]}G_{A^c}$, we get that 
\[
(G^*\cap H) \cdot (G^*\cap H) \supseteq \Big( x_A \times \left( G^*_{A^c} \cap \,^{[G:H]}G_{A^c} \right) \Big) \cdot \Big( x_A ^{-1} \times \left( G^*_{A^c} \cap \,^{[G:H]}G_{A^c} \right) \Big) \supseteq 
\]
\[
\supseteq 1_A \times \left( \,^{2[G:H]}G_{A^c} \right)=\,^{2[G:H]}G
\]
\end{proof}  

\begin{lemma} \label{lem:cyclotomic.moving} For every number field $E$ and for every natural number $N$ divisible by $\left( [E: \mathbb{Q} ]! \right)^3$, the set
\[
\Xi := \left\{ \sigma \in \Gal(E(\zeta_N)/E) \mid \left( \forall \ell |N \right) \, \sigma(\zeta_{\ell^{\val_\ell(N)}}) \neq \zeta_{\ell^{\val_\ell(N)}} \right\} 
\]
satisfies $\Xi \cdot \Xi \supseteq \,^{2[E: \mathbb{Q} ]}\Gal(E(\zeta_N)/E)$.
\end{lemma} 

\begin{proof} Let $M=\left( [E: \mathbb{Q} ]! \right)^3$. If $\ell \leq [E: \mathbb{Q} ]$ is a prime number, then $\ell^{[E: \mathbb{Q} ]} \geq [E: \mathbb{Q} ]^\ell$. Using the inequality $3\lfloor x \rfloor>x+1$, valid for $x \geq 1$, we get 
\begin{equation} \label{eq:cyclotomic.moving.val}
\ell^{\val_\ell(M)} \geq \ell^{3\left\lfloor \frac{[E : \mathbb{Q} ]}{\ell} \right\rfloor} > \ell^{\frac{[E: \mathbb{Q} ]}{\ell}+1} > [E: \mathbb{Q} ]+1.
\end{equation}

Suppose that $N$ is divisible by $M$. Let $\varphi:\Gal(E(\zeta_N)/E) \rightarrow \Gal(\mathbb{Q}(\zeta_N)/ \mathbb{Q})$ be the restriction map. The group $H:=\varphi \left( \Gal(E(\zeta_N)/E) \right)$ is a subgroup of index at most $[E: \mathbb{Q} ]$ of $G:=\Gal(\mathbb{Q}(\zeta_N)/ \mathbb{Q})=\mathbb{G}_m(\mathbb{Z} / N)=\prod_{\ell | N} \mathbb{G}_m(\mathbb{Z} / \ell^{\val_\ell(N)})$. By \eqref{eq:cyclotomic.moving.val}, the conditions of Lemma \ref{lem:prod.groups} hold and since $\Xi=\varphi ^{-1} \left( G^* \cap H \right)$, the claim follows.
\end{proof} 

The following follows from \cite[Theorems 3.1 and 4.1]{Lens77}

\begin{theorem} \label{thm:Lenstra.GRH} Let $F/E$ be a Galois extension of number fields, let $\sigma\in \Gal(F/E)$ be a conjugacy class, let $\gamma \in \mathbb{G}_m(O_E)$ be a primitive unit, and let $M$ be an integer. Assume that, for every prime $\ell$, if $\zeta_{\ell^{\val_\ell(M)+1}} \in F$ then $\sigma(\zeta_{\ell^{\val_\ell(M)+1}}) \neq \zeta_{\ell^{\val_\ell(M)+1}}$. Assuming the Generalized Riemann Hypothesis, the set
\[
\left\{ P\in \mathcal{P}^E \mid \sigma \in \left( \frac{F/E}{P} \right) \text{ and }\Big[ \mathbb{G}_m(O_E/P) : \left \langle \gamma \text{ (mod $P$)} \right \rangle \Big] \text{ divides } M \right\}
\]
has positive density.
\end{theorem} 

\begin{lemma} \label{lem:two.quadratic.extensions} Assume the GRH. Let $F/E$ be a quadratic extension of number fields and let $\gamma \in O_F$ be a primitive unit. There are natural numbers $M_0,M_1,M_2$ such that, for every quadratic extension $L/E$, every ideal $I\in \mathcal{I}^L$, and every finite set $T \subseteq \mathcal{P}^\mathbb{Q}$ of prime numbers, there are prime ideals $Q_1,Q_2\in \mathcal{P}^E$, $P_1,P_2\in \mathcal{P}^F$, $R_1,R_2\in \mathcal{P}^L$ such that \begin{enumerate}
\item \label{cond:2.quad.ext.split} For $i=1,2$, $Q_i=P_i\cap E=R_i\cap E$ and $Q_i$ splits in both $F$ and $L$.
\item \label{cond:2.quad.ext.principal} $P_1,P_2$ are principal.
\item \label{cond:2.quad.ext.class} $R_1 \cdot R_2=I^{M_0}$ as elements in $\Cl_L$.
\item \label{cond:2.quad.ext.T} For every $\ell \in T$, $\val_\ell(|\mathbb{G}_m(O_E/Q_1)|),\val_\ell(|\mathbb{G}_m(O_E/Q_2)|) \leq \val_\ell(M_1)$.
\item \label{cond:2.quad.ext.index} For $i=1,2$, $\Big[ \mathbb{G}_m(O_F/P_i): \langle \gamma \text{ (mod $P_i$)} \rangle \Big]$ divides $M_2$.
\end{enumerate} 
\end{lemma} 

\begin{proof} Denote the Hilbert class field of $F$ by $H_F$. We show that the claim holds with $M_0=M_1=2[H_F: \mathbb{Q} ]$ and $M_2=(M_0!)^4$.

Given $L,I,T$, let $H_L$ be the Hilbert class field of $L$, let $\tau_I\in \Gal(H_L/L)$ be the element corresponding to $I$ under the isomorphism $\Gal(H_L/L)\cong \Cl_L$, let $\Delta_{F \cdot L}$ be the discriminant of $F \cdot L$ and let 
\[
A=\left( M_0! \right) ^3 \cdot \prod_{\ell \in T \cup \mathcal{P}^\mathbb{Q}(\Delta_{F \cdot L})} \ell.
\]

Since $[H_F \cdot L: \mathbb{Q} ] \leq M_0$, the conditions of Lemma \ref{lem:cyclotomic.moving} hold for the field $H_F \cdot L$ and the integer $A$. Thus, the set
\[
\Xi_1 := \left\{ \sigma \in \Gal(H_F \cdot L(\zeta_A) / H_F \cdot L) \mid \left( \forall \ell |A \right) \, \sigma(\zeta_{\ell^{\val_\ell(A)}}) \neq \zeta_{\ell^{\val_\ell(A)}} \right\} 
\]
satisfies that $\Xi_1 \cdot \Xi_1 \supseteq \,^{2M_0}\Gal(H_F \cdot L(\zeta_A) / H_F \cdot L)$. It follows that the set
\[
\Xi_2:= \left\{ \sigma \in \Gal(H_F \cdot H_L(\zeta_A) / H_F \cdot L) \mid \left( \forall \ell |A \right) \, \sigma(\zeta_{\ell^{\val_\ell(A)}}) \neq \zeta_{\ell^{\val_\ell(A)}} \right\}
\]
satisfies that $\Xi_2 \cdot \Xi_2 \supseteq \,^{2M_0}\Gal(H_F \cdot H_L(\zeta_A) / H_F \cdot L)$. Thus there are two elements $\sigma_1,\sigma_2\in \Xi_2$ such that $(\sigma_1 \sigma_2)\restriction_{H_L}=\tau_I^{2M_0}$.

Since $[H_F \cdot L:F] \leq M_0$, there is a primitive unit $\delta \in \mathbb{G}_m(O_{H_F \cdot L})$ such that $\langle \gamma \rangle \supseteq \langle \delta^{M_0!} \rangle$. Let $\sigma \in \Xi_2$. We show that the condition of Theorem \ref{thm:Lenstra.GRH} holds for the extension $H_F \cdot H_L(\zeta_A)/H_F \cdot L$, the element $\sigma$, the element $\delta$, and the integer $(M_0!)^3$. Indeed, let $\ell$ be a prime number. If $\ell \leq M_0$, then $\val_\ell(A) \leq \val_\ell((M_0!)^3)+1$ and since $\sigma$ moves $\zeta_{\ell^{\val_\ell(A)}}$, it also moves $\zeta_{\ell^{\val_\ell((M_0!)^3)+1}}$. If $\ell > M_0$ and belongs to $T \cup \mathcal{P}^\mathbb{Q}(\Delta_{F \cdot L})$, then $\val_\ell((M_0!)^3)+1=1$ and $\sigma$ moves $\zeta_{\ell}$. Finally, if $\ell > M_0$ and does not belong to $T\cup \mathcal{P}_\mathbb{Q}(\Delta_{F \cdot L})$, then $\ell$ does not ramify in $H_F \cdot H_L(\zeta_A)$, so $\zeta_\ell \notin H_F \cdot H_L(\zeta_A)$.

By Theorem \ref{thm:Lenstra.GRH}, there are primes $S_1,S_2\in \mathcal{P}^{H_F \cdot L}$ such that $\sigma_i\in \left( \frac{H_F \cdot H_L(\zeta_A)/H_F \cdot L}{S_i} \right)$, such that $\Big[ \mathbb{G}_m(O_{H_F \cdot L}/S_i) : \left \langle \delta \text{ (mod $S_i$)} \right \rangle \Big]$ divides $(M_0!)^3$, and such that $S_i\cap E$ split completely in $H_F \cdot L$. Let $Q_i=S_i \cap E$, $P_i=S_i\cap F$, and $R_i=S_i \cap L$.

Claim \eqref{cond:2.quad.ext.split} clearly holds. Claim \eqref{cond:2.quad.ext.principal} hold since 
\[
\left( \frac{H_F/F}{P_i} \right) = \left( \frac{H_F \cdot H_L(\zeta_A) / F}{P_i} \right) \Big |_{H_F} = \left( \frac{H_F \cdot H_L(\zeta_A)/H_F \cdot L}{S_i} \right) \Big|_{H_F}=\sigma_i |_{H_F}=1.
\]
Similarly,
\[
\left( \frac{H_L/L}{R_1} \right) \cdot \left( \frac{H_L/L}{R_2} \right) = (\sigma_1 \sigma_2)\restriction_{H_L}=\tau_I^{2M_0},
\]
proving Claim \eqref{cond:2.quad.ext.class}. If $\ell \in T$ then 
\[
\left( \frac{H_F \cdot H_L(\zeta_A)/E}{Q_i} \right) \left( \zeta_{\ell^{\val_\ell(M_0)+1}} \right) = \sigma_i(\zeta_{\ell^{\val_\ell(M_0)+1}}) \neq \zeta_{\ell^{\val_\ell(M_0)+1}},
\]
so $\zeta_{\ell^{\val_\ell(M_0)+1}}\notin O_E/Q_i$, so $\val_\ell(|O_E / Q_i|)\leq \val_{M_0}+1$, proving Claim \eqref{cond:2.quad.ext.T}. Finally, since $O_{H_F \cdot L}/S_i \cong O_F/P_i$, Claim \eqref{cond:2.quad.ext.index} holds. 

\end{proof}

\begin{lemma} \label{lem:approximation.divisible} Assume the GRH. There are natural numbers $N_1,N_2$ such that for every two elements $a_1\in \mathbb{G}_m(\mathbf{A})$, $a_2\in \mathbb{G}_m(\mathbf{O})[2]$ and every two ideals $I,J\in \mathcal{I}^K$ there are \begin{enumerate}
\item Prime ideals $Q_1,Q_2\in \mathcal{P}^K \smallsetminus \mathcal{P} ^K(IJ)$.
\item An element $t\in O_+$, a trivialization $\rho$ of $T_t$, and elements $\alpha_1\in T_t(K)$, $\alpha_2\in T_t(O)$.
\item An element $s\in O_+$, a trivialization $\tau$ of $T_s$, and elements $\beta_1^{(1)},\beta_1^{(2)}\in T_s(K)$, $\beta_2^{(1)},\beta_2^{(2)}\in T_s(O)$.
\end{enumerate} 
such that the following hold \begin{enumerate}
\item \label{cond:approximation.divisible.alpha1} $\supp(\alpha_1) \subseteq \mathcal{P}_K(I)\cup \left\{ Q_1,Q_2 \right\}$, $\left( a_1^{N_1} \rho(\alpha_1) ^{-1} \right)_{\mathcal{P}_K(I)}\in \mathbb{G}_m(\mathbf{O})$, and $\val_{Q_1}(\rho(\alpha_1))=\val_{Q_2}(\rho(\alpha_1))=-1$.
\item \label{cond:approximation.divisible.alpha2} $\left( a_2 \rho(\alpha_2) ^{-1} \right)_{\mathcal{P}_K(I)}\in \mathbb{G}_m(\mathbf{O})[I]$.
\item \label{cond:approximation.divisible.beta1} $\supp(\beta_1^{(i)})=\left\{ Q_i \right\}$ and $\val_{Q_i}( \tau(\beta_1^{(i)}) )=-1$.
\item \label{cond:approximation.divisible.beta2} $\left( \rho(\alpha_2)^{N_2} \tau(\beta_2^{(i)}) ^{-1} \right)_{Q_i}\in \mathbb{G}_m(\mathbf{O})[Q_i]$ and $\left( \tau(\beta_2^{(i)}) \right)_{\mathcal{P}_K(J)}\in \mathbb{G}_m(\mathbf{O})[J^2]$.
\end{enumerate} 
\end{lemma} 

\begin{proof} We first claim that there is an element $s\in O_+$ such that the torsion-free rank of $\mathbb{G}_m(O_{K(\sqrt{-s})})$ is greater than the torsion-free rank of $\mathbb{G}_m(O)$. Indeed, if $S$ contains a non-real place, Dirichlet's unit theorem implies this for every $s\in O_+$. If this is not the case, then $K$ is totally real and $S$ is the set of real valuations, so $S\neq S_{def}$. Taking $v\in S \smallsetminus S_{def}$ and $s\in O_+$ such that $v(s)<0$, the claim follows from Dirichlet's unit theorem. It follows that there is a primitive element $\gamma \in O_{K(\sqrt{-s})}$ such that $N_{K(\sqrt{-s})/K}(\gamma)=1$. Let $M_0,M_1,M_2$ be the integers obtained by applying Lemma \ref{lem:two.quadratic.extensions} to the extension $K(\sqrt{-s})/K$ and $\gamma$. We will show that the claim of the lemma holds for $N_1=M_0$ and $N_2=M_1M_2$.

Let $a_1\in \mathbb{G}_m(\mathbf{A})$, $a_2\in \mathbb{G}_m(\mathbf{O})$, and $I\in \mathcal{I}^K$. By Lemma \ref{lem:a2.approximation.basic}, there is an element $t\in O_+$, a trivialization $\rho$ of $T_t$, and an element $\alpha_2$ such that Claim \eqref{cond:approximation.divisible.alpha2} holds.

By Lemma \ref{lem:two.quadratic.extensions}, there are prime ideals $Q_1,Q_2\in \mathcal{P}^K$, $P_1,P_2\in \mathcal{P}^{K(\sqrt{-s})}$, and $R_1,R_2\in \mathcal{P}^{K(\sqrt{-t})}$ such that \begin{enumerate}[label=(\alph*)]
\item \label{cond:approximation.divisible.split} For $i=1,2$, $Q_i=P_i\cap K=R_i\cap K$ and $Q_i$ splits in both $K(\sqrt{-s})$ and $K(\sqrt{-t})$.
\item \label{cond:approximation.divisible.P} $P_1,P_2$ are principal.
\item \label{cond:approximation.divisible.R} The fractional ideal $R_1 ^{-1} \cdot R_2 ^{-1} \cdot \prod_{Q\in \mathcal{S}(T_t)} (Q^\rho)^{M_0}$ is principal.
\item \label{cond:approximation.divisible.val} If $\ell$ is a prime number dividing $| \mathbb{G}_m(O_K/J^2) |$ then $\val_\ell(|\mathbb{G}_m(O_K/Q_1)|),\val_\ell(|\mathbb{G}_m(O_K/Q_2)|) \leq \val_\ell(M_1)$.
\item \label{cond:approximation.divisible.index} For $i=1,2$, $\Big[ \mathbb{G}_m(O_{K(\sqrt{-s})}/P_i): \langle \gamma \text{ (mod $P_i$)} \rangle \Big]$ divides $M_2$.
\end{enumerate}

By Claims \eqref{cond:approximation.divisible.split}, \eqref{cond:approximation.divisible.R}, and Lemma \ref{lem:a1.approximation.basic}, there is an element $\alpha_1\in T_t(K)$ such that Claim \eqref{cond:approximation.divisible.alpha1} holds. Similarly, using Claim \eqref{cond:approximation.divisible.P}, there are elements $\beta_1^{(i)}\in T_s(K)$ such that Claim \eqref{cond:approximation.divisible.beta1} holds.

Claim \eqref{cond:approximation.divisible.val} implies that 
\[
\gamma ^{M_1} \text{ (mod $P_i$)} \in \left \langle \gamma ^{| \mathbb{G}_m(O_K/J^2)|} \text{ (mod $P_i$)} \right \rangle.
\]
By Claim \eqref{cond:approximation.divisible.index},
\[
\Big[ \mathbb{G}_m(O_{K(\sqrt{-s})}/P_i): \left \langle \gamma^{| \mathbb{G}_m(O_K/J^2)|} \text{ (mod $P_i$)} \right \rangle \Big]\text{ divides $N_2$},
\]
and one can choose $\beta_2^{(i)}\in \left \langle \gamma ^{| \mathbb{G}_m(O_K/J^2)|} \right \rangle$ such that Claim \eqref{cond:approximation.divisible.beta2} holds.
\end{proof} 

\subsection{Approximation in $\mathbb{G}_m \times \mathbb{G}_m$ III}


For a natural number $n$ and a field $F$, denote the group of $n$-th roots of unity in $F$ by $\mu_n(F)$.

\begin{lemma} \label{lem:divisible.no.roots.of.1} Let $\ell$ be a prime number such that $\mu_\ell(K)=1$, let $I,J\in \mathcal{I}^K$, and let $a_1\in \mathbb{G}_m(\mathbf{A})$, $a_2\in \mathbb{G}_m(\mathbf{O})[2]$. Assume that $\mathcal{P} ^K(I) \cap \mathcal{P} ^K(\ell J)=\emptyset$ and that $\supp(a_1) \subseteq \mathcal{P} ^K(I)$. Then there are \begin{enumerate}
\item Prime ideals $Q_1,Q_2\in \mathcal{P} ^{K} \smallsetminus \mathcal{P} ^K(IJ\ell)$.
\item An element $t\in O_+$, a trivialization $\rho$ of $T_t$, and elements $\alpha_1\in T_t(K)$, $\alpha_2\in T_t(O)$.
\item An element $\beta \in T_t(\mathbf{O})$.
\end{enumerate}
such that \begin{enumerate}
\item \label{lem:divisible.no.roots.of.1.alpha.1} $\supp(\alpha_1) \subseteq \mathcal{P}^{K}(I) \cup \left\{ Q_1,Q_2 \right\}$, $\left( a_1 \rho(\alpha_1) ^{-1} \right)_{\mathcal{P}^{K}(I)} \in \mathbb{G}_m(\mathbf{O})$, $\val_{Q_1}(\rho(\alpha_1))=\val_{Q_2}(\rho(\alpha_1))=-1$, and $\left( \rho(\alpha_1) \right)_{\mathcal{P} ^K(J)}\in \mathbb{G}_m(\mathbf{O})[J]$.
\item \label{lem:divisible.no.roots.of.1.alpha.2} $\left( a_2 \rho(\alpha_2) ^{-1} \right)_{\mathcal{P} ^K(I)} \in \mathbb{G}_m(\mathbf{O})[I]$.
\item \label{lem:divisible.no.roots.of.1.beta} For $i=1,2$, $\rho(\alpha_2)_{\left\{ Q_i \right\}} \beta ^{-\ell}\in \mathbb{G}_m(\mathbf{O})[Q_i]$.
\end{enumerate} 
\end{lemma} 

\begin{proof} By Lemma \ref{lem:a2.approximation.basic}, there is an element $t\in O_+$, a trivialization $\rho$ of $T_t$, and an element $\alpha_2\in T_t(O)$ such that $\mathcal{P} ^K(IJ) \subseteq \mathcal{S}(T_t)$, such that Claim \ref{lem:divisible.no.roots.of.1.alpha.2} holds, and such that $\zeta_3\notin K(\sqrt{-t}) \smallsetminus K$. If $\ell=3$, the assumption $\zeta_\ell\notin K$ implies that $\zeta_3\notin K(\sqrt{-t})$. If $\ell>3$, then clearly $\zeta_\ell \notin K(\sqrt{-t})$.

Let $R_t$ be the ray class field of $K(\sqrt{-t})$ with respect to $J$. Since $\zeta_\ell \notin K(\sqrt{-t})$, the set
\[
\Omega := \left\{ \sigma \in \Gal \left( R_t(\zeta_\ell) / K(\sqrt{-t}) \right) \mid \sigma(\zeta_\ell)\neq \zeta_\ell \right\} 
\]
is the complement of a proper subgroup. Thus, $\Omega \cup \Omega ^2=\Gal \left( H_t(\zeta_\ell) / K(\sqrt{-t}) \right)$.

Let $\sigma \in \Gal(H_t/K(\sqrt{-t}))\cong \Cl_{K(\sqrt{-t})}$ be the element corresponding to $\prod_{Q\in \mathcal{P}^K(I)} \left( Q^ \rho \right) ^{\val_{Q}(a_1)}$. Then either there is  an element $\omega \in \Omega$ with $\omega |_{H_t}=\sigma$ or there are two elements $\omega_1,\omega_2$ such that $\omega_1|_{H_t} \cdot \omega_2|_{H_t}=\sigma$. We will deal with the second case. The first case is analogous and the proof will give $Q_1=Q_2$.

By Chebotarev's density theorem, there are primes $P_1,P_2\in \mathcal{P} ^{K(\sqrt{-t})} \smallsetminus \mathcal{P}^{K(\sqrt{-t})}(IJ\ell)$ such that \begin{enumerate}
\item $\left( \frac{H_t(\zeta_\ell)/K(\sqrt{-t})}{P_1} \right),\left( \frac{H_t(\zeta_\ell)/K(\sqrt{-t})}{P_2} \right) \in \Omega$.
\item $Q_1:=K\cap P_1$ and $Q_1:=K\cap P_2$ split in $K(\sqrt{-t})$.
\item $P_1 \cdot P_2 \cdot \prod_{Q\in \mathcal{P}^K(I)} \left( Q^ \rho \right) ^{\val_{Q}(a_1)}$ is principal.
\end{enumerate} 

By Lemma \ref{lem:a1.approximation.basic}, there is an element $\alpha_1\in T_t(K)$ for which Claim \ref{lem:divisible.no.roots.of.1.alpha.1} holds. Finally, since $Q_i$ split, $\mathbb{G}_m(O/Q_i) \cong \mathbb{G}_m(O_{K(\sqrt{-t})}/P_i)$. By the definition of $\Omega$, the group $\mathbb{G}_m(O_{K(\sqrt{-t})}/P_i)$ is $\ell$-divisible, so Claim \ref{lem:divisible.no.roots.of.1.beta} holds.
\end{proof} 

If $L$ is a number field, $P\in \mathcal{P} ^L$ is a prime ideal, and $\ell$ is a prime number such that $\mu_\ell(L)\neq 1$, let $n_{\ell,P}: \mathbb{G}_m(L_P)\times \mathbb{G}_m(L_P) \rightarrow \mu_\ell(L)$ be the local $\ell$th norm residue symbol (see e.g. \cite[\S15]{Mil71}). We will use the following properties of $n_{\ell,P}$: \begin{enumerate}
\item $n_{\ell,P}$ is anti symmetric and bimultiplicative.
\item $n_{\ell,P}(\mathbb{G}_m(L_P) \times \mathbb{G}_m(O_P))=\mu_\ell(L)$.
\item If $P\not|\; \ell$, then $n_{\ell,P}$ is trivial on $\mathbb{G}_m(O_P) \times \mathbb{G}_m(O_P)$.
\item If $P\not|\; \ell$, $\pi\in O_P$ is a uniformizer, and $a\in \mathbb{G}_m(O_P)$, then 
\[
n_{\ell,P}(\pi,a)=1 \iff \text{ $a$ is an $\ell$th power in $\mathbb{G}_m(O_P)$}.
\]
\item There is an open set $1\in U_P \subseteq \mathbb{G}_m(O_P)$ such that $n_{\ell,P} \left( U_P \times \mathbb{G}_m(K_P) \right)=1$.
\item If $a,b\in \mathbb{G}_m(L)$, then $n_{\ell,P}(a,b)=1$ for all but finitely many prime ideals $P$ and
\begin{equation} \label{eq:reciprocity.norm.residue}
\prod_{P\in \mathcal{P} ^L} n_{\ell,P}(a,b)=1.
\end{equation}
\end{enumerate} 

\begin{lemma} \label{lem:inverse.reciprocity} Let $J\in \mathcal{I}^K$, let $\ell$ be a prime number such that $\mu_\ell(K)\neq 1$, let $Q_1,Q_2\in \mathcal{P} ^K \smallsetminus \mathcal{P} ^K(J\ell)$ be two prime ideals, and let $a_1\in \mathbb{G}_m(\mathbf{A}),a_2\in \mathbb{G}_m(\mathbf{O})$ be such that $n_{\ell,Q_1}(a_1,a_2) \cdot n_{\ell,Q_2}(a_1,a_2)=1$. Then there are \begin{enumerate}
\item An element $t\in O_+$.
\item A trivialization $\rho$ of $T_t$ such that $\mathcal{P} ^K(J \ell Q_1 Q_2) \subseteq \mathcal{S}(T_t)$
\item Elements $\alpha_1\in T_t(K)$ and $\alpha_2\in T_t(O)$.
\item A prime $Q_3\in \mathcal{P} ^K \smallsetminus \mathcal{P} ^K(J\ell Q_1 Q_2)$.
\end{enumerate} 
such that \begin{enumerate}
\item \label{cond:inverse.reciprocity.alpha2} $\rho(\alpha_2) ^{-1} a_2 \in \mathbb{G}_m(\mathbf{O})[Q_1Q_2]$.
\item \label{cond:inverse.reciprocity.alpha1} $\supp(\rho(\alpha_1)) \subseteq \left\{ Q_1,Q_2,Q_3 \right\}$, $\val_{Q_1}(\rho(\alpha_1))=\val_{Q_1}(a_1)$, $\val_{Q_2}(\rho(\alpha_1))=\val_{Q_2}(a_1)$, $\val_{Q_3}(\rho(\alpha_1))=-1$, and $\rho(\alpha_1)_{\mathcal{P} ^K(J)}\in \mathbb{G}_m(\mathbf{O})[J]$.
\item \label{cond:inverse.reciprocity.Q3} $n_{\ell,Q_3}(\rho(\alpha_1),\rho(\alpha_2))=1$.
\end{enumerate} 
\end{lemma} 

\begin{proof} Without loss of generality, we can assume that $a_2\in \mathbb{G}_m(\mathbf{O})[J \ell]$. By Lemma \ref{lem:a2.approximation.basic}, there are $t\in O_+$, a trivialization $\rho$ of $T_t$ such that $\mathcal{P} ^K(J\ell Q_1Q_2) \subseteq \mathcal{S}(T_t)$ and an element $\alpha_2\in T_t(O)$ such that $\rho(\alpha_2) ^{-1} a_2\in \mathbb{G}_m(\mathbf{O})[\ell Q_1 Q_2]$. In particular, Claim \eqref{cond:inverse.reciprocity.alpha2} holds.

Choose $m$ such that, for every prime $Q\in \mathcal{P} ^{K}$ that divides $\ell$, the norm residue symbol $n_{\ell,Q}(a,b)$ vanishes if $a\equiv 1\text{ (mod $Q^m$)}$. Let $F/K(\sqrt{-t})$ be the ray class field of modulus $J \cdot \ell^m$ of $K(\sqrt{-t})$. Since $Q_1,Q_2$ are prime to $J \cdot \ell$, the fractional ideal $(Q_1 ^ \rho )^{\val_{Q_1}(a_1)} \cdot (Q_2 ^ \rho)^{\val_{Q_2}(a_1)}$ defines an element in the ray class group $\Cl^{J \cdot \ell^m}(K(\sqrt{-t}))$. By Chebotarev, there is a prime ideal $P_3\in \mathcal{P} ^{K(\sqrt{-t})} \smallsetminus \mathcal{P} ^{K(\sqrt{-t})}(J\ell Q_1Q_2)$ and an element $\beta \in \mathbb{G}_m(K(\sqrt{-t}))$ such that $P_3 \cap K$ splits in $K(\sqrt{-t})$, $\beta \equiv 1\text{ (mod $J \cdot \ell^m$)}$, and $(Q_1 ^ \rho )^{\val_{Q_1}(a_1)} \cdot (Q_2 ^ \rho)^{\val_{Q_2}(a_1)}=P_3 \cdot (\beta)$. Let $\alpha_1=\frac{\beta}{\sigma(\beta)}$ and $Q_3=P_3\cap K$. It is clear that Claim \eqref{cond:inverse.reciprocity.alpha1} holds.

From the product formula \eqref{eq:reciprocity.norm.residue},
\[
1=\prod_{P\in \mathcal{P}^{K(\sqrt{-t})}} n_{\ell,P}(\alpha_1,\alpha_2)=
\]
\[
=\prod_{P\in \mathcal{P}^{K(\sqrt{-t})}(\ell)} n_{\ell,P}(\alpha_1,\alpha_2) \cdot \prod_{P\in \left\{ P_1,P_2,P_3 \right\}} n_{\ell,P}(\alpha_1,\alpha_2) \cdot \prod_{P\in \mathcal{P}^{K(\sqrt{-t})} \smallsetminus \mathcal{P} ^{K(\sqrt{-t})}(\ell P_1P_2P_3)} n_{\ell,P}(\alpha_1,\alpha_2).
\]
Let $P\in \mathcal{P} ^{K(\sqrt{-t})}(\ell)$. By assumption, $\beta \equiv 1\text{ (mod $P^m$)}$ and $\beta \equiv 1\text{ (mod $\sigma(P)^m$)}$. Thus, $\alpha_1\equiv 1\text{ (mod $P^m$)}$ and $n_{\ell,P}(\alpha_1,\alpha_2)=1$. Thus, the first product is trivial. 

If $P\in \mathcal{P}^{K(\sqrt{-t})} \smallsetminus \mathcal{P} ^{K(\sqrt{-t})}(\ell P_1P_2P_3)$ then $(\alpha_1)_
P, (\alpha_2)_P\in \mathbb{G}_m(O_P)$, so the third product is trivial. 

It follows that
\[
1=n_{\ell,P_1}(\alpha_1,\alpha_2) \cdot n_{\ell,P_2}(\alpha_1,\alpha_2) \cdot n_{\ell,P_3}(\alpha_1,\alpha_2)=
\]
\[
=n_{\ell,Q_1}(\rho(\alpha_1),\rho(\alpha_2)) \cdot n_{\ell,Q_2}(\rho(\alpha_1),\rho(\alpha_2)) \cdot n_{\ell,Q_3}(\rho(\alpha_1),\rho(\alpha_2))=
\]
\[
=n_{\ell,Q_1}(a_1,a_2) \cdot n_{\ell,Q_2}(a_1,a_2) \cdot n_{\ell,Q_3}(\rho(\alpha_1),\rho(\alpha_2))=n_{\ell,Q_3}(\rho(\alpha_1),\rho(\alpha_2)),
\]
proving Claim \eqref{cond:inverse.reciprocity.Q3}.
\end{proof} 

\subsection{Approximation in the ultrapower}

We first extend the definitions in \S\ref{subsec:T_t.standard} to the non-standard model. All the definitions below are taken with respect to the fixed ultrafilter $\mathcal{N}$. \begin{enumerate}
\item Define $O_+^*$ to be the ultrapower of the semiring $O_+$. If $t\in O_+^*$, we choose a representative $t=[t_n]$, where $t_n\in O_+$, and define $T_t$ as the ultraproduct of the schemes $T_{t_n}$ (this ultraproduct exists because of Lemma \ref{lem:T_t.bounded.complexity}). This definition does not depend on the choice of representative. 
\item Given $t\in O_+^*$, let $\mathcal{S}(T_t)=\left\{ Q\in \mathcal{P} ^{K^*} \mid T_t \times \Spec(K^*_Q) \cong \mathbb{G}_m \right\}$. Choosing a representative $t=[t_n]$ with $t_n\in O_+$, we have $\mathcal{S}(T_t)=\left[ \mathcal{S}(T_{t_n})\right]$. 
\item A trivialization of $T_t$ is a morphism $\rho:T_t \times \Spec(\mathbf{A}^*) \rightarrow \mathbb{G}_m \times \Spec(\mathbf{A}^*)$ such that $\rho \times \Spec(\mathbf{A}^*_{\mathcal{S}(T_t)})$ is an isomorphism and $\rho \times \Spec(\mathbf{A}^*_{\mathcal{P} ^K \smallsetminus \mathcal{S}(T_t)})$ is the trivial homomorphism. If $t=[t_n]$ with $t_n\in O_+$ and $\rho_n$ are trivializations of $T_{t_n}$, then $\left[ \rho_n \right]$ is a trivialization of $T_t$.
\item If $\ell$ is a prime number such that $\mu_\ell(K)\neq1$, $P=[P_n]\in \mathcal{P} ^{K^*}$, and $a=[a_n],b=[b_n]\in K^*_P$, define $n_{\ell,P}(a,b)=[n_{\ell,P_n}(a_n,b_n)]\in [\mu_\ell(K)]=\mu_\ell(K^*)$.
\end{enumerate} 

Using the definitions above, the approximation lemmas \ref{lem:weak.approx.standard}, \ref{lem:approx.single.support.standard}, \ref{lem:approximation.divisible}, and \ref{lem:inverse.reciprocity} have obvious versions in the non-standard model obtained by changing $K$ to $K^*$, $O$ to $O^*$, etc. For example, the non-standard Lemma \ref{lem:weak.approx.standard} is that if $I\in \mathcal{I}^{K^*}$ and $a_1,a_2\in \mathbb{G}_m(\mathbf{A}^*)$ then there is an element $t\in O^*_+$, a trivialization $\rho$ of $T_t$, and elements $\alpha_1,\alpha_2\in T_t(K^*)$ such that \begin{enumerate}
\item $\mathcal{P} ^{K^*}(I) \subseteq \mathcal{S}(T_t)$ and $\left( a_i \rho(\alpha_i) ^{-1} \right)_{\mathcal{P}^{K^*}(I)} \in \mathbb{G}_m(\mathbf{O}^*)[I]$, for $i=1,2$.
\item The sets $\mathcal{P} ^{K^*}(I), \mathcal{P} ^{K^*}(tO^*), \supp(\alpha_1) \smallsetminus \mathcal{P} ^{K^*}(I),\supp(\alpha_2) \smallsetminus \mathcal{P} ^{K^*}(I)$ are pairwise disjoint.
\end{enumerate} 
The non-standard versions of Lemmas \ref{lem:weak.approx.standard}, \ref{lem:approx.single.support.standard}, \ref{lem:approximation.divisible}, and \ref{lem:inverse.reciprocity} all follow from the standard versions using coordinate-wise arguments.

\section{Split spin groups}

\begin{notation} Fix an integer $m \geq 3$ and let $f_s(x_1,\ldots,x_{2m})=\sum_{i=1}^m x_{2i-1}^2-x_{2i}^2$ be the split quadratic form of dimension $2m$.
\end{notation} 

\subsection{Generalized Steinberg symbols}

\begin{definition} Let $R$ be a ring in which 2 is invertible and let 
\[
1 \rightarrow D \rightarrow E \overset{\eta}{\rightarrow} \Spin_{f_s}(R) \rightarrow 1
\]
be a central extension. For every two commuting elements $g_1,g_2\in \Spin_{f_s}(R)$, choose $\widetilde{g_1}\in \eta ^{-1} (g_1)$ and $\widetilde{g_2}\in \eta ^{-1} (g_2)$ and denote $[g_1:g_2]:=[\widetilde{g_1},\widetilde{g_2}]$. The element $[g_1,g_2]\in D$ is independent of the choices of $\widetilde{g_i}$ and the map 
\[
[-:-]: \left\{ (g_1,g_2)\in \Spin_{f_s}(R)^2 \mid [g_1,g_2]=1 \right\} \rightarrow D
\]
is called the generalized Steinberg symbol.
\end{definition} 

\begin{lemma}\label{lemma:prop_stein_1}  Let $R$ be a ring in which 2 is invertible and let 
$$
1 \rightarrow D \rightarrow E \overset{\eta}\rightarrow \Spin{f_s}(R) \rightarrow 1
$$  
be a central extension. 

\begin{enumerate}
	\item\label{item:s00} For every $g\in \Spin_{f_s}(R)$, $[g:1]=[1:g]=1$.
	\item\label{item:s01} For every commuting $g_1,g_2\in \Spin_{f_s}(R)$, $[g_1:g_2]=[g_2:g_1]^{-1}$. 
	\item\label{item:s02}  For every $g_1,g_2,g_3 \in \Spin_{f_s}(R)$, if $g_1$ commutes with $g_2$ and $g_3$ then $[g_1:g_2g_3]=[g_1:g_2][g_1:g_3]$. 
	\item\label{item:s03}  For every $g_1,g_2,g_3 \in \Spin_{f_s}(R)$, if $g_1$ and $g_2$ commute with $g_3$  then $[g_1g_2:g_3]=[g_1:g_3][g_2:g_3]$. 
	\item\label{item:s03.5} For every commuting $g_1,g_2\in \Spin_{f_s}(R)$, $[g_1:g_2]^{-1}=[g_1^{-1}:g_2]=[g_1:g_2^{-1}]$.
	\item\label{item:s04}  For every commuting $g_1,g_2 \in \Spin_{f_s}(R)$ and every $l \in \Spin_{f_s}(R)$, $[g_1:g_2]=[lg_1l^{-1}:lg_2l^{-1}]$. 
	\item \label{item:s06} If $R$ is a topological ring and the extension is a topological  central extension then the generalized Steinberg symbol $[\cdot:\cdot]$ is continuous. 
\end{enumerate}
\end{lemma}

\begin{proof}
\begin{enumerate}
\item  Let $\widetilde{g}$ be a lift of $g$ to $E$. Then  $[g:1]=[\widetilde{g},1]=1 $ and similarly, $[1:g]=1$.
\item  For every group elements $x$ and $y$, $[x,y]=[y,x]^{-1}$.	
\item For every group elements $x$, $y$ and $z$, $[x,yz]=[x,y]y[x,z]y^{-1}$. The claim follows from the centrality of $D$. 
\item For every group elements $x$, $y$ and $z$, $[xy,z]=x[y,z]x^{-1}[x,z]$. The claim follows from the centrality of $D$. 
\item Follows from Items \ref{item:s00}, \ref{item:s02} and \ref{item:s03}.
\item The claim follows from the centrality of $D$. 
\item Let $g_1,g_2 \in \Spin_{f_s}$ be commutating elements. Fix lifts $\tilde{g}_1,\tilde{g_2}$ of $g_1,g_2$, respectively, and let $U \subseteq D$ be an open neighborhood of $[\tilde{g}_1,\tilde{g_2}]=[g_1:g_2]$. By the continuity of multiplication, there exist open subsets $U_1\ni \tilde{g}_1$ and $U_2 \ni \tilde{g}_2$ of $E$ such that $\{[l_1,l_2]\mid l_1 \in U_1 \land \ l_2 \in U_2 \}\subseteq U$. 
	  Since the extension is topological, $\eta$ is open and the induced isomorphism $\bar{\eta}:E/D \rightarrow \Spin_{f_s}(R)$ is a homeomorphism. Thus, every $g_3 \in \eta(U_1)$ and every $g_4 \in \eta(U_2)$  have lifts $\tilde{g}_3$ and $\tilde{g}_4$ such that $\tilde{g}_3\in U_1$ and $\tilde{g}_4 \in U_2$. If $g_3,g_4$ commute, then $[g_3:g_4]=[\tilde{g}_3,\tilde{g}_4]\in U$. 
\end{enumerate}
\end{proof}

\subsection{The coroots $h_1,h_2$ and a theorem of Matsumoto}

Using the embedding $\Spin_{f_s} \hookrightarrow \Cliff_{f_s}$, the Lie algebra of $\Spin_{f_s}$, which we denote by $\mathfrak{spin}_{f_s}$, is the span of the elements $e_ie_j$, where $i<j$. The subspace $\mathfrak{t}$ spanned by $e_1e_2,e_3e_4,\ldots,e_{2m-1}e_{2m}$ is a Cartan algebra in $\mathfrak{spin}_{f_s}$. The elements
\[
X=e_1e_3+e_2e_3-e_1e_4-e_2e_4 \text{ and } Y=e_3e_5+e_4e_5-e_3e_6-e_4e_6
\]
are root vectors: 
\[
\left[ e_1e_2,X \right]=-2X \quad\quad \left[ e_3e_4,X \right]=2X \quad\quad \left[e_{2k-1,2k},X\right]=0,
\]
for all $k\neq 1,2$, and
\[
\left[ e_3e_4,Y \right]=-2Y \quad\quad \left[ e_5e_6,Y \right]=2Y \quad\quad \left[e_{2l-1,2l},Y\right]=0,
\]
for all $l \neq 2,3$. We denote the root corresponding to $X$ by $\alpha_1$ and the root corresponding to $Y$ by $\alpha_2$.

Let $T=\exp \mathfrak{t}$ be the corresponding maximal torus. The map 
\[
(t_1,\ldots,t_n) \mapsto \prod_i \left( \frac{t_i+t_i ^{-1}}{2}+\frac{t_i-t_i ^{-1}}{2}e_{2i-1}e_{2i} \right)
\]
is an isogeny between $\mathbb{G}_m^m$ and $T$ whose kernel is the subgroup $\left\{ v\in \left\{ \pm1 \right\} ^m \mid v_1 \cdots v_m=1 \right\}$. We have 
\[
Ad \left( e^{t e_1e_2} \right) X= \exp \left( t \cdot ad(e_1e_2) \right) X= e^{-2t}X.
\]
and 
\[
Ad \left( e^{t e_3e_4} \right) X= \exp \left( t \cdot ad(e_3e_4) \right) X= e^{2t}X.
\]

Since $e^{t e_1e_2}=\left( \frac{e^t+e^{-t}}{2} \right) + \left( \frac{e^t-e^{-t}}{2} \right) e_1e_2$ and $e^{t e_3e_4}=\left( \frac{e^t+e^{-t}}{2} \right) + \left( \frac{e^t-e^{-t}}{2} \right) e_3e_4$, we get that
\[
\alpha_1 \left( \prod_i \left( \frac{t_i+t_i ^{-1}}{2}+\frac{t_i-t_i ^{-1}}{2}e_{2i-1}e_{2i} \right) \right) =\left( \frac{t_2}{t_1} \right) ^2.
\]

The coroot $h_1$ corresponding to $\alpha_1$ is the map $h_1:\mathbb{G}_m \rightarrow T$ such that, for any $O$-algebra $R$ and any $t\in R^\times$,
\[
h_1(t)=\frac{t+2+t ^{-1}}{4}+\frac{t-t ^{-1}}{4}e_1e_2-\frac{t-t ^{-1}}{4}e_3e_4-\frac{t-2+t ^{-1}}{4}e_1e_2e_3e_4=
\]
\[
=\left( \frac{s+s ^{-1}}{2}+\frac{s-s ^{-1}}{2}e_1e_2 \right) \cdot \left( \frac{s+s ^{-1}}{2}-\frac{s-s ^{-1}}{2}e_3e_4 \right),
\]
where $s$ is a square root of $t$ (possibly in a quadratic extension). Similarly,
\[
h_2(t)=\frac{t+2+t ^{-1}}{4}+\frac{t-t ^{-1}}{4}e_3e_4-\frac{t-t ^{-1}}{4}e_5e_6-\frac{t-2+t ^{-1}}{4}e_3e_4e_5e_6=
\]
\[
=\left( \frac{s+s ^{-1}}{2}+\frac{s-s ^{-1}}{2}e_3e_4 \right) \cdot \left( \frac{s+s ^{-1}}{2}-\frac{s-s ^{-1}}{2}e_5e_6 \right)
\]
is the coroot corresponding to $\alpha_2$.\\

\begin{lemma} \label{lemma:prop_stein_2} Let $F$ be a field of characteristic different from 2 is invertible and let 
\[
1 \rightarrow D \rightarrow E \overset{\eta}\rightarrow \Spin{f_s}(F) \rightarrow 1
\]
be a central extension. If $\alpha,1-\alpha \in F^\times$ then $[h_1(\alpha,\alpha ^{-1}),h_2(1-\alpha,(1-\alpha) ^{-1})]=1$.
\end{lemma} 

\begin{proof} This a standard property of Steinberg symbols, see \cite[Lemma 5.5]{Mat69}.
\end{proof} 

The following is a special case of \cite[Lemma 5.4 and Theorem 8.2]{Mat69}

\begin{theorem} \label{thm:Matsumoto} Let $F$ be a field of characteristic different from 2 and let
\begin{equation} \label{eq:Matsumoto}
1 \rightarrow D \rightarrow E \overset{\eta}{\rightarrow} \Spin_{f_s}(F) \rightarrow 1
\end{equation}
be a central extension. The extension \eqref{eq:Matsumoto} splits if and only if, for every $a_1,a_2\in F^ \times$, $[h_1(a_1):h_2(a_2)]=1$.
\end{theorem} 

\subsection{Steinberg symbols for $\Spin_{f_s}(\mathbf{A} ^*)$}

\begin{lemma} \label{lem:derived.congruence} For every $I\in \mathcal{I} ^{K^*}$, the derived subgroup of $\Spin_{f_s}(\mathbf{O} ^*)[I]$ is $\Spin_{f_s}(\mathbf{O} ^*)[I^2]$.
\end{lemma} 

\begin{proof} By a coordinate-wise argument, it is enough to show that there is a natural number $N$ such that, for every $I\in \mathcal{I} ^K$, 
\[
\Spin_{f_s}(\mathbf{O})[I^2]=\left[ \Spin_{f_s}(\mathbf{O})[I],\Spin_{f_s}(\mathbf{O})[I] \right] ^N.
\]
For this, it is enough to show that, for every $P\in \mathcal{P} ^K$ and $k\in \mathbb{N}$,
\[
\Spin_{f_s}(\mathbf{O})[P^{2k}]=\left[ \Spin_{f_s}(\mathbf{O})[P^k],\Spin_{f_s}(\mathbf{O})[P^k] \right] ^N.
\]
This follows from \cite[Lemma 3.22]{AM22}.
\end{proof} 

\begin{lemma} \label{lem:simple.A*} Let $\mathcal{Q} \subseteq \mathcal{P} ^{K^*}$ be an internal set and let $g\in \Spin_{f_s}(\mathbf{A} ^*)$. Assume that $g_P\notin Z \left( \Spin_{f_s}(\mathbf{A}) \right)$, for every $P\in \mathcal{Q}$. Then $\gcl_{\Spin_{f_s}(\mathbf{A} ^*)}(g_\mathcal{Q})$ generates $\Spin_{f_s}(\mathbf{A})_\mathcal{Q}$.
\end{lemma} 

\begin{proof} By a coordinate-wise argument, it is enough to show that there is a natural number $N$ such that, for every $P\in \mathcal{P} ^K$, \begin{enumerate}
\item If $g\in \Spin_{f_s}(K_P) \smallsetminus Z \left( \Spin_{f_s}(K_P) \right)$, then $\left\langle \gcl_{\Spin_{f_s}(K_P)}(g) \right\rangle = \Spin_{f_s}(K_P)$ and $\wid \left( \gcl_{\Spin_{f_s}(K_P)}(g) \right) < N$.
\item If $g\in \Spin_{f_s}(O_P) \smallsetminus Z \left( \Spin_{f_s}(O_P) \right) \cdot \Spin_{f_s}(O_P)[P]$, then $\left\langle \gcl_{\Spin_{f_s}(O_P)}(g) \right\rangle = \Spin_{f_s}(O_P)$ and $\wid \left( \gcl_{\Spin_{f_s}(O_P)}(g) \right) <N$.
\end{enumerate}  

Let $g\in \Spin_{f_s}(K_P) \smallsetminus Z \left( \Spin_{f_s}(K_P) \right)$. By \cite[Lemma 3.16]{AM22}, $\gcl_{\Spin_{f_s}(K_P)}(g)^{2m^2}$ contains a non-trivial transvection, so $\gcl_{\Spin_{f_s}(K_P)}(g)^{8m^3}=\Spin_{f_s}(K_P)$, proving the first claim.

The second claim follows from \cite[Lemma 3.22]{AM22}.
\end{proof} 

\begin{definition} Let $1 \rightarrow D \rightarrow E \overset{\eta}\rightarrow \Spin{f_s}(\mathbf{A} ^*) \rightarrow 1$ be a central extension and let $\mathcal{Q} \subseteq \mathcal{P} ^{K^*}$ be an internal set. \begin{enumerate}
\item For $g_1,g_2\in \Spin_{f_s}(\mathbf{A}^*)$, denote $[g_1:g_2]_\mathcal{Q}:=[(g_1)_\mathcal{Q}:(g_2)_\mathcal{Q}]$.
\item For $a_1,a_2\in (\mathbf{A}^*) ^ \times$, denote $\{ a_1 : a_2 \}_\mathcal{Q}:=[h_1(a_1) : h_2(a_2)]_\mathcal{Q}$. If $\mathcal{Q}=\mathcal{P}^{K^*}$, we omit it from notation.
\end{enumerate}
\end{definition}

\begin{proposition} \label{prop:Steinberg.symbols} Let $1 \rightarrow D \rightarrow E \rightarrow \Spin_{f_s}(\mathbf{A} ^*) \rightarrow 1$ be a topological central extension with discrete center that has a continuous splitting over a congruence subgroup $\Spin_{f_s}(\mathbf{O} ^*)[J_0]$. Then \begin{enumerate}
\item \label{prop:Steinberg.symbols.disjoint.commute} Let $\mathcal{Q}_1,\mathcal{Q}_2 \subseteq \mathcal{P} ^{K^*}$ be disjoint internal sets. If $g_1\in \Spin_{f_s}(\mathbf{A}^*)_{\mathcal{Q}_1}$ and $g_2\in \Spin_{f_s}(\mathbf{A}^*)_{\mathcal{Q}_2}$, then $g_1$ and $g_2$ commute and $\left[ g_1 : g_2 \right]=1$.
\item \label{prop:Steinberg.symbols.disjoint.decompose} Let $\mathcal{Q}_1,\mathcal{Q}_2 \subseteq \mathcal{P} ^{K^*}$ be disjoint internal sets. If $g_1,g_2\in \Spin_{f_s}(\mathbf{A} ^*)$ commute, then $[g_1:g_2]_{\mathcal{Q}_1 \cup \mathcal{Q}_2}=[g_1:g_2]_{\mathcal{Q}_1} \cdot [g_1:g_2]_{\mathcal{Q}_2}$.
\item \label{prop:Steinberg.symbols.J_0} If $g_1\in \Spin_{f_s}(\mathbf{O}^*)$ and $g_2\in \Spin_{f_s}(\mathbf{O} ^*)[J_0^2]$ commute, then $[g_1:g_2]=1$.
\item \label{prop:Steinberg.symbols.continuity} Let $a_1\in \mathbb{G}_m(\mathbf{A}^*)$, let $a_2\in \mathbb{G}_m(\mathbf{O} ^*)$, and let $\mathcal{Q} \subseteq \mathcal{P}_{K^*}$ be an internal set. There is an ideal $I\in \mathcal{I}_{K^*}(\mathcal{Q})$ such that, for every $b_1,b_2\in \mathbb{G}_m(\mathbf{A}^*)$ such that $\left( a_1 b_1 ^{-1} \right)_{\mathcal{Q}}\in \mathbf{G}_m(\mathbf{O}^*)[J_0^2]$ and $\left( a_2 b_2 ^{-1} \right)_{\mathcal{Q}}\in \mathbf{G}_m(\mathbf{O}^*)[I]$, we have $\left\{ a_1 : a_2 \right\}_\mathcal{Q} = \left\{ b_1 : b_2 \right\}_\mathcal{Q}$.
\item \label{prop:Steinberg.symbols.val.-1} Let $P\in \mathcal{P}_{K^*} \smallsetminus \mathcal{P}_{K^*}(2J_0)$ and let $a_1,a_2,b_1,b_2\in \mathbb{G}_m(\mathbf{A}^*)$ such that $\val_P(a_1)=-1$, $a_1 b_1 ^{-1} \in \mathbb{G}_m(\mathbf{O} ^*)$, $(a_2)_P\in \mathbb{G}_m(\mathbf{O}^*)$, and $a_2 b_2 ^{-1} \in \mathbb{G}_m(\mathbf{O}^*)[P]$. Then $\left\{ a_1 : a_2 \right\}_P=\left\{ b_1 : b_2 \right\}_P$.
\end{enumerate} 
\end{proposition}

\begin{proof} \begin{enumerate}
\item It is clear that $\Spin_{f_s}(\mathbf{A} ^*)_{\mathcal{Q}_1}$ and $\Spin_{f_s}(\mathbf{A} ^*)_{\mathcal{Q}_2}$ commute. Choose $g_3',g_3''\in \Spin_{f_s}(\mathbf{A}^*)$ such that $[g_3',g_3'']_P\notin Z \left( \Spin_{f_s}(\mathbf{A} ^*)_P \right)$, for every $P\in \mathcal{Q}_2$. Denote $g_3=[g_3',g_3'']$. Since $D$ is abelian,
\[
[g_1:g_3]=[g_1:g_3'][g_1:g_3''][g_1:(g_3') ^{-1}][g_1:(g_3'') ^{-1}]=1.
\]
By Lemma \ref{lem:simple.A*}, $g_2=(h_1 g_3^{\epsilon_1} h_1 ^{-1})\cdots(h_N g_3^{\epsilon_N} h_N ^{-1})$, for some $N\in \mathbb{N}$, $h_i\in \Spin_{f_s}(\mathbf{A} ^*)_{\mathcal{Q}_2}$, and $\epsilon_i\in \left\{ \pm 1 \right\}$. Thus,
\[
[g_1:g_2]=\prod_1^N \left[ g_1:h_i g_3^{\epsilon_i} h_i ^{-1}\right]=\prod_1^N \left[ h_i g_1 h_i ^{-1}:h_i g_3^{\epsilon_i} h_i ^{-1}\right]=\prod_1^N [g_1:g_3 ^{\epsilon_i}]=1.
\]

\item This follows from \ref{prop:Steinberg.symbols.disjoint.commute} and the equality $g_{\mathcal{Q}_1 \cup \mathcal{Q}_2}=g_{\mathcal{Q}_1} \cdot g_{\mathcal{Q}_2}$.

\item Denote the section over $\Spin_{f_s}(\mathbf{O} ^*)[J_0]$ by $s_{J_0}$. Choose $\widetilde{g_1}\in \eta ^{-1}(g_1)$ and let $\chi_{g_1},\chi_{\widetilde{g_1}}$ be the conjugation maps $\chi_{g_1}(x)=g_1 x g_1 ^{-1}$ and $\chi_{\widetilde{g_1}}(x)=\widetilde{g_1} x \widetilde{g_1} ^{-1}$, respectively. The map $\chi_{\widetilde{g_1}} \circ s_{J_0} \circ \chi_{g_1} ^{-1}$ is a section of $\eta$ over $\Spin_{f_s}(\mathbf{O}^*)[J_0]$, so it coincides with $s_{J_0}$ on the derived subgroup of $\Spin_{f_s}(\mathbf{O} ^*)[J_0]$. By Lemma \ref{lem:derived.congruence}, $\chi_{\widetilde{g_1}} \circ s_{J_0} (g_2)=s_{J_0} \circ \chi_{g_1} (g_2)$, so
\[
[g_1:g_2]=[\widetilde{g_1},s_{J_0}(g_2)]=\chi_{\widetilde{g_1}} \left( s_{J_0}(g_2) \right) \cdot s_{J_0}(g_2) ^{-1} = 
\]
\[
=s_{J_0} \left( \chi_{g_1}(g_2) \right) \cdot s_{J_0}(g_2) ^{-1} =s_{J_0} \left( \chi_{g_1}(g_2) \cdot g_2 ^{-1} \right) = s_{J_0} \left( [g_1,g_2] \right) =1
\]
\item By continuity of generalized Steinberg symbols (Lemma \ref{lemma:prop_stein_1}), there is an ideal $I$ such that $\left\{ a_1 : c \right\}_\mathcal{Q} = 1$ if $c\in \mathbb{G}_m(\mathbf{O} ^*)[I]$. Thus,
\[
\left\{ a_1 : a_2 \right\}_\mathcal{Q} = \left\{ b_1 : b_2 \right\}_\mathcal{Q} \cdot \left\{ a_1 b_1 ^{-1} : b_2 \right\}_\mathcal{Q} \cdot \left\{ a_1 : a_2 b_2 ^{-1} \right\}_\mathcal{Q} = \left\{ b_1 : b_2 \right\}_\mathcal{Q} .
\]
\item Let $a_1=(\alpha_1,\alpha_1 ^{-1}), a_2=(\alpha_2,\alpha_2 ^{-1}), b_1=(\beta_1,\beta_1 ^{-1}), b_2=(\beta_2,\beta_2 ^{-1})$. Without loss of generality, $\val_P(\alpha_1)=1$. Then $\val_P(\alpha_2)=1$ and $\val_P(\alpha_2 ^{-1} \beta_2 -1) \geq 1$. 

Assume first that $\val_P(\alpha_2 ^{-1} \beta_2 -1)=1$. Denote $\gamma=1-\alpha_2 ^{-1} \beta_2$ and $c=(\gamma,\gamma ^{-1})\in \mathbb{G}_m(\mathbf{A} ^*)$. Since $\val_P(\gamma)=1$, we have $\left( a_1 c ^{-1} \right)_P, \left( b_1 c ^{-1} \right)_P \in \mathbb{G}_m(\mathbf{O} ^*)$. Thus,
\[
\left\{ a_1 : a_2 \right\}_P\overset{\ref{prop:Steinberg.symbols.J_0}}{=}\left\{ c : a_2 \right\}_P \overset{\ref{lemma:prop_stein_2}}{=} \left\{ c : (1-c)a_2 \right\}_P = \left\{ c : b_2 \right\}_P \overset{\ref{prop:Steinberg.symbols.J_0}}{=} \left\{ b_1 : b_2 \right\}_P.
\]

If $\val_P(\alpha_2 ^{-1} \beta_2-1) \geq 2$, then, since $P$ is not dyadic, there is $\delta \in (K_P^*)^ \times$ such that $\val_P(\alpha_2 \delta ^{-1} -1)=\val_P(\beta_2 \delta ^{-1} -1)=1$. Writing $d=(\delta,\delta ^{-1})\in \mathbb{G}_m(\mathbf{O})$, we have $\left\{ a_1 : a_2 \right\}_P=\left\{ a_1 : d \right\}_P=\left\{ b_1 : b_2 \right\}_P$.
\end{enumerate} 
\end{proof} 

\subsection{Commutators in $\Spin_{f_s}(O)$} \label{subsec:bdd.gen.f_s}

\begin{lemma} \label{lem:CKP} For every $I\in \mathcal{I}^{K^*}$, the commutator subgroup of $\SL_3(O^*)[I]$ contains $\SL_3(O^*)[I^2]$.
\end{lemma} 

\begin{proof} It is enough to show that there is a constant $C$ such that, for every ideal $I\in \mathcal{I}^{K}$, every element of $\SL_3(O)[I^2]$ is a product of $C$ commutators of elements of $\SL_3(O)[I]$. Let $e_{i,j}(t)$ be the elementary matrix with $t$ in the $(i,j)$th entry. For an ideal $I\in \mathcal{I}^K$, denote $X_I=\left\{ g ^{-1} e_{1,3}(t) g\mid t\in I , g\in \SL_3(O) \right\}$, $Y_I=\left\{ [g,h] \mid g,h\in \SL_3(O)[I] \right\}$, and let $E(K,I)= \langle X_I \rangle$. 

The identity $[e_{1,2}(t),e_{2,3}(s)]=e_{1,3}(st)$ implies that, if $t\in I^2$, then $e_{1,3}(t)\in Y_I$. By \cite{Morr07}, there is a constant $C$ such that, for every $I\in \mathcal{I}_K$, $E(K,I)=X_I^C$. Thus, for every $I\in \mathcal{I}_K$,
\[
E(K,I^2) \subseteq Y_I^C.
\]
By the Congruence Subgroup Property, $E(K,I^2)=\SL_3(O)[I^2]$, proving the claim.
\end{proof}

\begin{proposition}\label{prop:CKP} For every $I\in \mathcal{I}^{K^*}$, the commutator subgroup of $\Spin_{f_s}(O^*))[I]$, $\Spin_{f_s}(O^*)[I]'$, contains $\Spin_{f_s}(O^*)[I^4]$.
\end{proposition}
\begin{proof} Choosing a Chevalley basis, for any root $\alpha$ we get a 1-parameter subgroup $x_\alpha : \mathbb{G}_a \rightarrow \Spin_{f_s}$. Denote by $E(K^*,I)$ the normal subgroup of $\Spin_{f_s}(O^*)$ generated by the root elements $\left\{ x_\alpha(t) \mid t\in I , \alpha \in \Phi\right\}$. We use the following claims: \begin{enumerate}
\item \label{cond:CKP.roots} If $\alpha,\beta$ are roots such that $\alpha + \beta$ is also a root, then there is $\epsilon_{\alpha,\beta}\in \left\{ \pm 1 \right\}$ such that $[x_\alpha(t),x_\beta(s)]=x_{\alpha+\beta}(\epsilon_{\alpha,\beta}ts)$. Thus, $E(K^*,I^2) \subseteq \Spin_{f_s}(O^*)[I]'$.
\item \label{cond:CKP.CKP} By \cite[Theorem 4.8]{SS18}, there is a $K^*$-embedding $\rho : \SL_3 \rightarrow \Spin_{f_s}$ such that $\Spin_{f_s}(O^*)[I^2] \subseteq E(K^*,I) \cdot \rho(\SL_3(O^*,I))$.
\item \label{cond:CKP.lemma} By Lemma \ref{lem:CKP}, $\SL_3(O^*,I^2) \subseteq \SL_3(O^*,I)'$.
\end{enumerate} 
Thus,
\[
\Spin_{f_s}(O^*)[I^4] \overset{\eqref{cond:CKP.CKP}}\subseteq E(K^*,I^2) \cdot \rho \left( \SL_3(O^*,I^2) \right) \overset{\eqref{cond:CKP.roots},\eqref{cond:CKP.lemma}}\subseteq \Spin_{f_s}(O^*)[I]'.
\]
\end{proof}

\section{Finite conjugacy width} \label{sec:FCW}

In this section we prove Theorem \ref{thm:main}. An outline of the proof is as follows:
\begin{enumerate}
\item In \S\ref{subsec:FCW.split.finite.metaplectic}, we show that Theorem \ref{thm:main} follows from the following:


\begin{theorem} \label{thm:split.finite.metaplectic} Assume the Generalized Riemann Hypothesis. Let $K,S,n,f_a$ be as in \S\ref{sec:notations}. Assume that $n \geq 12$, that $f_a$ splits over $\mathbf{A}$, that the $S$-congruence kernel of $\Spin_{f_a}$ is trivial, and that there is a place $w_0\in S$ such that $i_{f_a}(K_{w_0}^n) \geq 2$. Then every $\Spin_{f_a}(K^*)$-metaplectic extension has a finite kernel.
\end{theorem} 

\item In \S\ref{subsec:FCW.criterion} we rephrase Theorem \ref{thm:split.finite.metaplectic} as a claim about central extensions of $\Spin_{f_s}(\mathbf{A} ^*)$ and find a criterion for triviality of Steinberg symbol $\left\{ a_1 : a_2 \right\}$.

\item In \S\ref{subsec:FCW.approximation} we use the approximation results from \S\ref{sec:approximation.tori} and the criterion from \S\ref{subsec:FCW.criterion} to prove Theorem \ref{thm:split.finite.metaplectic}.
\end{enumerate} 

\subsection{Reduction to Theorem \ref{thm:split.finite.metaplectic}} \label{subsec:FCW.split.finite.metaplectic} 

\begin{lemma} \label{lem:reduction.to.subspace} Assume that $n \geq 17$ and that there is nonarchimedean place $w_0$ in $S$. Then there is a subspace $V \subseteq K^n$ such that \begin{enumerate}
\item $\dim V=12$.
\item $f_a\restriction V$ splits over $\mathbf{A}$.
\item $i_{f_a}(V \otimes K_{w_0}) \geq 2$.
\item $i_{f_a}(V^\perp \otimes K_{w_0}) \geq 1$.
\end{enumerate}
\end{lemma} 

\begin{proof} For every place $v$ we define a 12-dimensional quadratic space $(W_v,q_v)$: \begin{enumerate}
\item If $v$ is non-real, let $(W_v,q_v)$ be the split 12-dimensional quadratic space.
\item If $v$ is real and $(K_v^n,f_a)$ has signature $(r,s)$ define $(W_v,q_v)$ as follows: \begin{enumerate}
\item If $r \leq 5$, let $(W_v,q_v)$ be the 12-dimensional negative definite form.
\item If $6 \leq r \leq 11$, let $(W_v,q_v)$ be the 12-dimensional form with signature $(6,6)$.
\item If $r \geq 12$, let $(W_v,q_v)$ be the 12-dimensional positive definite form.
\end{enumerate} 
\end{enumerate} 
The discriminants and Hasse invariants of $(W_v,q_v)$ are all 1. By \cite[Theorem 72:1]{Ome71}, there is a quadratic space $(U,g)$ over $K$ such that $(U \otimes K_v,g)$ is isometric to $(W_v,q_v)$, for every $v$. By construction, for every $v$ there is an isometric embedding $(U \otimes K_v,g)\cong (W_v,q_v) \hookrightarrow (K_v^n,f_a)$. By Hasse--Minkowski, there is an isometric embedding $(U,g) \hookrightarrow (K^n,f_a)$. Letting $V$ be the image of this embedding, the first two properties follow from the construction of $(W_v,q_v)$ and the last two properties follow from the universality of 5-dimensional forms over $K_{w_0}$.
\end{proof} 

\begin{notation} Let $(V,q)$ be a quadratic space and let $w_0\in S$. We denote by $ACW(q,w_0)$ the conjunction of the following conditions: \begin{enumerate}
\item $q$ is anisotropic.
\item The $S$-congruence kernel of $\Spin_q$ is trivial.
\item $i_q \left( V \otimes K_{w_0} \right) \geq 2$.
\end{enumerate} 
\end{notation} 

\begin{lemma} \label{lem:commutator.is.thick} Let $(V,q)$ be a quadratic space over $K$ and let $w_0\in S$. Assume that $ACW(q,w_0)$ holds. Let $X$ be the $\mathcal{C}_{\Spin_q,O}$-sifter $X_\Delta = [\Delta,\Delta]$. There is a natural number $n$ such that $C^n$ is a thick.
\end{lemma} 

\begin{proof} Denote $\Gamma=\Spin_q(O)$. Let $N,J$ be as in Theorem \ref{thm:Kneser.commutators} and let $n=\max \left\{ N,3 \right\}$. We show that Definition \ref{def:thick.sifter} holds for $X^n$ with $\Delta_1=\Gamma[J]$.

Let $\Delta \in \mathcal{C}_{\Spin_q,O}$. Strong Approximation and a local computation (see \cite[Corollary 2.8]{AGKS}) show that the congruence closure of $X_\Delta^3$ has non-empty interior.

Let $\Delta_2$ be a principal congruence subgroup. By strong approximation, there is a 1-separated element $g\in \left[ \Delta_2 , \Delta_2 \right]$. By Theorem \ref{thm:Kneser.commutators}, there is a principal congruence subgroup $\Delta_3$ such that
\[
\left[ \Delta_1,\Delta_3\right] \subseteq \gcl_{\Spin_{f_a}(O)}(g)^N \subseteq \left[ \Delta_2 , \Delta_2 \right]^N.
\]
\end{proof} 

\begin{lemma} \label{lem:finite.metaplectic.implies.BCW} Let $(V,q)$ be a nondegenerate quadratic space over $K$ and let $w_0\in S$. Assume that $ACW(q,w_0)$ holds and that every $\Spin_q(K^*)$-metaplectic extension has a finite kernel. Then every $S$-arithmetic subgroup of $\Spin_q(K)$ has bounded commutator width.
\end{lemma} 

\begin{proof} We verify the conditions of Theorem \ref{thm:commutator.width.finite.metaplectic}. \begin{enumerate}
\item The $S$-congruence kernel is trivial by assumption.
\item The condition on the commutator sifter holds by Lemma \ref{lem:commutator.is.thick}.
\item Normal generation by congruence subgroups holds by Corollary \ref{cor:spin*.perfect}.
\item Finite metaplectic kernel holds by assumption.
\end{enumerate} 
Thus, $\Spin_q(O)$ has bounded commutator width. By the strong CSP, every group commensurable to $\Spin_q(O)$ also has bounded commutator width.
\end{proof} 

\begin{lemma} \label{lem:profinite.BCW.weak} Let $G$ be a simply connected simple algebraic group over $K$ and let $H \subseteq G(\mathbf{A})$ be a compact and open subgroup. Then there is a constant $N$ such that \begin{enumerate}
\item For every $h\in H$, $\wid \left( \gcl_H(h) \right) < N$.
\item For every $I\in \mathcal{I} ^K$, $\wid\left( \left[ H[I],H[I] \right] \right) < N$.
\end{enumerate} 
\end{lemma} 

\begin{proof} \begin{enumerate}
\item This follows from \cite[Lemmas 3.20 and 3.22]{AM22}.
\item By definition $H[I]=H\cap G(\mathbf{O})[I]$, so we can assume that $H \subseteq G(\mathbf{O})$. Since $H$ is open there are natural numbers $(a_P)_{P\notin S}$ such that $\prod_{P \notin S} G(O_P)[P^{a_P}] \subseteq H$ and, moreover, only finitely many $a_P$ are nonzero. 

By \cite[Lemmas 3.20 and 3.22]{AM22} there is a constant $C$ and, for every $P\in \mathcal{I} ^K$, a constant $c_P$ such that \begin{enumerate}
\item For every $n\in \mathbb{N}$, 
\[
G(O_P)[P^{2n+c_P}] \subseteq \big[ G(O_P)[P^n] , G(O_P)[P^n] \big]^C \subseteq G(O_P)[P^{2n}].
\]
\item $c_P=0$ for all but finitely many $P$.
\end{enumerate} 

If $I=\prod_{P\notin S} P^{b_P}$, then
\[
\prod_{P\notin S} G(O_P)[P^{a_P+b_P}] \subseteq H[I] \subseteq \prod_{P\notin S} G(O_P)[P^{b_P}],
\]
so
\[
\prod_{P\notin S} G(O_P)[P^{2(a_P+b_P)+c_P}] \subseteq \left[ H[I] , H[I] \right]^C \subseteq \prod_{P\notin S} G(O_P)[P^{2b_P}].
\]
The claim follows since $\left[ \prod_{P\notin S} G(O_P)[P^{2b_P}] : \prod_{P\notin S} G(O_P)[P^{2(a_P+b_P)+c_P}] \right]$ is bounded uniformly in $I$.
\end{enumerate}
\end{proof} 

\begin{lemma} \label{lem:BCW.goes.up} Let $(W,q)$ be quadratic space over $K$, let $V \subseteq W$ be a subspace, let $w_0\in S$, and let $\Gamma \subseteq \Spin_q(K)$ be an $S$-arithmetic subgroup. Assume that both $ACW(q,w_0)$ and $ACW(q\restriction V,w_0)$ hold and that $i_{w_0}(V^\perp) \geq 1$. If $\Gamma \cap \Spin_{q\restriction V}(K)$ has bounded commutator width, then so does $\Gamma$.
\end{lemma} 

\begin{proof} Denote $\Delta=\Gamma \cap \Spin_{q\restriction V}(K)$. By assumption, there is $N_1$ such that $\left[ \Delta[I] , \Delta[I] \right]^{N_1}= \Delta[I]'$, for every $I\in \mathcal{I} ^K$. Let $N_2$ be the constant obtained by applying Lemma \ref{lem:Kneser.uff} to $q$ and $\Gamma$. Let $N_3$ be the constant obtained by applying Lemma \ref{lem:profinite.BCW.weak} to $\Spin_q$ and the congruence closure of $\Gamma$. We show that the commutator width of any congruence subgroup of $\Gamma$ is bounded by $N_1+N_2+N_3$.

Let $I\in \mathcal{I} ^K$. By the assumptions there is an ideal $I_1 \subseteq I$ such that $\left[ \Delta[I] , \Delta[I] \right]^{N_1} \supseteq \Delta[I_1]$. By strong approximation, there is a 1-separated element $\gamma \in \left[ \Gamma[I_1],\Gamma[I_1] \right]$. By \ref{lem:Kneser.uff}, there is an ideal $I_2 \subseteq I_1$ such that $\Gamma[I_2] \subseteq \gcl_\Gamma(\gamma)^{N_2} \cdot \Spin_{q\restriction V}(O)$, so 
\[
\Gamma[I_2] \subseteq \gcl_\Gamma(\gamma)^{N_2} \cdot \Delta[I_1] \subseteq \left[ \Gamma[I] , \Gamma[I] \right]^{N_1+N_2}.
\] 
By Lemma \ref{lem:profinite.BCW.weak}, the width of $\left[ \Gamma[I] , \Gamma[I] \right]$ is at most $N_1+N_2+N_3$.
\end{proof} 

\begin{lemma} \label{lem:commutator.width.to.conjugacy.width} Let $(V,q)$ be a quadratic space over $K$, let $w_0\in S$, and let $\Gamma \subseteq \Spin_q(K)$ be an $S$-arithmetic group. Assume that $ACW(q,w_0)$ holds and that $\Gamma$ has bounded commutator width. Then there is a constant $C$ such that, for every $\gamma \in \Gamma$,
\[
\wid_\Gamma (\gcl_\Gamma (\gamma)) \leq C \cdot \max \left\{ \wid_{\Spin_q(K_v)}\left( \gcl_{\Spin_q(K_v)} (\gamma) \right) \mid \text{$q_v$ is anisotropic} \right\}.
\]
\end{lemma} 

\begin{proof} Let $N_1$ be the constant obtained by applying Theorem \ref{thm:Kneser.commutators} to $\Gamma$ and let $N_2$ be the constant obtained by applying Lemma \ref{lem:profinite.BCW.weak} to $\Spin_q$ and the congruence closure of $\Gamma$. We show that the claim holds for $C=N_1+N_2$.

Let $\gamma \in \Gamma \smallsetminus Z(\Gamma)$. Denote $w=\max \left\{ \wid_{\Spin_q(K_v)}\left( \gcl_{\Spin_q(K_v)} (\gamma) \right) \mid \text{$q_v$ is anisotropic} \right\}$. By strong approximation, $\gcl_\Gamma(\gamma)^w$ is dense in $\prod_{\text{$q_v$ anisotropic}} \Spin_q(K_v)$, so it contains a 1-separated element $\delta$. By Theorem \ref{thm:Kneser.commutators} and the assumption that $\Gamma$ has the strong CSP, $\gcl_\Gamma(\gamma)^{N_1w} \supseteq \gcl_\Gamma(\delta)^{N_1}$ contains a congruence subgroup. By Lemma \ref{lem:profinite.BCW.weak}, $\gcl_\Gamma (\gamma)^{N_1w+N_2}$ is a subgroup and the claim follows.
\end{proof}

\begin{proof}[Proof that Theorem \ref{thm:split.finite.metaplectic} implies Theorem \ref{thm:main}] Assume that the condition of Theorem \ref{thm:main} hold. If $S$ contains a nonarchimedean place, let $w_0$ be that place. Otherwise, let $w_0$ be the place from the second condition. In both cases, $ACW(f_a,w_0)$ holds.

If $S$ contains a nonarchimedean place, let $V \subseteq K^n$ be the subspace given by Lemma \ref{lem:reduction.to.subspace}. If $S$ does not contain a nonarchimedean place, let $V=K^n$. In both cases, $f_a\restriction V$ splits over $\mathbf{A}$ and $ACW(f_a\restriction V,w_0)$ holds.

By Theorem \ref{thm:split.finite.metaplectic}, every $\Spin_{f_a\restriction V}(K^*)$-metaplectic extension has finite kernel. By Lemma \ref{lem:finite.metaplectic.implies.BCW}, $\Gamma \cap \Spin_{f_a\restriction V}(K)$ has bounded commutator width. By Lemma \ref{lem:BCW.goes.up}, $\Gamma$ has bounded commutator width, so Lemma \ref{lem:commutator.width.to.conjugacy.width} implies the claim of Theorem \ref{thm:main}.
\end{proof} 

\subsection{$\psi \left( \Spin_{f_a}(K^*) \right)$-metaplectic extensions and their Steinberg symbols} \label{subsec:FCW.criterion}

Assume now that $n,K,S,f_a$ are as in Theorem \ref{thm:split.finite.metaplectic} and fix an isometry $\psi:(\mathbf{A} ^n,f_a) \rightarrow (\mathbf{A} ^n,f_s)$. 

$\psi$ induces an isometry $\left( (\mathbf{A} ^*)^n,f_s) \right) \rightarrow \left( (\mathbf{A} ^*)^n,f_s \right)$ and an isomorphism of topological groups $\Spin_{f_a}(\mathbf{A} ^*) \rightarrow \Spin_{f_s}(\mathbf{A} ^*)$, both of which we continue to denote by $\psi$. We identify the completion of $\psi \left( \Spin_{f_a}(K^*) \right)$ with $\Spin_{f_s}(\mathbf{A} ^*)$.

If
\begin{equation} \label{eq:f_a.metaplectic.ext}
1 \rightarrow D \rightarrow E \stackrel{\xi}{\rightarrow} \Spin_{f_a}(\mathbf{A} ^*) \rightarrow 1
\end{equation}
is a $\Spin_{f_a}(K^*)$-metaplectic extension, then
\begin{equation} \label{eq:f_a.metaplectic.ext}
1 \rightarrow D \rightarrow E \stackrel{\eta:=\psi \circ \xi}{\rightarrow} \Spin_{f_s}(\mathbf{A} ^*) \rightarrow 1
\end{equation}
is a $\psi \left( \Spin_{f_a}(K^*) \right)$-metaplectic extension, namely, it is a topological central extension with a discrete kernel, a continuous splitting over an open subgroup of the form $\Spin_{f_s}(\mathbf{O} ^*)[J_0]$, and a (possibly discontinuous) splitting over $\psi \left( \Spin_{f_a}(K^*) \right)$ whose image is dense. \\

By the last paragraph, Theorem \ref{thm:split.finite.metaplectic} is equivalent to the statement that every $\psi \left( \Spin_{f_a}(K^*) \right)$-metaplectic extension splits. The proof uses properties of Steinberg symbols in such extensions. In the rest of this subsection, we prove one such property, Lemma \ref{lem:switcharoo}.\\

\begin{lemma} \label{lem:Witt} Let $R$ be either a local field or the ring of integers in a local field, let $(V,q)$ be a non-degenerate quadratic space of dimension at least 6 over $R$, and let $v_1,v_2\in V$. Assume that $q(v_1)=q(v_2)\in R ^ \times$ and that 2 is invertible in $R$. Then there is an element $r\in \Spin_q(R)$ such that $r(v_1)=v_2$.
\end{lemma} 

\begin{proof} Assume first that $v_1$ and $v_2$ are orthogonal. Then $q(v_1+v_2)=2q(v_1)\in R^ \times$ and the restriction of $q$ to $\left\{ v_1,v_2 \right\}^\perp$ is non-degenerate. Since every non-degenerate quadratic form over $R$ of dimension at least 4 is universal, there is a vector $v_3\in \left\{ v_1,v_2 \right\}^\perp$ with $q(v_3)=\frac{1}{2q(v_1)}$. The element $r:=(v_1-v_2)v_3\in \Cliff_{q}(R)$ is in $\Spin_q(R)$ and $r(v_1)=rv_1r'=v_2$.

If $v_1$ and $v_2$ are not orthogonal, find a vector $u\in \left\{ v_1,v_2 \right\} ^\perp$ with $q(u)=q(v_1)$. By the proof above, there are elements $r_1,r_2\in \Spin_{q}(R)$ such that $r_i(v_i)=u$, for $i=1,2$. The element $r:=r_2 ^{-1} r_1$ maps $v_1$ to $v_2$.
\end{proof} 



\begin{lemma} \label{lem:non.split.tori.compact} Let $n \in \mathbb{N}_{\geq 7}$, let $t\in O_+$, let $Q\in \mathcal{P}^K \smallsetminus \mathcal{S}(T_t)$, and let $v_1,v_2\in K_Q^n$ be orthogonal vectors of norms $f_s(v_1)=1,f_s(v_2)=t$. Let $\varphi:T_t \rightarrow \Spin_{f_s}$ be the $K_Q$ morphism defined by $\varphi \left( x+y\sqrt{-t} \right)=x+yv_1v_2$. Then there is $r \in \Spin_{f_s}(K_Q)$ such that $r \varphi \left( T_t(K_Q) \right) r ^{-1} \subseteq \Spin_{f_s}(O_Q)$.
\end{lemma} 

\begin{proof} By rescaling, we can assume that either $\val_Q(t)=0$ (if $Q$ is inert) or $\val_Q(t)=1$ (if $Q$ is ramified). It follows that $T_t(K_Q)=\left\{ x+y\sqrt{-t} \mid x,y\in O_K \right\}$.

Let $u_1\in O_Q^n$ be a vector of norm 1 (for example, take $e_1$). By Lemma \ref{lem:Witt}, there is an element $r_1\in \Spin_{f_s}(K_Q)$ such that $r_1(v_1)=u_1$. The subspace $U:= u_1^\perp$ is non-degenerate and $r_1(v_2)\in U$. Moreover, the quadratic module $U\cap O^n$ is non-degenerate, so there is an element $u_2\in U\cap O^n$ with $f_s(u_2)=t$. By Lemma \ref{lem:Witt}, there is an element $r_2\in \Spin_{f_s\restriction_U}(K_Q)$ such that $r_2(r_1(v_2))=u_2$. Using the embedding $U \hookrightarrow K_Q^n$, we consider $r_2$ as an element of $\Spin_{f_s}(K_Q)$.

Denote $r:=r_2r_1$. Since $r_2(u_1)=u_1$, we have $r(v_1)=u_1$. By construction, $r(v_2)=u_2$. Thus, for $x+y\sqrt{-t}\in T_t(K_Q)$, 
\[
r \varphi(x+y\sqrt{-t}) r ^{-1} =x+yu_1u_2\in\Spin_{f_s}(O_Q).
\]
\end{proof} 



\begin{lemma} \label{lem:T_t.in.spin} Assume that $n \geq 10$. Let $t\in O_+$, let $\rho$ be a trivialization of $T_t$, and let $\alpha_1,\alpha_2 \in T_t(K)$.  Then there are elements $g_1,g_2\in \Spin_{f_a}(K)$ and an element $r \in \Spin_{f_s}(\mathbf{A})$ such that, for $i=1,2$, \begin{enumerate}
\item \label{item:T_t.in.spin.St} $\left( r \psi(g_i) r ^{-1} \right)_{\mathcal{S}(T_t)}=h_i(\rho(\alpha_i))$.
\item \label{item:T_t.in.spin.supp} $\left( r \psi(g_i) r ^{-1} \right)_{\mathcal{P} ^K \smallsetminus \supp(\alpha_i)}\in \Spin_{f_s}(\mathbf{O})$. 
\end{enumerate} 
\end{lemma} 

\begin{proof} Since $h_i$ are integral, it is enough to replace Claim \ref{item:T_t.in.spin.supp} by
\begin{enumerate}[start=3]
\item \label{item:T_t.in.spin.St.c} $\left( r \psi(g_i) r ^{-1} \right)_{\mathcal{P} ^K \smallsetminus \mathcal{S}(T_t)}\in \Spin_{f_s}(\mathbf{O})$. 
\end{enumerate} 

By Hasse--Minkowski (and the assumptions $n>10$ and $t\in O_+$), there are $f_a$-orthogonal vectors $v_1,v_2,v_3,u_1,u_2,u_3\in K^n$ such that $f_a(v_1)=f_a(v_2)=f_a(v_3)=1$ and $f_a(u_1)=f_a(u_2)=f_a(u_3)=t$. Choose elements $\beta_i=x_i'+\sqrt{-t}y_i'\in T_t(\overline{K})$ such that $\beta_i ^2=\alpha_i$ and define
\[
g_1=(x_1'+y_1' v_1u_1)(x_1'-y_1'v_2u_2)
\]
and
\[
g_2=(x_2'+y_2' v_2u_2)(x_2'-y_2'v_3u_3).
\]
For $i=1,2$, $\Gal(\overline{K}/K) \cdot \beta_i=\left\{ \pm \beta_i \right\}$, so $g_i\in \Spin_{f_a}(\overline{K})^{\Gal(\overline{K}/K)}=\Spin_{f_a}(K)$.

Let 
\[
\mathcal{Q} := \left\{ Q\in \mathcal{P} ^K \smallsetminus \mathcal{P} ^K(2J_f) \mid (\forall i) \Big( \alpha_i\in T_t(O_Q) \text{ and } u_i\in O_Q^n \text{ and } v_i\in O_Q^n \Big) \right\}.
\]
$\mathcal{Q}$ is cofinite and, for $i=1,2$, $\psi(g_i)_\mathcal{Q} \in \Spin_{f_s}(\mathbf{O})$.

For every $Q\in \mathcal{P} ^K$ we define an element $r_Q\in \Spin_{f_s}(K_Q)$ as follows: \begin{enumerate}[label=\alph*.]
\item \label{item:T_t.in.spin.split.bad} If $Q\in \mathcal{S}(T_t) \smallsetminus \mathcal{Q}$, apply Lemma \ref{lem:Witt} to get an element $r_Q\in \Spin_{f_s}(K_Q)$ such that 
\begin{equation} \label{eq:T_t.in.spin.v}
h_Q \left( \psi(v_1) \right) = e_1 \quad h_Q \left( \psi(v_2) \right) = e_3 \quad h_Q \left( \psi(v_3) \right) = e_5
\end{equation}
and
\begin{equation} \label{eq:T_t.in.spin.u}
h_Q \left( \psi(u_1) \right) = \rho_Q(\sqrt{-t})e_2 \quad h_Q \left( \psi(u_2) \right) = \rho_Q(\sqrt{-t})e_4 \quad h_Q \left( \psi(u_3) \right) = \rho_Q(\sqrt{-t})e_6.
\end{equation}
\item \label{item:T_t.in.spin.split.good} If $Q\in \mathcal{S}(T_t) \cap \mathcal{Q}$, apply Lemma \ref{lem:Witt} to get an element $r_Q\in \Spin_{f_s}(O_Q)$ such that \eqref{eq:T_t.in.spin.v} and \eqref{eq:T_t.in.spin.u} hold.
\item \label{item:T_t.in.spin.non.split.bad} If $Q\in \mathcal{S}(T_t)^c \smallsetminus \mathcal{Q}$, apply Lemma \ref{lem:non.split.tori.compact} to get an element $r_Q\in \Spin_{f_s}(K_Q)$ such that $r_Q \psi(g_1) r_Q ^{-1}, r_Q \psi(g_2) r_Q ^{-1} \in \Spin_{f_s}(O_Q)$.
\item If $Q\in \mathcal{S}(T_t) \cap \mathcal{Q}$, let $r_Q=1$.
\end{enumerate} 

Since $\mathcal{Q}$ is cofinite, the sequence $r:=(r_Q)_Q$ belongs to $\Spin_{f_s}(\mathbf{A})$. For $i=1,2$, if $Q\notin \mathcal{S}(T_t)$, then $\left( r \psi(g_i) r ^{-1} \right)_Q=r_Q \psi(g_i) r_Q ^{-1} \in \Spin_{f_s}(O_Q)$ either by Claim \ref{item:T_t.in.spin.non.split.bad} or by $\psi(g_i)_\mathcal{Q} \in \Spin_{f_s}(\mathbf{O})$. This proves Claim \ref{item:T_t.in.spin.St.c}. If $Q \in \mathcal{S}(T_t)$, then Claims \ref{item:T_t.in.spin.split.bad} and \ref{item:T_t.in.spin.split.good} imply that
\[
\left( r \psi(g_1) r ^{-1} \right)_Q=r_Q\Big( (x_1'+y_1' \psi(v_1) \psi (u_1))(x_1'-y_1'\psi(v_2) \psi(u_2)) \Big)_Q r_Q ^{-1} =
\]
\[
=(x_1')^2+x_1'y_1' \rho_Q(\sqrt{-t}) e_1e_2-x_1'y_1'\rho_Q(\sqrt{-t})e_3e_4+t(y_1')^2e_1e_2e_3e_4=
\]
\[
\frac{x_1+1}{2}+\frac{y_1\rho_Q(\sqrt{-t})}{2}e_1e_2-\frac{y_1\rho_Q(\sqrt{-t})}{2}e_3e_4+(y_1')^2te_1e_2e_3e_4=h_1(\rho_Q(\alpha_1)).
\]
Similarly, $\left( r \psi(g_2) r ^{-1} \right)_Q=h_2(\rho_Q(\alpha_2))$, proving Claim \ref{item:T_t.in.spin.St}.

\end{proof} 

\begin{lemma} \label{lem:switcharoo} Let \eqref{eq:f_a.metaplectic.ext} be a $\psi \left( \Spin_{f_a}(K^*) \right)$-metaplectic extension. Let $t\in O^*_+$, let $\mathcal{Q} \subseteq \mathcal{P}_{K^*}$ be an internal subset, let $\rho$ be a trivialization of $T_t$, and let $\alpha_1,\alpha_2\in T_t(K^*)$. Assume that $\mathcal{P}^{K^*}(J_0),\mathcal{Q} \subseteq \mathcal{S}(T_t)$ and that either $\mathcal{P}^{K^*}(J_0) \subseteq \mathcal{Q}$ or $\rho(\alpha_2)_{\mathcal{P} ^{K^*}(J_0)}\in \mathbb{G}_m(\mathbf{O}^*)[J_0^2]$. Then 
\[
\left\{ \rho(\alpha_1) : \rho(\alpha_2) \right\}_{\mathcal{Q}}=\left\{ \rho(\alpha_1) : \rho(\alpha_2 ^{-1})\right\}_{\left( \supp(\alpha_1)\cup \supp(\alpha_2) \right) \smallsetminus \mathcal{Q}}.
\]
\end{lemma} 

\begin{proof} Denote $\mathcal{R}=\supp(\alpha_1)\cup \supp(\alpha_2)$. By Lemma \ref{lem:properties.T_t}, $\mathcal{R} \subseteq \mathcal{S}(T_t)$. Let $g_1,g_2\in \Spin_{f_a}(K^*)$ and $r\in \Spin_{f_s}(\mathbf{A} ^*)$ be the elements obtained by applying Lemma \ref{lem:T_t.in.spin} coordinate-wise to $t,\rho,\alpha_1,\alpha_2$. By Proposition \ref{prop:Steinberg.symbols},
\[
1=\left[ \psi(g_1) : \psi(g_2) \right] = \left[ r \psi(g_1) r ^{-1} : r \psi(g_2) r ^{-1} \right]=
\]
\[
=\left[ r \psi(g_1) r ^{-1} : r \psi(g_2) r ^{-1} \right]_{\mathcal{Q} \cup \mathcal{R}\cup \mathcal{P}^{K^*}(J_0)} \cdot \left[ r \psi(g_1) r ^{-1} : r \psi(g_2) r ^{-1} \right]_{\left( \mathcal{Q} \cup \mathcal{R} \cup \mathcal{P}^{K^*}(J_0) \right) ^c}.
\]
Since $\left( r \psi(g_i) r ^{-1}\right)_{\left( \mathcal{Q} \cup \mathcal{R} \cup \mathcal{P}^{K^*}(J_0) \right) ^c}\in \Spin_{f_s}(\mathbf{O} ^*)[J_0^2]$, Proposition \ref{prop:Steinberg.symbols} implies that $\left[ r \psi(g_1) r ^{-1} : r \psi(g_2) r ^{-1} \right]_{\left( \mathcal{Q} \cup \mathcal{R} \cup \mathcal{P}^{K^*}(J_0) \right) ^c}=1$. Since $\mathcal{Q} \cup \mathcal{R} \cup \mathcal{P}^{K^*}(J_0) \subseteq \mathcal{S}(T_t)$, we get
\[
1=\left[ r \psi(g_1) r ^{-1} : r \psi(g_2) r ^{-1} \right]_{\mathcal{Q} \cup \mathcal{R}\cup \mathcal{P}^{K^*}(J_0)}=\left\{  \rho(\alpha_1) : \rho(\alpha_2) \right]_{\mathcal{Q} \cup \mathcal{R}\cup \mathcal{P}^{K^*}(J_0)}.
\]
Thus,
\[
1=\left\{ \rho(\alpha_1) : \rho(\alpha_2) \right\}_{\mathcal{Q}} \cdot \left\{ \rho(\alpha_1) : \rho(\alpha_2) \right\}_{\mathcal{R} \smallsetminus \mathcal{Q}} \cdot \left\{ \rho(\alpha_1) : \rho(\alpha_2) \right\}_{\mathcal{P}^{K^*}(J_0) \smallsetminus \left( \mathcal{R} \cup \mathcal{Q} \right)}.
\]
In order to prove the lemma, it remains to show that $\left\{ \rho(\alpha_1) : \rho(\alpha_2) \right\}_{\mathcal{P}^{K^*}(J_0) \smallsetminus \left( \mathcal{R} \cup \mathcal{Q} \right)}=1$. If $\mathcal{P}^{K^*}(J_0) \subseteq \mathcal{Q}$, this is  tautologically true. If $\rho(\alpha_2)_{\mathcal{P} ^{K^*}(J_0)}\in \mathbb{G}_m(\mathbf{O})[J_0^2]$ then $h_2(\rho(\alpha_2))_{\mathcal{P} ^{K^*}(J_0)}\in \Spin_{f_s}(\mathbf{O} ^*)[J_0^2]$. Since $\mathcal{P}^{K^*}(J_0) \smallsetminus \left( \mathcal{R} \cup \mathcal{Q} \right) \subseteq \mathcal{S}(T_t) \smallsetminus  \supp(\alpha_1)$, we get that 
\[
h_1(\alpha_1)_{\mathcal{P}^{K^*}(J_0) \smallsetminus \left( \mathcal{R} \cup \mathcal{Q} \right)}=\left( \ell g_1 \ell ^{-1} \right)_{\mathcal{P}^{K^*}(J_0) \smallsetminus \left( \mathcal{R} \cup \mathcal{Q} \right)}\in \Spin_{f_s}(\mathbf{O} ^*).
\]
By Proposition \ref{prop:Steinberg.symbols}, $\left\{ \rho(\alpha_1) : \rho(\alpha_2) \right\}_{\mathcal{P}^{K^*}(J_0) \smallsetminus \left( \mathcal{R} \cup \mathcal{Q} \right)}=1$ in this case as well.
\end{proof}

\subsection{Proof of Theorem \ref{thm:split.finite.metaplectic}} \label{subsec:FCW.approximation}

Let $n,K,S,f_a$ be as in Theorem \ref{thm:split.finite.metaplectic}. Fix an isometry $\psi:(\mathbf{A} ^n,f_a) \rightarrow (\mathbf{A} ^n,f_s)$. By \S\ref{subsec:FCW.criterion}, it is enough to prove that every $\psi \left( \Spin_{f_a}(K^*) \right)$-metaplectic extension has a finite kernel.

Let
\begin{equation} \label{eq:FCW.split.finite.D_0.ses.f_s}
1 \rightarrow D \rightarrow E \overset{\eta}{\rightarrow} \Spin_{f_s}(\mathbf{A}^*) \rightarrow 1
\end{equation}
be a $\psi \left( \Spin_{f_a}(K^*) \right)$-metaplectic extension. By assumption, \eqref{eq:FCW.split.finite.D_0.ses.f_s} has a continuous section $s_{\mathbf{O} ^*}$ over $\Spin_{f_s}(\mathbf{O} ^*)[J_0]$ and a (possibly discontinuous) section $s_{K^*}$ over $\psi \left( \Spin_{f_a}(K^*) \right)$.

We define the following subgroups of $D$:

\begin{definition} \begin{enumerate}
\item For $Q\in \mathcal{P} ^{K^*}$, define
\[
D_Q:=\Big \langle \left\{ a_1 : a_2 \right\}_Q \mid a_1\in \mathbb{G}_m(\mathbf{A} ^*) \quad \val_Q(a_1)=-1 \quad a_2\in \mathbb{G}_m(\mathbf{O} ^*)[2] \Big \rangle .
\]
\item Let $D_0 \supseteq D_1 \supseteq D_2 \supseteq D_3$ be the following subgroups of $D$:
\[
D_0:=\Big \langle \left\{ a_1 : a_2 \right\} \mid a_1,a_2\in \mathbb{G}_m(\mathbf{A} ^*) \Big \rangle .
\]
\[
D_1:=\Big \langle \left\{ a_1 : a_2 \right\}_\mathcal{Q} \mid a_1\in \mathbb{G}_m(\mathbf{A} ^*) \quad \mathcal{Q} \cap \mathcal{P} ^{K^*}(2J_0)=\emptyset \quad a_2\in \mathbb{G}_m(\mathbf{O} ^*) \Big \rangle .
\]
\[
D_2:=\Big \langle \left\{ a_1 : a_2 \right\}_\mathcal{Q} \mid a_1\in \mathbb{G}_m(\mathbf{A} ^*) \quad \mathcal{Q} \cap \mathcal{P} ^{K^*}(2J_0)=\emptyset \quad a_2\in \mathbb{G}_m(\mathbf{O} ^*)[2] \Big \rangle .
\]
\[
D_3:=\Big \langle D_Q \mid Q\in \mathcal{P} ^{K^*} \smallsetminus (\mathcal{P} ^K \cup \mathcal{P} ^{K^*}(J_0) \Big \rangle .
\]
\end{enumerate} 
\end{definition} 

If $A$ is an abelian group and $M$ is a natural number, denote 
\[
^M A:= \left\{ Mx \mid x\in A \right\}.
\] 
Theorem \ref{thm:split.finite.metaplectic} follows from the following claims: \begin{enumerate}
\item $D=D_0$.
\item $D_0=D_1$ and $D_2=D_3$.
\item $|D_1 / D_2|$ is finite.
\item For every natural number $M$, $|D_3/\,^{M}D_3| \leq M$.
\item $D_3$ has finite exponent: there is $M\in \mathbb{N}$ such that $\,^{M}D_3=1$.
\end{enumerate} 
We prove these claims in the following subsections.

\subsubsection{$D=D_0$}
Consider the extension
\begin{equation} \label{eq:ses.D/D_0}
1 \rightarrow D/D_0 \rightarrow E/D_0 \overset{\eta}{\rightarrow} \Spin_{f_s}(\mathbf{A}^*) \rightarrow 1
\end{equation}
induced from \eqref{eq:FCW.split.finite.D_0.ses.f_s}. By definition, the Steinberg symbols for \eqref{eq:ses.D/D_0} vanish. By Theorem \ref{thm:Matsumoto}, there is a (possibly discontinuous) section $\sigma$ of \eqref{eq:ses.D/D_0} over the subgroup $\Spin_{f_s}(K^*)$. In addition, the sections $s_{\mathbf{O} ^*}$ and $s_{K^*}$ induce corresponding sections $\overline{s_{\mathbf{O}^*}}$ and $\overline{s_{K^*}}$ of \eqref{eq:ses.D/D_0} over $\Spin_{f_s}(\mathbf{O} ^*)[J_0]$ and $\psi \left( \Spin_{f_a}(K^*) \right)$, respectively. The section $\overline{s_{\mathbf{O} ^*}}$ is continuous and the section $\overline{s_{K^*}}$ has a dense image.

Since $\Spin_{f_s}(K^*) \cap \Spin_{f_s}(\mathbf{O} ^*[J_0])=\Spin_{f_s}(O^*[J_0])$ and by Proposition \ref{prop:CKP}, the restrictions of $\sigma$ and $\overline{s_{\mathbf{O} ^*}}$ to $\Spin_{f_s}(O^*)[J_0^4]$ coincide. By Lemma \ref{lemma:con_hom}, there is a continuous splitting $s$ of \eqref{eq:ses.D/D_0} over $\Spin_{f_s}(\mathbf{A} ^*)$. By Corollary \ref{cor:spin*.perfect}, $\psi \left( \Spin_{f_a}(K^*) \right)$ is perfect, so the restriction of $s$ to $\psi \left( \Spin_{f_a}(K^*) \right)$ coincides with $\overline{s_{K^*}}$. It follows that $\overline{s_{K^*}}$ is a continuous section with dense image, so $D/D_0=1$, as we wanted to prove.

\subsubsection{$D_0=D_1$ and $D_2=D_3$.}

$D_0=D_1$ follows from the following lemma:
\begin{lemma} \label{lem:weak.approx.ultraproduct} Let $a_1,a_2\in \mathbb{G}_m(\mathbf{A}^*)$. Then there are elements $b_1,c_1\in \mathbb{G}_m(\mathbf{A} ^*)$ and $b_2,c_2\in \mathbb{G}_m(\mathbf{O} ^*)$ such that $\supp(b_1)$, $\supp(c_1)$, and $\mathcal{P}^{K^*}(2 J_0)$ are pairwise disjoint and 
\[
\left\{ a_1:a_2 \right\} = \left\{ b_1:b_2 \right\}_{\supp(b_1)} \cdot \left\{ c_1 : c_2 \right\}_{\supp(c_1)}.
\]
\end{lemma} 

\begin{proof} By Proposition \ref{prop:Steinberg.symbols}, there is a small internal set $\mathcal{Q} \subseteq \mathcal{P} ^{K^*}$ such that $\left\{ a_1 : a_2 \right\}=\left\{ a_1 : a_2 \right\}_\mathcal{Q}$. Without loss of generality we can assume that $\mathcal{P} ^{K^*}(2J_0) \subseteq \mathcal{Q}$. 

By Proposition \ref{prop:Steinberg.symbols}\eqref{prop:Steinberg.symbols.continuity} there is an ideal $I\in \mathcal{I} ^{K^*}(\mathcal{Q})$ such that if $x_1,x_2\in \mathbb{G}_m(\mathbf{A} ^*)$ satisfy $(a_1 x_1 ^{-1} )_{\mathcal{Q}}\in \mathbb{G}_m(\mathbf{O} ^*)[J_0^2]$ and $(a_2 x_2 ^{-1} )_\mathcal{Q} \in \mathbb{G}_m(\mathbf{O} ^*)[I]$, then $\left\{ a_1 : a_2 \right\}_\mathcal{Q} = \left\{ x_1 : x_2 \right\}_\mathcal{Q}$. Without loss of generality we can assume that $\mathcal{P} ^{K^*}(I)=\mathcal{Q}$.

By Lemma \ref{lem:weak.approx.standard} and a coordinate-wise argument, there are $t\in O_+^*$, a trivialization $\rho$ of $T_t$, and elements $\alpha_1,\alpha_2\in T_t(K^*)$ such that \begin{enumerate}[label=\alph*.]
\item \label{claim:weak.approx.ultraproduct.t} $\mathcal{Q} \subseteq \mathcal{S}(T_t)$.
\item \label{claim:weak.approx.ultraproduct.a_i} $\left( a_i \rho(\alpha_i) ^{-1} \right)_\mathcal{Q}\in \mathbb{G}_m(\mathbf{O} ^*)[IJ_0^2]$, for $i=1,2$.
\item \label{claim:weak.approx.ultraproduct.supp} $\supp(\alpha_1) \smallsetminus \mathcal{Q}$ and $\supp(\alpha_2) \smallsetminus \mathcal{Q}$ are disjoint.
\end{enumerate} 

Let $b_1:=\rho(\alpha_1)_{\supp(\alpha_1) \smallsetminus \mathcal{Q}}$, $b_2:=\rho(\alpha_2)_{\supp(\alpha_1) \smallsetminus \mathcal{Q}}$, $c_1:=\rho(\alpha_2)_{\supp(\alpha_2) \smallsetminus \mathcal{Q}}$, and $c_2:=\rho(\alpha_1)_{\supp(\alpha_2) \smallsetminus \mathcal{Q}}$. By Claim \ref{claim:weak.approx.ultraproduct.supp}, $b_2,c_2\in \mathbb{G}_m(\mathbf{O} ^*)$ and $\supp(b_1)$, $\supp(c_2)$, and $\mathcal{Q}$ are disjoint. Finally,
\[
\left\{ a_1 : a_2 \right\} = \left\{ a_1 : a_2 \right\}_\mathcal{Q} \overset{\ref{claim:weak.approx.ultraproduct.a_i}}{=} \left\{ \rho(\alpha_1) : \rho(\alpha_2) \right\}_\mathcal{Q} \overset{\ref{lem:switcharoo}}{=} \left\{ \rho(\alpha_1) : \rho(\alpha_2) \right\}_{\supp(\alpha_1)\cup \supp(\alpha_2) \smallsetminus \mathcal{Q}}\overset{\ref{claim:weak.approx.ultraproduct.supp}}{=}
\]
\[
=\left\{ \rho(\alpha_1) : \rho(\alpha_2) \right\}_{\supp(\alpha_1)\smallsetminus \mathcal{Q}} \cdot \left\{ \rho(\alpha_1) : \rho(\alpha_2) \right\}_{\supp(\alpha_2) \smallsetminus \mathcal{Q}}=\left\{ b_1 : b_2 \right\}_{\supp(b_1)} \cdot \left\{ c_2 : c_1 \right\}_{\supp(c_1)}.
\]
\end{proof} 

$D_2=D_3$ follows from the following lemma:

\begin{lemma} \label{lem:approx.single.support.ultraproduct} Let $a_1\in \mathbb{G}_m(\mathbf{A}^*)$, let $a_2\in \mathbb{G}_m(\mathbf{O} ^*)$, and let $\mathcal{Q} \subseteq \mathcal{P} ^{K^*}$ be an internal set disjoint from $\mathcal{P} ^{K^*}(2J_0)$. Then there is a prime ideal $Q\in \mathcal{P}^{K^*} \smallsetminus (\mathcal{P} ^{K} \cup \mathcal{P} ^{K^*}(J_0))$ and elements $b_1\in \mathbb{G}_m(\mathbf{A} ^*), b_2\in \mathbb{G}_m(\mathbf{O} ^*)$ such that $\val_Q(b_1)=-1$ and $\left\{ a_1 : a_2 \right\}_\mathcal{Q} = \left\{ b_1 : b_2 \right\}_{Q}$.
\end{lemma} 

\begin{proof} By Proposition \ref{prop:Steinberg.symbols}, we can assume that $\mathcal{Q}$ is small. Since $\mathcal{Q}$ is disjoint from $\mathcal{P}^{K^*}(2J_0)$, we can further assume that $(a_1)_{\mathcal{P} ^{K^*}(J_0)}\in \mathbb{G}_m(\mathbf{O} ^*)$ and $a_2\in \mathbb{G}_m(\mathbf{O} ^*)[2J_0^2]$. By Proposition \ref{prop:Steinberg.symbols}, $\left\{ a_1 : a_2 \right\}_{\mathcal{P} ^{K^*}(J_0)}=1$. 

By Proposition \ref{prop:Steinberg.symbols}\eqref{prop:Steinberg.symbols.continuity} there is an ideal $I\in \mathcal{I} ^{K^*}(\mathcal{Q}\cup \mathcal{P} ^{K^*}(J_0))$ such that if $x_1,x_2\in \mathbb{G}_m(\mathbf{A} ^*)$ satisfy $(a_1 x_1 ^{-1} )_{\mathcal{Q}\cup \mathcal{P} ^{K^*}(J_0)}\in \mathbb{G}_m(\mathbf{O} ^*)$ and $(a_2 x_2 ^{-1} )_{\mathcal{Q}\cup \mathcal{P} ^{K^*}(J_0)} \in \mathbb{G}_m(\mathbf{O} ^*)[I]$, then $\left\{ a_1 : a_2 \right\}_{\mathcal{Q}\cup \mathcal{P} ^{K^*}(J_0)} = \left\{ x_1 : x_2 \right\}_{\mathcal{Q}\cup \mathcal{P} ^{K^*}(J_0)}$. Without loss of generality we can assume that $\mathcal{P} ^{K^*}(I)=\mathcal{Q}\cup \mathcal{P} ^{K^*}(J_0)$.

Write $a_1=[a_{1,n}],a_2=[a_{2,n}],I=[I_n]$. Applying Lemma \ref{lem:approx.single.support.standard} to $a_{1,n},a_{2,n},n! \cdot J_{0,n}^2 \cdot I_n$ and using a diagonal argument, there is an ideal $Q\in \mathcal{P} ^{K^*} \smallsetminus (\mathcal{P} ^K \cup \mathcal{P} ^{K^*}(J_0))$, an element $t\in O_+^*$, a trivialization $\rho$ of $T_t$, and elements $\alpha_1\in T_t(K^*)$, $\alpha_2\in T_t(O^*)$ such that \begin{enumerate} [label=\alph*.]
\item \label{claim:approx.single.support.ultraproduct.a1} $\left( a_1 \rho(\alpha_1) ^{-1} \right)_{\mathcal{Q}\cup \mathcal{P} ^{K^*}(J_0)}\in \mathbb{G}_m(\mathbf{O})$.
\item \label{claim:approx.single.support.ultraproduct.a2} $\left( a_2 \rho(\alpha_2) ^{-1} \right)_{\mathcal{Q}\cup \mathcal{P} ^{K^*}(J_0)}\in \mathbb{G}_m(\mathbf{O})[IJ_0^2]$.
\item \label{claim:approx.single.support.ultraproduct.supp} $\supp(\alpha_1) \smallsetminus \mathcal{Q}=\left\{ Q \right\}$ and $\val_Q(\alpha_1)=-1$.
\end{enumerate} 
Thus,
\[
\left\{ a_1 : a_2 \right\}_\mathcal{Q} = \left\{ a_1 : a_2 \right\}_{\mathcal{Q} \cup \mathcal{P} ^{K^*}(J_0)}\overset{\ref{claim:approx.single.support.ultraproduct.a1},\ref{claim:approx.single.support.ultraproduct.a2}}{=}\left\{ \rho(\alpha_1) : \rho(\alpha_2) \right\}_{\mathcal{Q}\cup \mathcal{P} ^{K^*}(J_0)} \overset{\ref{lem:switcharoo}}{=} \left\{ \rho_(\alpha_1) : \rho(\alpha_2) \right\}_{\supp(\alpha_1) \smallsetminus \mathcal{Q}}=
\]
\[
=\left\{ \rho(\alpha_1) : \rho(\alpha_2) \right\}_Q.
\]
\end{proof} 

\subsubsection{$D_1/D_2$ is finite}

This follows from the following lemma:

\begin{lemma} Let $k:=| \mathbb{G}_m(\mathbf{O} ^*) / \mathbb{G}_m(\mathbf{O} ^*)[2] |<2^{[K:\mathbb{Q}]}$. Then $|D_1/D_2| \leq k^{2k}$.
\end{lemma} 

\begin{proof} Let $x_1,\ldots,x_k$ be representatives of $\mathbb{G}_m(\mathbf{O} ^*) / \mathbb{G}_m(\mathbf{O} ^*)[2]$. For every $a_1\in \mathbb{G}_m(\mathbf{A} ^*)$ such that $\supp(a_1)\cap \mathcal{P} ^{K^*}(2)=\emptyset$ and every $a_2\in \mathbb{G}_m(\mathbf{O} ^*)$ there are $i,j$ such that $\left\{ a_1 : a_2 \right\} \in \left\{ x_i : x_j \right\}_{\mathcal{P} ^{K^*}(2)} \cdot D_2$. Thus, every element of $D_1/D_2$ can be written as a product of cosets of the form $\left\{ x_i : x_j \right\}_{\mathcal{P} ^{K^*}(2)} \cdot D_2$. Using the bilinearity relations, it can also be written as a product of at most $k$ such cosets.
\end{proof} 

\subsubsection{$D_3$ has finite exponent}

This follows from the following proposition:

\begin{proposition} \label{prop:D.bounded.exponent} Assume the GRH. There is a natural number $M$ such that if $a_1\in \mathbb{G}_m(\mathbf{A}^*)$, $a_2\in \mathbb{G}_m(\mathbf{O}^*)[2]$, and $\mathcal{Q} \subseteq \mathcal{P}_{K^*}$ is a small internal set disjoint from $\mathcal{P}^{K^*}(2J_0)$, then $\left\{ a_1 : a_2 \right\}_\mathcal{Q} ^M=1$.
\end{proposition} 

\begin{proof} We will show that the claim holds with $M=N_1N_2$, where $N_1,N_2$ are the constants provided by Lemma \ref{lem:approximation.divisible}.

We can assume, without loss of generality, that $2|J_0$, that $\left( a_1 \right)_{\mathcal{P} ^{K^*}(J_0)}\in \mathbb{G}_m(\mathbf{O} ^*)$, and that $a_2\in \mathbb{G}_m(\mathbf{O} ^*)[J_0^2]$. 

By Proposition \ref{prop:Steinberg.symbols}\eqref{prop:Steinberg.symbols.continuity}, there is an ideal $I\in \mathcal{I}^{K^*}(\mathcal{Q})$ such that if $b_1,b_2\in \mathbb{G}_m(\mathbf{A}^*)$ satisfy $\left( a_1^{N_1} b_1 ^{-1} \right)_{\mathcal{Q}} \in \mathbb{G}_m(\mathbf{O}^*)$ and $ \left( a_2 b_2 ^{-1} \right)_{\mathcal{Q}} \in \mathbb{G}_m(\mathbf{O}^*)[I]$, then $\left\{ a_1^{N_1} : a_2 \right\}_{\mathcal{Q}} = \left\{ b_1 : b_2 \right\}_{\mathcal{Q}}$. We can assume, without loss of generality, that $\mathcal{P}_{K^*}(I)=\mathcal{Q}$.

Let $Q_1,Q_2,t,\rho,\alpha_1,\alpha_2,s,\tau,\beta_1^{(1)},\beta_1^{(2)},\beta_2$ be the elements obtained by applying Lemma \ref{lem:approximation.divisible} and a coordinate-wise argument to the elements $a_1,a_2$ and the ideals $IJ_0^2,J_0$. By \ref{lem:approximation.divisible}\eqref{cond:approximation.divisible.alpha1}, $\rho(\alpha_1)_{\mathcal{P} ^{K^*}(J_0)}\in \mathbb{G}_m(\mathbf{O} ^*)$; by \ref{lem:approximation.divisible}\eqref{cond:approximation.divisible.alpha2}, $\rho(\alpha_2)_{\mathcal{P} ^{K^*}(J_0)}\in \mathbb{G}_m(\mathbf{O} ^*)[J_0^2]$. Thus, by Proposition \ref{prop:Steinberg.symbols}, $\left\{ \rho(\alpha_1) : \rho(\alpha_2) \right\}_{\mathcal{P} ^{K^*}(J_0)}=1$.

We have
\[
\left\{ a_1^{N_1} : a_2 \right\}_\mathcal{Q} \overset{\ref{lem:approximation.divisible},\eqref{cond:approximation.divisible.alpha1}\eqref{cond:approximation.divisible.alpha2}}{=} \left\{ \rho(\alpha_1) : \rho(\alpha_2) \right\}_\mathcal{Q} = \left\{ \rho(\alpha_1) : \rho(\alpha_2) \right\}_{\mathcal{Q}\cup \mathcal{P}^{K^*}(J_0)} \overset{\ref{lem:switcharoo}}{=} 
\]
\[
=\left\{ \rho(\alpha_1) : \rho(\alpha_2) \right\}_{\supp(\alpha_1) \smallsetminus \left( \mathcal{Q}\cup \mathcal{P}^{K^*}(J_0) \right) }\overset{\ref{lem:approximation.divisible}\eqref{cond:approximation.divisible.alpha1}}{=}\left\{ \rho(\alpha_1) : \rho(\alpha_2) \right\}_{Q_1} \cdot \left\{ \rho(\alpha_1) : \rho(\alpha_2) \right\}_{Q_2}.
\]

By Proposition \ref{prop:Steinberg.symbols} and Lemma \ref{lem:switcharoo}, for $i=1,2$,
\[
\left\{ \rho(\alpha_1) : \rho(\alpha_2)^{N_2} \right\}_{Q_i} \overset{\ref{lem:approximation.divisible},\eqref{cond:approximation.divisible.beta1}\eqref{cond:approximation.divisible.beta2}}{=} \left\{ \tau(\beta^{(i)}_1) : \tau(\beta_2) \right\}_{Q_i} \overset{\ref{lem:approximation.divisible}\eqref{cond:approximation.divisible.beta2}}{=} \left\{ \tau(\beta^{(i)}_1) : \tau(\beta_2) \right\}_{\supp(\beta^{(i)}_1) \smallsetminus \left\{ Q_i \right\}}\overset{\ref{lem:approximation.divisible}\eqref{cond:approximation.divisible.beta1}}=1.
\]

Thus,

\[
\left\{ a_1 : a_2 \right\}_\mathcal{Q} ^{N_1N_2}=\left\{ a_1^{N_1} : a_2 \right\}_\mathcal{Q} ^{N_2}=\left\{ \rho(\alpha_1) : \rho(\alpha_2) \right\}_{Q_1}^{N_2} \cdot \left\{ \rho(\alpha_1) : \rho(\alpha_2) \right\}_{Q_2}^{N_2}=
\]
\[
=\left\{ \rho(\alpha_1) : \rho(\alpha_2)^{N_2} \right\}_{Q_1}\cdot \left\{ \rho(\alpha_1) : \rho(\alpha_2)^{N_2} \right\}_{Q_2}=1.
\]
\end{proof} 

\subsubsection{For every $M$, $[D_3:\,^MD_3] \leq M$}

\begin{definition} If $Q\in \mathcal{P}^{K^*}$ is a prime ideal, $\ell$ is a prime number, and $a_1,a_2\in \mathbb{G}_m(\mathbf{A} ^*)$, let $n_{Q,\ell}(a_1,a_2)=n_{Q,\ell}(\pi_{Q}(a_1),\pi_{Q}(a_2))$, where $\pi_Q$ is the projection $\mathbf{A}^* \rightarrow K^*_{Q}$.
\end{definition} 

\begin{lemma} \label{lem:ell.divisible} Let $\ell$ be a prime number such that $\mu_\ell(K)=1$. Then $D_3$ is $\ell$-divisible.
\end{lemma} 

\begin{proof} It is enough to show that every generator of $D_3$ is an $\ell$th power. Let $Q\in \mathcal{P} ^{K^*} \smallsetminus (\mathcal{P} ^K \cup \mathcal{P} ^{K^*}(J_0)$, let $a_1\in \mathbb{G}_m(\mathbf{A} ^*)$ be such that $\val_Q(a_1)=-1$, and let $a_2\in \mathbb{G}_m(\mathbf{O} ^*)[2]$. By Proposition \ref{prop:Steinberg.symbols}\eqref{prop:Steinberg.symbols.continuity}, there is an ideal $I\in \mathcal{I} ^{K^*}(\left\{ Q \right\})$ such that $\left\{ a_1 : a_2 \right\}_Q = \left\{ b_1 : b_2 \right\}_Q$ whenever $a_1 b_1 ^{-1} \in \mathbb{G}_m(\mathbf{O} ^*)$ and $a_2 b_2 ^{-1} \in \mathbb{G}_m(\mathbf{O} ^*)[I]$. 

By Lemma \ref{lem:divisible.no.roots.of.1} and a coordinate-wise argument similar to the one in the proof of Lemma \ref{lem:approx.single.support.ultraproduct}, there are ideals $Q_1,Q_2\in \mathcal{P}^{K^*} \smallsetminus \left( \mathcal{P}^{K^*}(IJ_0\ell) \cup \mathcal{P} ^K \right)$, an element $t\in O_+^*$, a trivialization $\rho$ of $T_t$, and elements $\alpha_1\in T_t(K^*),\alpha_2\in T_t(O^*),\beta\in T_t(\mathbf{O}^*)$ such that 
\begin{enumerate}[label=\alph*]
\item \label{lem:divisible.no.roots.of.1.alpha.1.ultraproduct} $\supp(\alpha_1) \subseteq \left\{ Q,Q_1,Q_2 \right\}$, $\left( a_1 \rho(\alpha_1) ^{-1} \right)_{Q} \in \mathbb{G}_m(\mathbf{O}^*)$, $\val_{Q_1}(\rho(\alpha_1))=\val_{Q_2}(\rho(\alpha_1))=-1$, and $\left( \rho(\alpha_1) \right)_{\mathcal{P} ^K(J_0)}\in \mathbb{G}_m(\mathbf{O}^*)[J_0]$.
\item \label{lem:divisible.no.roots.of.1.alpha.2.ultraproduct} $\left( a_2 \rho(\alpha_2) ^{-1} \right)_{Q} \in \mathbb{G}_m(\mathbf{O}^*)[I]$.
\item \label{lem:divisible.no.roots.of.1.beta.ultraproduct} $\rho(\alpha_2)_{\left\{ Q_1,Q_2 \right\}} \beta ^{-\ell}\in \mathbb{G}_m(\mathbf{O}^*)[Q_1Q_2]$.
\end{enumerate}
We have
\[
\left\{ a_1 : a_2 \right\}_Q=\left\{ \rho(\alpha_1) : \rho(\alpha_2) \right\}_Q\overset{\ref{lem:switcharoo}}{=} \left\{ \rho(\alpha_1) : \rho(\alpha_2) \right\}_{Q_1} \cdot \left\{ \rho(\alpha_1) : \rho(\alpha_2) \right\}_{Q_2}\overset{\ref{prop:Steinberg.symbols}\eqref{prop:Steinberg.symbols.val.-1}}{=}
\]
\[
=\left\{ \rho(\alpha_1) : \beta^\ell \right\}_{Q_1} \cdot \left\{ \rho(\alpha_1) : \beta^\ell \right\}_{Q_2}.
\]
\end{proof}

\begin{lemma} \label{lem:trivial.n.trivial.Steinberg} Let $\ell$ be a prime number, let $Q\in \mathcal{P}^{K^*} \smallsetminus \mathcal{P}^{K^*}(2J_0 \ell)$, and let $a\in \mathbb{G}_m(\mathbf{A} ^*)$ and $b\in \mathbb{G}_m(\mathbf{O}^*)$ be such that $\val_Q(a)=-1$ and $n_{\ell,Q}(a,b)=1$. Then $\left\{ a : b \right\}_{Q}\in D_3^\ell$.
\end{lemma} 

\begin{proof} We can assume that $\left( a \right)_{\mathcal{P} ^{K^*} \smallsetminus \left\{ Q \right\}}=\left( b \right)_{\mathcal{P} ^{K^*} \smallsetminus \left\{ Q \right\}}=1$. Let $\pi\in \mathbb{G}_m(K^*_Q)$ be a uniformizer, and write $a=\pi ^{-1}  a_0$, where $a_0\in \mathbb{G}_m(\mathbf{O} ^*)$. By assumption,
\[
1=n_{\ell,Q}(a,b)=n_{\ell,Q}(\pi,b) \cdot n_{\ell,Q}(a_0,b)=n_{\ell,Q}(\pi,b).
\]
Thus, $b$ is an $\ell$-th power, $b=b_0^\ell$, and
\[
\left\{ a : b \right\}_{Q}=\left\{ \pi : b \right\}_Q \cdot \left\{ a_0 : b\right\}_Q=\left\{ \pi : b_0 \right\}_Q^{\ell}.
\]
\end{proof}

\begin{lemma} \label{lem:Moore} Let $\ell$ be a prime number such that $\mu_\ell(K)\neq 1$, let $Q_1,Q_2\in \mathcal{P} ^{K^*} \smallsetminus \mathcal{P} ^{K^*}(J_0 \ell)$, and let $a_1\in \mathbb{G}_m(\mathbf{A} ^*)$ and $a_2\in \mathbb{G}_m(\mathbf{O} ^*)$. Assume that $\val_{Q_1}(a_1)=-1$. Then there exist $b_1\in \mathbb{G}_m(\mathbf{A} ^*)$ and $b_2\in \mathbb{G}_m(\mathbf{O} ^*)$ such that $\val_{Q_2}(b_1)=-1$ and $\left\{ a_1 : a_2 \right\}_{Q_1} \in \left\{ b_1 : b_2 \right\}_{Q_2} \cdot \,^\ell D_3$.
\end{lemma} 

\begin{proof} We can assume that $Q_1\neq Q_2$. By changing $(a_1)_{Q_2}$ and $(a_2)_{Q_2}$, we can assume that $\val_{Q_2}(a_2)=-1$ and that $n_{\ell,Q_1}(a_1,a_2) \cdot n_{\ell,Q_2}(a_1,a_2)=1$. In this case, we show that the claim holds for $b_1=a_1$ and $b_2=a_2$.

By Lemma \ref{lem:inverse.reciprocity} and a coordinate-wise argument similar to the one in the proof of Lemma \ref{lem:approx.single.support.ultraproduct}, there is an element $t\in O_+^*$, a trivialization $\rho$ of $T_t$, elements $\alpha_1\in T_t(K^*)$, $\alpha_2\in T_t(O^*)$, and a prime ideal $Q_3\in \mathcal{P} ^{K^*} \smallsetminus (\mathcal{P} ^K \cup \mathcal{P} ^{K^*}(Q_1Q_2J_0)$ such that 
\begin{enumerate}[label=\alph*.]
\item $a_2 \rho(\alpha_2) ^{-1} \in \mathbb{G}_m(\mathbf{O} ^*)[Q_1Q_2]$.
\item $\supp(\rho(\alpha_1)) \subseteq \left\{ Q_1,Q_2,Q_3 \right\}$, $(a_1 \rho(\alpha_1) ^{-1})_{\mathcal{P} ^{K^*}(J_0) \cup \left\{ Q_1,Q_2 \right\} }\in \mathbb{G}_m(\mathbf{O} ^*)[J_0]$, and $\val_{Q_3}(\rho(\alpha_1))=-1$.
\item $n_{\ell,Q_3}(\rho(\alpha_1),\rho(\alpha_2))=1$.
\end{enumerate} 
We have
\[
\left\{ a_1 : a_2 \right\}_{Q_1}\overset{\ref{prop:Steinberg.symbols}\eqref{prop:Steinberg.symbols.val.-1}}{=}\left\{ \rho(\alpha_1) : \rho(\alpha_2) \right\}_{Q_1}\overset{\ref{lem:switcharoo}}{=}\left\{ \rho(\alpha_1) : \rho(\alpha_2) \right\}_{\left\{ Q_2,Q_3 \right\}}\overset{\ref{prop:Steinberg.symbols}}{=}\left\{ a_1 : a_2 \right\}_{Q_2} \cdot \left\{ \rho(\alpha_1) : \rho(\alpha_2) \right\}_{Q_3}.
\]
By Lemma \ref{lem:trivial.n.trivial.Steinberg}, $\left\{ \rho(\alpha_1) : \rho(\alpha_2) \right\}_{Q_3}\in \,^\ell D_3$ and the claim follows.
\end{proof} 

\begin{corollary} For every $M$, $[D_3:\,^MD_3] \leq M$.
\end{corollary} 

\begin{proof} Induction on $M$. Let $\ell | M$. Then $[D_3:\,^MD_3] \leq [D_3:\,^\ell D_3] \cdot [D_3:\,^{\frac{M}{\ell}}D_3]$, so, by induction, it is enough to prove that $[D_3:\,^\ell D_3] \leq \ell$.

If $\mu_\ell(K)=1$ this is Lemma \ref{lem:ell.divisible}. Else, fix a $Q$. By \ref{lem:Moore}, $D_3=D_Q \cdot \,^{\ell}D_3$. By Lemma \ref{lem:trivial.n.trivial.Steinberg}, $|D_Q \cdot \,^{\ell}D_3 / ^{\ell}D_3| \leq \ell$. Thus $|D_3/\,^{\ell}D_3| \leq \ell$.
\end{proof}

\section{Proof of Theorem \ref{thm:main2}}

\subsection{Preliminary lemmas} We will use two lemmas in the proof of Theorem \ref{thm:main2}. The first is a version of Rapinchuk's lemma (see \cite[\S3]{PlRa93} or \cite[Lemma 2.6]{Lub95}):

\begin{lemma} \label{lem:Rapinchuk} Suppose that $\Gamma$ does not satisfy the Congruence Subgroup Property. Then, there is a surjection $\widehat{\Gamma} \twoheadrightarrow H$, a finite simple group $F$, and a short exact sequence $$1 \rightarrow F^\infty \rightarrow H \rightarrow X \rightarrow 1$$ such that $\Gamma$ acts faithfully on $F^\infty$. If $F$ is abelian, then there are $v_0\notin S$ and a subgroup $L \subseteq X$ that is commensurable to $G(O_{v_0})$ such that $\left[ \rho(X) : \rho(L) \right]<\infty$, where $\rho:X \rightarrow \Aut(F^\infty)$ is induced by the conjugation.
\end{lemma}

\begin{proof} Consider the short exact sequence $1 \rightarrow C \rightarrow \widehat{G(K)} \rightarrow G(\mathbb{A}_S) \rightarrow 1$ from \S\ref{subsec:sketch} and note that $C \subseteq \widehat{G(O)}$. By assumption, $C$ is infinite. Let $C_0=\left[ \widehat{G(K)},C \right]$. Since
\[
1 \rightarrow C/C_0 \rightarrow \widehat{G(K)}/C_0 \rightarrow G(\mathbb{A}_S) \rightarrow 1
\]
is a central metaplectic extension, we get that $C/C_0$ is finite.

Since $C$ is infinite, there is a normal subgroup $C_1 \triangleleft C_0$ such that $F:=C_0/C_1$ is a finite simple group. Let $L=\bigcap_{g\in \widehat{G(K)}} g ^{-1} C_1 g$. Clearly, $L$ is a normal subgroup of $\widehat{G(K)}$. If the conjugation action of $\widehat{G(K)}$ on $C/L$ was trivial, the extension
\begin{equation} \label{eq:Rapinchuk.0}
1 \rightarrow C/L \rightarrow \widehat{G(K)}/L \rightarrow G(\mathbb{A}_S) \rightarrow 1
\end{equation}
would be central, contradicting $L \subsetneq C_0$. Hence, the conjugation action of $\widehat{G(K)}$ on $C/L$ is nontrivial. Since $\widehat{G(K)}$ has no nontrivial finite quotients, it follows that $C/L$ is infinite. Because $C_0/L$ is both infinite and embeds in $F^\infty$, we get that $C_0/L \cong F^\infty$.

Consider the short exact sequence
\begin{equation} \label{eq:Rapinchuk.1}
1 \rightarrow C_0/L\cong F^\infty \rightarrow \widehat{G(K)}/L \rightarrow \widehat{G(K)}/C_0 \rightarrow 1.
\end{equation}

We claim that \eqref{eq:Rapinchuk.1} is not central (this is clear if $F$ is non-abelian). Indeed, if it were central, then $\widehat{G(K)}$ would act trivially on $C_0/L$ and on $C/C_0$. For any $z\in C/L$, the map $\widehat{G(K)} \rightarrow C_0/L$ given by $g \mapsto [z,g]$ would be a homomorphism. Since $\widehat{G(K)}$ has no finite quotients, we would get that $C/L$ is central in $\widehat{G(K)}$, which contradicts $L \subsetneq C_0$.

If $F$ is non-abelian, the short exact sequence 
\[
1 \rightarrow F^\infty \rightarrow \widehat{G(O)}/L \rightarrow \widehat{G(O)}/C_0 \rightarrow 1,
\]
contained in \eqref{eq:Rapinchuk.1}, satisfies the requirements of the lemma: if $\Gamma$ does not act faithfully on $F^\infty$, then the action of $G(K)$ on $F^\infty$ in \eqref{eq:Rapinchuk.1} is not faithful, hence (by the simplicity of $G(K)$) trivial, contradicting the claim that \eqref{eq:Rapinchuk.1} is not central.



Now suppose that $F=C_p$. The group $\widetilde{\widetilde{G}}:=\widehat{G(K)}/C_0$ is a central extension of $G(\mathbb{A}_S)=\widehat{G(K)}/C$ by the finite group $D:=C/C_0$. We denote $U:=\prod_{v\notin S}G(O_v)$ and, for a subgroup $H \subseteq G(\mathbb{A}_S)$, we denote the pre-image of $H$ in $\widetilde{\widetilde{G}}$ by $\widetilde{\widetilde{H}}$.

Returning to \eqref{eq:Rapinchuk.1}, we claim that there is a $\widetilde{\widetilde{U}}$-invariant subgroup $V \subseteq C_0/L$ of finite index. Indeed, $\widetilde{\widetilde{G}}$ acts continuously on $C_0/L$ and, therefore, acts continuously on the Pontrjagin dual $\left( C_0/L \right) ^\vee$, which is a discrete space. For any $\xi \in \left( C_0/L \right) ^\vee$, the orbit $\widetilde{\widetilde{U}} \cdot \xi$ is compact and hence finite. Taking the intersection of the kernels of the elements in $\widetilde{\widetilde{U}} \cdot \xi$, we get the required subgroup $V$.

We now claim that there is a valuation $v_0\notin S$ such that the orbit of $V$ under $\widetilde{G(K_{v_0})}$ is infinite. Indeed, since $G(K)$ acts faithfully on $C_0/L$, by a similar argument to the argument we used to prove that \eqref{eq:Rapinchuk.1} is not central, the set $\left\{ g^{-1} V g \mid g\in \widetilde{\widetilde{G}}\right\}$ is infinite. By continuity, the set-wise  stabilizer $St({\widetilde{\widetilde{G}}},V)$ of $V$ in $\widetilde{\widetilde{G}}$ is open,   so $St({\widetilde{\widetilde{G(K_v)}}},V)=St({\widetilde{\widetilde{G}}},V) \cap \widetilde{\widetilde{G(K_v)}}$ is open. Since  $G(K_v) $ does not have open finite index subgroups, if $St({\widetilde{\widetilde{G(K_v)}}},V) $ does  not of infinite index in $\widetilde{\widetilde{G(K_v)}}$, its index divides $|D|$. Since $G(K) $  is generated by its $|D|^{th}$ powers and is dense in  $\widetilde{\widetilde{G}}$, it follows that there is $v_0\notin S$ such that $\widetilde{\widetilde{G(K_{v_0})}} \cdot V$ is infinite.

Let $W=\bigcap_{g\in \widetilde{\widetilde{G(K_{v_0})}}} g^{-1} V g$. We will show that the short exact sequence
\begin{equation} \label{eq:Rapinchuk.Cp}
1 \rightarrow C_0/W \rightarrow \widehat{G(O)}/W \rightarrow \widetilde{\widetilde{U}} \rightarrow 1
\end{equation}
satisfies the claims of the lemma.

Firstly, $\widetilde{\widetilde{G(K_{v_0})}} \cdot V$ is infinite, so $[C_0:W]=\infty$. It follows that $C_0/W \cong C_p^\infty$.

Next, $\Gamma$ acts faithfully in \eqref{eq:Rapinchuk.Cp}. Indeed, \eqref{eq:Rapinchuk.Cp} extends to the short exact sequence
\begin{equation} \label{eq:Rapinchuk.Kv0}
1 \rightarrow C_0/W\cong C_p^\infty \rightarrow Z/W \rightarrow \widetilde{\widetilde{U}} \cdot \widetilde{\widetilde{G(K_{v_0})}} \rightarrow 1,
\end{equation}
where $Z \subseteq \widehat{G(K)}/L$ is the preimage of $\widetilde{\widetilde{U}} \cdot \widetilde{\widetilde{G(K_{v_0})}}$. If $\gamma \in \Gamma \smallsetminus \left\{ 1 \right\}$ fixes $C_0/W$ in \eqref{eq:Rapinchuk.Cp}, then $G(K)$ would fix $C_0/W$ in \eqref{eq:Rapinchuk.Kv0}, a contradiction. 

Finally, let $U'=\prod_{v\neq v_0} G(O_v)$. There is an open subgroup $M$ of $U'$ for which \begin{enumerate}
\item There is a section $s:M \rightarrow \widetilde{\widetilde{M}}$ of the central extension $1 \rightarrow D \rightarrow \widetilde{\widetilde{M}} \rightarrow M \rightarrow 1$ (because $D$ is finite).
\item $s(M)$ acts trivially on $C_0/V$ (because $C_0/V$ is finite).
\item $s(M)$ commutes with $\widetilde{\widetilde{G(K_{v_0})}}$ (because $\widetilde{\widetilde{G(K_{v_0})}}$ is compactly generated and, for every compact set $X \subseteq \widetilde{\widetilde{G(K_{v_0})}}$, $[s(M),X]=1$, if $M$ is small enough).
\end{enumerate} 
It follows that $s(M)$ fixes $C_0/W$, as required by the lemma.
\end{proof} 

\begin{lemma} \label{lem:small.p.adic}  Let $\mathbf{G}$ be a non-isotropic semisimple group over a non-archimedean local field $F$, let $L \subseteq \mathbf{G}(F)$ be a compact open subgroup, and let $p$ be a prime number. There is a constant $D$ such that, for every homomorphism $\rho: L \rightarrow \GL_N(\mathbb{F}_p)$, $|\rho(L)| \leq D N^{\dim \mathbf{G}}$.
\end{lemma} 

\begin{proof} Without loss of generality, we can assume that $F=\mathbb{Q}_\ell$ and that $L \subset \mathbf{G}(\mathbb{Z}_\ell)[\ell]$. For an open subgroup $H \subseteq \mathbf{G}(\mathbb{Z}_\ell)$, we denote $\delta(H)=\min \left\{ k \mid \mathbf{G}(\mathbb{Z}_\ell)[\ell^{k}] \subseteq H\right\}$. For any normal and open subgroup $H \triangleleft L$, 
\begin{equation} \label{eq:normal.in.L}
\mathbf{G}(\mathbb{Z}_\ell)[\ell^{\delta(H)}] \subseteq H \subseteq \mathbf{G}(\mathbb{Z}_\ell)[\ell^{\delta(H)-\delta(L)}].
\end{equation}

We distinguish between two cases:
\begin{enumerate}
\item[Case 1:] $\ell=p$. $L$ is a pro-$p$ group, so, for every $g\in L$, $\rho(g)$ is nilpotent and, hence, $\rho(g)^{p^M}=1$, where $M=\lfloor \log_p(N)\rfloor$. Denote the word $x^{p^M}$ by $w$. Since $\mathbf{G}(\mathbb{Z}_p)[p^{\delta(L)}]$ is a powerful pro-$p$ group, 
\[
\ker(\rho) \supseteq \left\langle w(L) \right\rangle \supseteq \left\langle w\left( G(\mathbb{Z}_p)[p^{\delta(L)}] \right) \right\rangle \supseteq G(\mathbb{Z}_p)[p^{\delta(L)+M}],
\]
so
\[
|\rho(L)| \leq \left[ G(\mathbb{Z}_p): G(\mathbb{Z}_p)[p^{\delta(L)+M}]\right] \leq \left[ G(\mathbb{Z}_p): G(\mathbb{Z}_p)[p^{\delta(L)}]\right] \cdot N^{\dim \mathbf{G}}.
\]

\item[Case 2:] $\ell\neq p$. Let $\mathbf{T} \cong \mathbb{G}_m \subset \mathbf{G}$ be a split torus in $\mathbf{G}$. There is a unipotent element $u\in L$ that is an eigenvector of $\mathbf{T}$: for some character $\lambda:\mathbf{T} \rightarrow \mathbb{G}_m$, we have $t ^{-1} u t=u^{\lambda(t)}$, for every $t\in \mathbf{T}(\mathbb{Z}_\ell)\cap L$. We can assume that $u\notin \mathbf{G}(\mathbb{Z}_\ell)[\ell^{\delta(L)+1}]$. The subgroup $\lambda(\mathbf{T}(\mathbb{Z}_\ell) \cap L)$ has finite index in $\mathbb{Z}_\ell^ \times$. We denote this index by $c$.

Since $\mathbf{G}(\mathbb{Z}_\ell)[\ell]$ is powerful and since $u\notin \mathbf{G}(\mathbb{Z}_\ell)[\ell^{\delta(L)+1}]$, we get that $u^{\ell^{\delta(\ker \rho)-2\delta(L)-1}}\notin \mathbf{G}(\mathbb{Z}_\ell)[\ell^{\delta(\ker\rho)-\delta(L)}]$. By \eqref{eq:normal.in.L}, $u^{\ell^{\delta(\ker \rho)-2\delta(L)-1}}\notin \ker(\rho)$. 

Since $L$ is a pro-$\ell$ group, the order of $\rho(u)$ is a power of $\ell$ and is at least $\ell^{\delta(\ker \rho)-2\delta(L)}$. Since $\ell\neq p$, this means that $\rho(u)$ is semisimple and one of its eigenvalues is a root of 1 of order at least $\ell^{\delta(\ker \rho)-2\delta(L)}$. Using conjugation and by the definition of $c$, $\rho(u)$ has at least $\frac1c \ell^{\delta(\ker \rho)-2\delta(L)}$ different eigenvalues, so $N \geq \frac1c \ell^{\delta(\ker \rho)-2\delta(L)}$. It follows that
\[
| \rho(L) | \leq [\mathbf{G}(\mathbb{Z}_\ell)[\ell]:\mathbf{G}(\mathbb{Z}_\ell)[\ell^{\delta(\ker \rho)}]] \leq \ell^{\delta(\ker\rho) \cdot \dim \mathbf{G}} \leq \left( c \ell^{2\delta(L)}\right)^{\dim G} \cdot N^{\dim \mathbf{G}}.
\]

\end{enumerate} 
\end{proof} 

\subsection{Proof of Theorem \ref{thm:main2}}

\begin{proof}[Proof of Theorem \ref{thm:main2}] The implication $CSP \implies FCW$ follows from \cite[Lemma 3.20]{AM22} and \cite[Lemma 3.22]{AM22}. It remains to prove the converse. 

Without loss of generality, we can assume that $\Gamma=\mathbf{G}(O)$ does not have CSP. Applying Lemma \ref{lem:Rapinchuk}, we get a short exact sequence $1 \rightarrow F^\infty \rightarrow H \rightarrow X \rightarrow 1$ and, if $F$ is abelian, a place $v_0\notin S$ and a subgroup $L \subseteq X$ satisfying the conditions of  Lemma \ref{lem:Rapinchuk}. Denote the characteristic of the residue field of $K_{v_0}$ by $\ell$. We will show that $H$ has infinite conjugacy width and, therefore, so does $\widetilde{\Gamma}$. 

Suppose first that $F$ is non-abelian. Consider the composition 
\[
\widehat{\Gamma} \rightarrow H \rightarrow \Aut \left( F^\infty \right) \cong \Sym(\infty) \ltimes \Aut(F)^\infty \rightarrow \Sym(\infty),
\]
where $\Sym(\infty)$ denotes the group of all permutations of $\mathbb{N}$ with the topology of pointwise convergence.

By compactness, the orbits of $\widehat{\Gamma}$ on $\mathbb{N}$ are finite. Since $\Aut(F)^\infty$ has a finite exponent, the sizes of the orbits of $\widehat{\Gamma}$ are unbounded (otherwise, some non-trivial elements of $\Gamma$ would act trivially). 

Let $O_1,O_2,\ldots \subset \mathbb{N}$ be the $\widehat{\Gamma}$-orbits. Choose representatives $i_j \in O_j$, choose a non-trivial element $f\in F$, and let $v\in F^\infty$ be the element
\[
v_k=\begin{cases} f & k\in \left\{ i_1,i_2, \ldots, \right\} \\ 1 & k\notin \left\{ i_1,i_2, \ldots, \right\} \end{cases}.
\]
We claim that the conjugacy class of $v$ has infinite width in $H$. Indeed, it is easy to see that $\left \langle \gcl_H(v) \right \rangle=F^\infty$, but, for any two natural numbers $C,n$, if $g\in \gcl_H(v) ^C$, then $|\supp(g) \cap O_n| \leq C$. Since $|O_n|$ are unbounded, $\gcl_H(v) ^C \neq F^\infty$, for any $C$.\\

Next, suppose that $F$ is abelian. Then $F=C_p$, for some prime $p$. We first claim that there is an element $u\in C_p^\infty$ such that $\rho(X) \cdot u$ is infinite.  Assume that $\rho(X) \cdot x$ is finite, for every $x\in C_p^\infty$. The function $x \mapsto | \rho(X) \cdot x|$ is unbounded since, otherwise, a non-trivial subgroup of $\Gamma$ would act trivially. For every $x,y\in C_p^\infty$, we have 
\begin{equation} \label{eq:triangle.ineq}
| \rho(X)(x+y) | \leq | \rho(X)x| \cdot | \rho(X)y|.
\end{equation}
It follows that the function $x \mapsto | \rho(X) \cdot x |$ is unbounded on every open subgroup of $C_p^\infty$. We recursively construct a decreasing sequence $W_k$ of $\rho(X)$-invariant open subgroups of $C_p^\infty$ and a sequence $u_k$ of elements of $C_p^\infty$ such that \begin{enumerate}
\item $u_{k+1}\in W_k$.
\item $\bigcap W_k=0$.
\item The $X$-orbit of the element $(u_0+\cdots+u_k)+W_k \in C_p^\infty / W_k$ has size greater than $k$.
\end{enumerate} 
For $k=0$, take $W_0=C_p^\infty,u_k=0$. Given $W_0,\ldots,W_k,u_0,\ldots,u_k$, choose $u_{k+1}\in W_k$ such that $\left| \rho(X)(u_1+\cdots+u_k)\right|(k+1)<\left| \rho(X)(u_{k+1})\right|$. Since $u_{k+1}=(u_1+\ldots+ u_{k+1})-(u_1+\ldots+ u_{k})$,
by \eqref{eq:triangle.ineq}, $\left| \rho(X)(u_1+\cdots+u_{k+1})\right|>k+1$, so there is an open subgroup $W_{k+1}$ such that the orbit $\rho(X)(u_1+\cdots+u_{k+1}+W_{k+1})$ has size greater than $k+1$. Shrinking $W_{k+1}$, we can further require that $W_{k+1} \subseteq W_k$, that the first $k$ coordinates of every element in $W_{k+1}$ are zero (so $\bigcap W_k=0$), and that $W_{k+1}$ is $\rho(X)$-invariant.

It follows from the first two properties that the series $u_1+u_2+\cdots$ converges. Denote the sum by $u$. Since the projection of $\rho(X) \cdot u$ to $C_p^\infty / V_k$ has size greater than $k$, for every $k$, it follows that $\rho(X) \cdot u$ is infinite. We will show that the width of the conjugacy class of $u$ in $H$ is infinite.

Denote $V=\overline{\langle \rho(X) \cdot u \rangle}$. By compactness, there is a decreasing sequence $V=V_0 \supseteq V_1 \supseteq V_2 \supseteq \cdots$ of $\rho(X)$-invariant open subgroups of $V$ such that $\cap V_i=0$. Denote the projection $V \rightarrow V/V_n$ by $\pi_n$ and denote the representation of $X$ on $V/V_n$ by $\rho_n$. We have
\[
\left| \pi_n \left( \langle \rho(X) \cdot u \rangle \right) \right| = \left| \langle \rho_n(X) \cdot \pi_n(u) \rangle \right| = \left| V/V_n \right| \geq 2^{\dim V/V_n},
\]
since $\pi_n(u)$ generates $V/V_n$.

On the other hand, by the assumption on $L$, there are finitely many elements $g_1,\ldots,g_m\in X$ such that $\rho(X)=\cup g_i \rho(L)$. Thus, by Lemma \ref{lem:small.p.adic}, there exists a constant $D$ such that for every $C$, 
\[
\left| \pi_n \left( \left( \rho(X) \cdot u \right)^C \right) \right| = \left| \left( \rho_n(X) \cdot \pi_n(u) \right)^C \right| = \left| \left( \bigcup_1^m \rho_n(g_i) \rho_n(L) \cdot \pi_n(u) \right)^C \right| \leq 
\]
\[
\leq m^C\left| \rho_n(L) \cdot \pi_n(u) \right|^{C} \leq \left( m \left| \rho_n(L)\right|\right)^{C} \leq \left( D \cdot \dim V/V_n \right)^{2C}.
\]

Since $\dim V/V_n$ is unbounded, we get 
\[
\langle \gcl_H(u) \rangle = \langle \rho(X) \cdot u \rangle \not\subseteq \left( \rho(X) \cdot u \right)^C = \gcl_H(u)^C,
\]
so the width of the conjugacy class of $u$ in $H$ is infinite.
\end{proof}

\begin{proposition} \label{prop:Conj.implies.FCW} Suppose that Conjecture \ref{conj:local.to.global.one.group} holds for $(K,S,\mathbf{G},\Gamma)$. Then the pro-finite completion of $\Gamma$ has finite conjugacy width.
\end{proposition} 

\begin{proof} Let $T \subseteq S$ be the set of real places $v\in S$ for which $\mathbf{G}(K_v)$ is compact. If $v\in T$ then $\mathbf{G}(K_v)$ is connected so \cite[Lemma 3.17]{AM22} implies that there is an open set $U_v \subseteq \mathbf{G}(K_v)$ and a natural number $N_v$ such that $\wid_{\mathbf{G}(K_v)}\left( \gcl_{\mathbf{G}(K_v)}(g) \right) < N_v$, for every $g\in U_v$. Let $N=\max \left\{ N_v \right\}$.

Assume that \eqref{eq:local.to.global.conj} holds. We will show that the width of every conjugacy class in $\widehat{\Gamma}$ is bounded by $CN$. For this, it is enough to show that, for every finite index normal subgroup $\Delta \triangleleft \Gamma$ and every $\gamma \in \Gamma$, the width of $\gcl_{\Gamma / \Delta}(\gamma \Delta)$ is bounded by $CN$. 

Let $\Delta$ and $\gamma$ as above. Since $\Delta$ is dense in $\prod_T \mathbf{G}(K_v)$, there is $\delta \in \Delta$ such that $\gamma \delta \in U_v$, for every $v\in T$. This implies that $\wid \left( \gcl_{\mathbf{G}(K_v)}(\gamma \delta) \right) < N$. By \eqref{eq:local.to.global.conj}, $\wid \left( \gcl_\Gamma(\gamma \delta) \right) < CN$, which implies that $\wid \left( \gcl_{\Gamma / \Delta}(\gamma \Delta) \right) <CN$.

\end{proof} 

\section{Proof of Theorem \ref{thm:stability}} \label{sec:stability}

\begin{lemma} \label{lem:c-separated} Let $\Gamma$ be as in Theorem \ref{thm:main}. There is a constant $D>0$ such that if $\gamma \in \Gamma $ is an element of infinite index, then $\langle \gamma \rangle$ contains a $D$-separated element.
\end{lemma}

\begin{proof} One of the eigenvalues of $\gamma$ is not a root of unity; call it $a$. Every $v\in S_{def}$ defines an embedding $\varphi_v:K(a) \hookrightarrow \mathbb{C}$ such that $\varphi_v(a)\in U(1)$. Let $H \subseteq U(1)^{|S_{def}|}$ be the closure of the subgroup generated by the element $(\varphi_v(a))_{v\in S_{def}}$. Since $a$ is not a root of unity, every coordinate projection $p_v:H \rightarrow U(1)$ is onto. Denoting the Haar measure on a group $G$ by $\mu_G$, we get that $\mu_H \left( p_v ^{-1}(A) \right) = \mu_{U(1)}(A)$, for every $A \subseteq U(1)$. It follows that there is an element $h\in H$ such that the angle between $(p_v(h)$ and $\pm 1)$ is greater than $\frac{\pi}{|S_{def}|}$, for every $v\in S_{def}$. Thus, there is $m$ such that the angle between $p_v(\gamma ^m)$ and $\pm1$ is greater than $\frac{\pi}{2|S_{def}|}$, which implies that $\gamma ^m$ is $\frac{\pi}{2|S_{def}|}$-separated.
\end{proof} 

\begin{proof}[Proof of Theorem \ref{thm:stability}] 

Suppose that $f: \Gamma \rightarrow S_N$ is a function such that $d_N(f(xy),f(x)f(y))< \delta$, for every $x,y\in \Gamma$.

Choose $\gamma \in \Gamma$ of infinite order. By \cite[Theorem 1.2]{BeCh}, there is $N \leq N' \leq (1+10^4 \delta)N$ and a homomorphism $\tau : \Gamma \rightarrow S_{N'}$ such that $d_{N'}(f(\gamma ^n),\tau(\gamma ^n))< 10^4 \delta$, for every $n$. In particular, $d_{N'}(f(\gamma ^{m \cdot N'!},1)<10^4 \delta$, for every $m$.

By Lemma \ref{lem:c-separated}, there is $m$ such that $\gamma ^{m \cdot N'!}$ is $c$-separated. By Theorem \ref{thm:main}, the width of $\gcl_\Gamma \left( \gamma ^{m \cdot M!} \right)$ is bounded by $Cc$. Denoting the normal subgroup generated by $\gamma ^{m \cdot M!}$ by $\Delta$, we have $d_{N'}(f(\Delta),1)<10^5Cc \delta$. By \cite[Theorem 2.20]{BeCh}, there is $N'<N''<(1+10^9Cc \delta)N'$ and a homomorphism $\rho:\Gamma \rightarrow S_{N''}$ with $d(f(x),\rho(x))<10^9Cc \delta$, for every $x\in \Gamma$.

\end{proof}

\appendix

\section{Conjugacy classes in spin groups}

\subsection{Setting} \label{subsec:Setting.Appendix}

We keep $K,S,O$ as in \S\ref{sec:notations}. Fix an integer $m \geq 6$ and a nondegenerate quadratic form $q$ on $K^m$. We assume that $q$ is anisotropic, that the $S$-congruence kernel of $\Spin_q$ is trivial, and that there is $w_0\in S$ for which the Witt index of $(K_{w_0}^m,q)$ is at least 2. Finally, let $\Delta \subseteq \Spin_q(K)$ be an $S$-arithmetic subgroup.

\subsection{$\Spin_q(K_v)$}

\begin{notation} \begin{enumerate}
\item $d_{S^{m-1}}$ is the round metric on the unit sphere $S^{m-1} \subseteq \mathbb{R} ^m$ normalized so $d_{S^{m-1}}(x,y)$ is the angle between $x$ and $y$. 
\item $d_{\SO(m)}$ is the metric on $\SO(m)$ given by $d_{\SO(m)}(g,h)=\max \left\{ d_{S^{m-1}}(gx,hx) \mid x\in S^{m-1} \right\}$.
\item $d_{\SO(m)}$ is a Riemannian metric on $\SO(m)$. We denote by $d_{\Spin(m)}$ the induced Riemannian metric on $\Spin(m)$.
\end{enumerate} 
\end{notation} 



\begin{lemma} \label{lem:width.in.K} There is a constant $M$ such that for every place $v$ and every element $g\in \Spin_q(K_v) \smallsetminus Z(\Spin_q(K_v))$, \begin{enumerate}
\item $\gcl_{\Spin_q(K_v)}(g)$ generates $\Spin_q(K_v)$.
\item If $v$ is not real or $v$ is real and $(K_v^m,q)$ is not definite, then $\wid_{\Spin_q(K_v)} \left( \gcl_{\Spin_q(K_v)}(g) \right) \leq M$.
\item If $v$ is real and $(K_v^m,q)$ is definite, then
\[
\frac{1}{d_{\Spin_q(K_v)}(g,Z(\Spin_q(K_v))} \leq \wid_{\Spin_q(K_v)}\left( \gcl_{\Spin_q(K_v)}(g)\right) \leq \frac{M}{d_{\Spin_q(K_v)}(g,Z(\Spin_q(K_v)))}.
\]
\end{enumerate} 
\end{lemma} 

\begin{proof} \begin{enumerate}
\item By \cite[Lemma 3.16]{AM22}, the normal subgroup generated by $g$ is open, so it contains all unipotents. These generate $\Spin_q(K_v)$.
\item This follows from \cite[Lemma 3.18]{AM22} if $\Spin_q(K_v)$ is not compact and from \cite[Lemma 3.20]{AM22} if $\Spin_q(K_v)$ is compact.
\item In this case, $\Spin_q(K_v)=\Spin(m)$. Denote $w=\wid_{\Spin(m)} \left( \gcl_{\Spin(m)}(g) \right)$, $Z=Z(\Spin(m))$, and $d=d_{\Spin(m)} \left( g,Z \right)$.

If $A,B \subseteq \Spin(m)$ then $d_{\Spin(m)}\left( A \cdot B,Z \right) \leq d_{\Spin(m)}\left( A,Z \right) + d_{\Spin(m)}\left(B,Z \right)$. Thus
\[
\frac{\pi}{2} \leq d_{\Spin(m)}\left( \gcl_{\Spin(m)}(g)^w,Z \right) \leq w \cdot d,
\]
and the first inequality follows. 

Denote the image of $g$ in $\SO(m)$ by $\overline{g}$. By \cite[Lemma 3.17]{AM22} there is a constant $c$ such that if $d>\frac{1}{10}$ then $w<c$. If $d<\frac{1}{10}$, then one of the eigenvalues of $\overline{g^2}$ is not in the ball of radius $\frac{d}{2}$ around 1. This implies that, for some $N \leq \frac{2}{d}$, $d(g^N,Z)>\frac{1}{10}$, so $w < \frac{2c}{d}$.
\end{enumerate} 
\end{proof} 

\subsection{$\Spin_q(O)$}

\begin{definition} Let $\epsilon >0$. \begin{enumerate}
\item We say that $g\in \SO_q(K)$ is $\epsilon$-separated if, for every real place $v\in S$ such that $q_v$ is definite, $d_{\SO_q(K_v)} \left( g,Z(\SO_q(K_v)) \right) > \epsilon$.
\item We say that $g\in \Spin_q(K)$ is $\epsilon$-separated if its image in $\SO_q(K)$ is $\epsilon$-separated.
\end{enumerate} 
\end{definition} 

\begin{lemma} \label{lem:Kneser.uff} Under Setting \ref{subsec:Setting.Appendix} there is an integer $N$ such that the following holds: if $c_1,\ldots,c_m \in O^m$ are non-isotropic orthogonal vectors and $6 \leq k \leq m-6$ satisfy
\[
i_q \left( K_{w_0} c_1+\cdots+K_{w_0} c_k \right) \geq 1 \text{ and } i_q \left( K_{w_0} c_k + \cdots +K_{w_0} c_{m}  \right) \geq 2,
\]
then, denoting $U:=K c_1 + \cdots +K c_k$, \begin{enumerate}
\item For every 1-separated $\delta \in \Delta$, the set $\gcl_\Delta(\delta) \cdot (c_1,\ldots,c_k)$ contains an adelic neighborhood of $(c_1,\ldots,c_k)$ inside $\Spin_q(O) \cdot (c_1,\ldots,c_k)$. In particular, there is $I\in \mathcal{I} ^K$ such that $\Delta[I] \subseteq \gcl_{\Delta}(\delta)^{N} \cdot \Spin_{q\restriction U^\perp}(O)$.
\item There is $J\in \mathcal{I} ^K$ such that $\Delta[J] \subseteq \left( \gcl_{\Delta} \left( \Delta \cap \Spin_{q\restriction U^\perp}(O) \right) \right)^{N+1}$.
\end{enumerate} 
\end{lemma} 

\begin{proof} \begin{enumerate}
\item The proof is almost identical to \cite[Lemma 4.16]{AM22} and is omitted; see Remark \ref{rem:4.16.AM22}. 
\item By weak approximation, there is an element $\delta \in \Delta \cap \Spin_{q\restriction U}(O)$ that is 1-separated as an element of $\Spin_{q\restriction U}(O)$ and thus also as an element of $\Spin_q(O)$. By part 1, there is an ideal $J$ such that $\Delta[J] \subseteq \gcl_{\Delta}(\delta)^{N} \cdot \Spin_{q\restriction U^\perp}(O)$. This implies that $\Delta[J] \subseteq \gcl_\Delta(\delta)^N \cdot \left( \Delta \cap \Spin_{q\restriction U^\perp}(O) \right) \subseteq \left( \gcl_{\Delta} \left( \Delta \cap \Spin_{q\restriction U^\perp}(O) \right) \right)^{N+1}$.
\end{enumerate} 
\end{proof} 

\begin{remark} \label{rem:4.16.AM22} The differences between \cite[Lemma 4.16]{AM22} and the first part of Lemma \ref{lem:Kneser.uff} are:
\begin{itemize}
\item \cite[Lemma 4.16]{AM22} is stated for $k=3$. The generalization to $1 \le k \le n-6$ follows from a simple induction.
\item \cite[Lemma 4.16]{AM22} is stated for centerless congruence subgroups of $\Theta_q(K)$. The assumption that the center of the congruence subgroup is trivial is used in a different part of \cite{AM22} and is not needed for \cite[Lemma 4.16]{AM22}. 
\end{itemize}
\end{remark}

\begin{theorem} \label{thm:Kneser.commutators} For every $S$-arithmetic subgroup $\Delta \subseteq \Spin_q(K)$ there is a natural number $N$ and an ideal $J\in \mathcal{I} ^K$ such that, for every 1-separated element $\delta \in \Delta$ there exists $I\in \mathcal{I} ^K$ for which 
\[
\left[ \Spin_q(O)[I] , \Spin_q(O)[J] \right] \subseteq \gcl_{\Delta}(\delta)^N.
\]
\end{theorem} 

\begin{proof} Choose non-isotropic and orthogonal vectors $c_1,\ldots,c_m\in O^m$ such that 
\[
i_q \left( K_{w_0} c_1+\cdots+K_{w_0} c_7 \right) \geq 2 \text{ and } i_q \left( K_{w_0} c_6 + \cdots +K_{w_0} c_{m}  \right) \geq 2.
\]
Denote $U:=K c_1 + \cdots +K c_6$, $\Delta_U:=\Delta \cap \Spin_{q\restriction U}(O)$, and $\Delta_{U^\perp}:=\Delta \cap \Spin_{q\restriction U^\perp}(O)$. Lemma \ref{lem:Kneser.uff} implies that there are $N,J$ such that $\Delta[I] \subseteq \gcl_{\Delta}(\delta)^N \cdot \Delta_{U}$ and $\Delta[J] \subseteq \gcl_{\Delta} \left( \Delta_{U^\perp} \right)^{N+1}$.

Since $\left[ \Delta , \delta \right] \subseteq \gcl_\Delta(\delta)^2$, it follows that $\left[ \Delta , \gcl_\Delta(\delta) \right] \subseteq \gcl_\Delta(\delta)^2$. The identity $[xy,z]=x[y,z]x ^{-1} [x,z]$ implies that 
\[
\left[ \Delta[I] , \Delta_{U^\perp} \right] \subseteq \left[ \Delta_U \cdot \gcl_\Delta(\delta)^N , \Delta_{U^\perp} \right] \subseteq \left[ \gcl_\Delta (\delta) , \Delta \right] ^N \cdot \left[ \Delta_{U} , \Delta_{U^\perp} \right] \subseteq \gcl_\Delta (\delta)^{2N},
\]
so $\left[ \Delta[I] , \gcl_\Delta \left( \Delta_{U^\perp} \right) \right] \subseteq \gcl_\Delta(\delta)^{2N}$. Thus,
\[
\left[ \Delta[I] , \Delta[J] \right] \subseteq \left[ \Delta[I] , \gcl_\Delta \left( \Delta_{U^\perp} \right) ^{N+1} \right] \subseteq \left[ \Delta[I] , \gcl_\Delta \left( \Delta_{U^\perp} \right) \right] ^{N+1} \subseteq \gcl_\Delta(\delta)^{2N(N+1)}.
\]
\end{proof}

\subsection{$\Spin_q(K)$}

The proofs in this subsection mimic the ones of \cite{Knes56}. 

\begin{notation} \begin{enumerate}
\item We denote the image of $\Spin_q(K) \rightarrow \SO_q(K)$ by $\Theta_q(K)$. Equivalently, $\Theta_q(K)$ is the kernel of the Spinor norm $\theta:\SO_q(K) \rightarrow K^ \times / (^2K^ \times)$.
\item If $U \subseteq K^m$ is a subspace, denote $\Theta_{q\restriction U}(K)=\Theta_q(K) \cap \SO_{q\restriction U}(K)$.
\item For every non-isotropic $a \in K^n$, denote the reflection in the hyperplane $a^\perp$ by $\tau_a \in \O_q(K)$. 
\item For $a \in K^n$ denote the image of $a$ in $\Cliff_q$ by $e_a$.
\end{enumerate} 
\end{notation}

Note that, for every non-isotropic $a \in K^m$ and every $b \in K^m$, $\frac{1}{q(a)}e_ae_be'_{a}=e_{-\tau_a(b)}$ where $':\Cliff_q \rightarrow \Cliff_q$ is the canonical involution.  

\begin{definition}
Let $E$ be a field and let $f:E^n \rightarrow E$	a regular quadratic from.
\begin{enumerate}
	\item $f$ is called universal if $f(E^n)=E$. 
	\item A  subspace $V$ of $E^n$ is called a $f$-universal subspace of $E$  if it is regular and $f(V)=f(E^n)$.
\end{enumerate}
\end{definition}

\begin{lemma}\label{lemma:rrr} Let  $U$ be a  $q$-universal subspace of $K^m$ such that $\dim U \ge 5$. Then for every $\alpha \in O_q(K)$, there exist $\beta,\gamma \in O_{q \restriction_U}(K)$ such that $\alpha\beta$ and $\gamma\alpha$ belong to $\Theta_q(K)$.
\end{lemma}
\begin{proof}
Since $(\alpha^{-1}\gamma^{-1})^{-1}=\gamma\alpha$, it is enough to prove that existence of $\beta$. It is clear that there exists $\beta_1 \in O_{q \restriction_U}(K)$ such that $\alpha\beta_1 \in \SO_q(K)$. Let $a$ be the spinor norm of $\alpha\beta_1$. Let $b,c \in K$ be such that $a=bc$ and for every $v \in S_{real}$, $b,c \in q(U_v)$. Since $\dim U \ge 5$, for every non-real $v$, $q$ is universal over $U_v$ so $b,c \in q(U_v)$. The Hasse principle implies that $b,c \in q(U)$, say $b=q(b')$ and $c=q(c')$ for some $b',c'\in K^n$. Let $\beta_2=\tau_{b'}\tau_{c'}$. Then $\alpha\beta_1\beta_2 \in \Theta_q(K)$ and $\beta_1\beta_2 \in O_{q \restriction_U}(K)$.
\end{proof}

\begin{lemma}\label{lemma:equal_gcl} Let $U$ be a  $q$-universal subspace of $K^n$ of dimension at least 5 and let $\alpha \in \Theta_{q \restriction U^\perp}(K)$. Then $\gcl_{\Theta_q(K)}(\alpha)=\gcl_{O_q(K)}(\alpha)$.
\end{lemma}
\begin{proof}
	Let $\beta \in O_q(K)$. We have to show that $\beta \alpha\beta^{-1}\in \gcl_{\Theta_q(K)}(\alpha)$.  Lemma \ref{lemma:rrr} implies that there exists $\gamma \in O_{q \restriction_U}(K)$ such that $\beta\gamma \in \Theta_q(K)$. The claim follows from the fact that $\beta \alpha\beta^{-1}=(\beta \gamma)\alpha(\beta\gamma)^{-1}$.
\end{proof}

\begin{lemma}\label{lemma:eps_to_1} There exists a constant $N$ such that the following holds. For every $\epsilon >0$, every $\epsilon$-separated element $\alpha \in \Theta_q(K)$, and every 6-dimensional regular subspace $U$ of $K^m$,  $\left( \gcl_{\Theta_q(K)}(\alpha) \right) ^{\frac{N}{\epsilon}}\cap \Theta_{q\restriction_U}(K)$ is dense in $\prod_{v \in S\text{ is real}}\Theta_{q \restriction_U}(K_v)$.  In particular, there exists a 1-separated element in $\left( \gcl_{\Theta_q(K)}(\alpha)\right) ^\frac{N}{\epsilon} \cap \Theta_{q\restriction_U}(K)$.
\end{lemma}

\begin{proof} Let $M$ be the constant of Lemma \ref{lem:width.in.K}. We show that the claim holds with $N=2M$. 

Denote the set of real places in $S$ by $S_{real}$. Suppose that $\alpha$ is $\epsilon$-separated. By weak approximation, there is an element $\beta \in \left( \gcl_{\Theta_q(K)}(g) \right)^{\frac{M}{\epsilon}}$ and, for every $v\in S_{real}\cup \left\{ w_0 \right\}$, an element $a_v \in K_v^m$ such that \begin{enumerate}
\item \label{item:eps.real} If $v \in S_{real}$ then $\Span \left\{ a_v,\beta a_v,\beta ^2a_v \right\}$ is regular and embeds isometrically into $U \otimes K_v$.
\item \label{item:eps.indef} If $v \in S_{real}$ and $q_v$ is indefinite then $q_{\{ a_v,\beta a_v,\beta^2a_v {\}}^{\perp}}$ is indefinite. 
\item \label{item:eps.def} If $v\in S_{real}$ and $q_v$ is definite then the angles between $a_v,\beta a_v,\beta^2a_v$ are between $89^\circ$ and $91^\circ$.
\end{enumerate} 

Since $K^m$ is dense in $\prod_{v\in S_{real}\cup \left\{ w_0 \right\}} K_v^m$ and the conditions above are open, we can perturb $a_v$ keeping the above conditions and assume that there is a vector $b\in K^m$ such that $b_v=a_v$, for all $v\in S_{real}\cup \left\{ w_0 \right\}$. Since $\dim U = 6$, for every non-real place $v$, there is an isometric embedding $\Span \left\{ b_v,\beta b_v,\beta ^2 b_v \right\} \hookrightarrow U \otimes K_v$. By \eqref{item:eps.real} and Hasse--Minkowski, there is an isometric embedding $\Span \left\{ b,g_1b,g_1^2 b \right\} \hookrightarrow U$. By \eqref{item:eps.indef} and a similar argument to the one in the proof of Lemma \ref{lem:Witt}, there is  $\kappa \in \Theta_q(K)$ such that $\kappa \left( \Span \left\{ b,\beta b,\beta^2b \right\} \right) \subseteq U$. Let $c=k(b)$ and $\gamma =\kappa \beta \kappa ^{-1}$.

Denote $W=\Span \left\{ c, \gamma c, \gamma ^2c \right\} \subseteq U$. By weak approximation, there is an element $\eta \in \Theta_{q\restriction W}$ such that, for all $v\in S_{real}$ such that $q_v$ is definite, the angle between $\gamma \eta ^{-1}\gamma ^{-1}(c)$ and $\eta (c)$ is between $89^\circ$ and $91^\circ$. It follows that $[\eta ,\gamma ]\in \left( \gcl_{\Theta_q(K)}(\alpha) \right)^{\frac{2M}{\epsilon}} \cap \Theta_{q\restriction U}(K)$ is a 1-separated element.
\end{proof}

\begin{lemma}\label{lemma:rrrr21} Let $N$ be the constant from Lemma \ref{lemma:eps_to_1}. For every $1$-separated element $\alpha \in \Theta_q(K)$, there exists a $1$-separated element $\beta \in (\gcl_{\Theta_q(K)}(\alpha))^{N}$ such that $\gcl_{O_q(K)}(\beta)= \gcl_{\Theta_q(K)}(\beta)$.
\end{lemma}
\begin{proof}
By the Hasse principal, a subspace $U$ of $K^n$ is $q$-universal if for every place $v$, $U_v$ is a $q$-universal subspace $K_v^n$. For every non-real place $v$, any regular  quadratic space of dimension at least five is universal. Thus, a regular six dimensional subspace $U$ of $K^n$ is $q$-universal if and only if for every real place $v$, $U_v$ is $q$-universal. A simple approximation argument with respect to the finitely many real places implies that there exists a regular six dimensional $q$-universal subspace $U$ of $K^n$. By the choice of $N$, $(\gcl_{\Theta_q(K)}(\alpha))^{N}$ contains a strongly $1$-separated element $\beta\in \Theta_{q \restriction_{U^\perp}}(K)$. Lemma \ref{lemma:equal_gcl} implies that $\gcl_{O_q(K)}(\beta)= \gcl_{\Theta_q(K)}(\beta)$.\end{proof}

\begin{lemma}\label{lemma:tri} Let $\gamma \in \Gamma$ and  $a,b,c \in K^n$ be non-isotropic vectors such that
 $q(a)=q(b)=q(c)$ and $(a,\gamma a)=(a,c)=(b,c)$. 
Then there exists  $\delta \in (\gcl_{O_q(K)}(\gamma))^2$ such that $b=\delta a $.
\end{lemma}
\begin{proof}
By Witt Lemma, there exist $\sigma,\tau \in O_q(K)$ such that $\sigma(a,\gamma a)=(a,c)$ and $\tau(a,\gamma a)=(b,c)$. Thus $\delta a=b$ where $\delta:=\tau \gamma^{-1} \tau^{-1} \sigma \gamma \sigma^{-1}\in (\gcl_{O_q(K)}(\gamma))^2$.
\end{proof}

\begin{lemma}\label{lemma:kneser_arg2}  Let $N$ is the constant from Lemma \ref{lemma:eps_to_1}. Let $a,b \in K^n$ be non-isotropic vectors such that, $b \ne \pm a$, $q(a)=q(b)$ and the hyperplane $U:=\{c \in K^n \mid (a,c)=(b,c) \}$ is $q$-universal. Then for every $1$-separated element $\alpha \in \Theta_q(K)$,  there exists $\delta \in \left(\gcl_{\Theta_q(K)}(\alpha)\right)^{2{N}^2}$ such that $\gcl_{O_q(K)}(\delta) \subseteq \left(\gcl_{\Theta_q(K)}(\alpha)\right)^{2
{N}^2}$ and $\delta a= b$.
\end{lemma}

\begin{proof} 
By Lemma \ref{lemma:rrrr21}, there exists a $1$-separated element $\alpha^\star \in \big(\gcl_{\Theta_q(K)}(\alpha)\big)^{N}$ such that $\gcl_{O_q(K)}(\alpha^\star)=\gcl_{\Theta_q(K)}(\alpha^\star)$.

	For every $\beta \in \Theta_q(K)$, define a three dimensional quadratic space $(V_\beta,f_\beta)$ in the following way: $V_\beta$ is spanned over $K$ by $a',b',c'$, $f_\beta(a')=f_\beta(b')=f_\beta(c')=q(a)$, $(a',b')_{f_\beta}=(a,b)_q$ and $(a',c')_{f_\beta}=(b',c')_{f_\beta}=(a,\beta a)_q$. Let $d'$ be the orthogonal projection of $c'$ on $(Ka'+Kb')^\perp$. For every $\beta \in \Theta_q(K)$ for which $f_\beta$ is regular the following conditions are equivalent:
	\begin{enumerate}
		\item There exists an isometric embedding $\iota_\beta$ of $V_\beta$ in $K^n$ under which $a'$ and $b'$ are mapped to $a$ and $b$.
		\item $f_\beta(d') \in q\left((Ka+Kb)^\perp\right)$.
	\end{enumerate}
	
	Assume for the moment that there exists $\beta \in \Theta_q(K)$ for which the embedding $\iota_\beta$ exists (and hence $f_\beta$ is regular). By the definition of $N$, there exists  $\gamma \in \gcl_{\Theta_q(K)}(\alpha^\star)^{N}$ such that for every real place $v$, $\gamma$ is close enough to $\beta$ so that $f_\gamma(d') \in q\left((K_va+K_vb)^\perp\right)$ and $f_\gamma$ is regular.  Since $n \ge 7$, for every non-real $v$, $(K_va'+K_vb')^\perp$ is a universal subspace so $f_\gamma(d') \in q\left((K_va+K_vb)^\perp\right)$. The Hasse principal implies that $f_\gamma(d') \in q\left((Ka+Kb)^\perp\right)$ so an isometric embedding $\iota_\gamma:V_\gamma\rightarrow K^n$ exists.  Lemma \ref{lemma:tri} applied to $c:=\iota_\gamma(c')$ implies that there exists $$\delta \in \left(\gcl_{O_q(K)}(\gamma)\right)^2 \subseteq \left(\gcl_{O_q(K)}\left((\gcl_{\Theta_q(K)}(\alpha^\star))^{N}\right)\right)^{2}\subseteq\left(\gcl_{\Theta_q(K)}(\alpha^\star)\right)^{2N}\subseteq \left(\gcl_{\Theta_q(K)}(\alpha)\right)^{2{N}^2}$$ such that $\delta a =b$.

We now prove the existence of $\beta$.  Since $U$ is $q$-universal, there exists $c_1 \in U$ such that $q(a)=q(c_1)$. By Witt's theorem. there exists $\beta_1 \in O_q(K)$ such that $\beta_1a=c_1$. Since $U$ is $q$-universal, Lemma \ref{lemma:rrr} implies that there exists $\beta_2 \in O_{q \restriction_U }(K)$ such that $\beta:=\beta_2\beta_1\in \Theta_q(K)$. Denote $c:=\beta a \in U$. Since $c \in U$, $(a,c)=(b,c)$ and  $(a,\beta a)=(a,c)=(b,c)$. Thus, there exists an embedding of $V_\beta$ in $K^n$ which sends $a'$, $b'$ and $c'$ to $a$, $b$ and $c$. Since $f$ is anisotropic, $V_\beta$ is regular. 
\end{proof}

\begin{lemma}\label{lemma:almdone2}  Let $N$ be the constant from Lemma \ref{lemma:eps_to_1}. Let $\alpha$ be a $1$-separated element and $a,b \in K^n$ be non-isotropic vectors such that $\tau_a\tau_b \in \Theta_q(K)$.  Then $\tau_a\tau_b \in \left(\gcl_{\Theta_q(K)}(\alpha)\right)^{4{N}^2}$.
\end{lemma}
\begin{proof}  
We can assume that $\tau_a\tau_b\ne 1$ so, since $q$ is anisotropic, $U:=Ka+Kb$ is a regular subspace. We claim that there exists a non-isotropic $c \in Ka+Kb$ such that $(Kc)^\perp$ is $q$-universal. A simple approximation argument with respect to the finitely many real places implies that there exists a non-isotropic $c\in U$ such that for every real place $v$, $(K_vc)^\perp$ is $q$-universal. Since $n \ge 6$, for every non-real place $v$, $(K_vc)^\perp$ is universal. The Hasse principal implies   that $(Kc)^\perp$ is $q$-universal.
 
Since $\tau_a\tau_b\tau_c \in O_{q \restriction_U}(K)$ and every element in $O_{q \restriction_U}(K)$ of determinant $-1$ is a reflection, there exists a non-isotropic $d \in (Ka+Kb)$ such that  $\tau_a\tau_b=\tau_c\tau_d \in \Theta_q(K)$. Since $\tau_c\tau_d \in \Theta_q(K)$, by multiplying $d$ with a scalar, we can assume that $q(c)=q(d)$. Since $q$ is anisotropic,  $q(e)\ne 0$  and $d=\tau_e(c)$ where $e:=c-d$. Thus,

$$\tau_a\tau_b=\tau_c\tau_d =\tau_c\tau_{\tau_e(c)}=[\tau_c,\tau_e]=\tau_{\tau_c(e)}\tau_e.$$

 Since  $\tau_a\tau_b\ne 1$, $e \ne \pm\tau_c(e)$. Moreover, $\{x \in K^n \mid (e,x)=(\tau_c(e),x)\}=(Kc)^\perp$ is $q$-universal.  By Lemma \ref{lemma:kneser_arg2}, there exists $\beta \in \Theta_q(K)$ such that $ \gcl_{O_q(K)}(\beta)\subseteq \left(\gcl_{\Theta_q(K)}(\alpha)\right)^{2{N}^2}$ and $\beta (e)=\tau_c (e)$. It follows that 
  $$\tau_a\tau_b=\tau_{\tau_c(e)}\tau_e=\tau_{\beta(e)}\tau_e=[\beta,\tau_e]\in\left( \gcl_{O_q(K)}(\beta)\right)^2\subseteq \left(\gcl_{\Theta_q(K)}(\alpha)\right)^{4{N}^2}.$$
\end{proof}

\begin{proposition}\label{prop:spin_simple2} Let $N$ be the constant from Lemma \ref{lemma:eps_to_1}. Then for every $1$-separated element  $\alpha \in \Spin_q(K)$, $(\gcl_{\Spin_q(K)}(\alpha))^{12(n+1){N}^2}=\Spin_q(K)$.
\end{proposition} 
\begin{proof} 
Let $\alpha \in \Spin_q(K)$ be a $1$-separated element and let $\alpha^\star$ be its image in $\Theta_q(K)$. 
For every non-isotropic pairwise orthogonal vectors $a,b,c \in K^n$ such that $q(a)=q(b)=q(c)$, and any lifts $\delta_1$ and $\delta_2$ of $\tau_a\tau_b$ and $\tau_b\tau_c$ to $\Spin_{f_s}(K)$, $[\delta_1,\delta_2]=-1$. Since $\{-1\}$ is the kernel of the epimorphism from $\Spin_q(K)$ to $\Theta_q(K)$, it is enough to prove that $\big(\gcl_{\Theta(K)}(\alpha^\star)\big)^{4(n+1){N}^2}=\Theta_q(K)$.

 Let  $\gamma \in \Theta_q(K)$. Since every element in $\SO_q(K)$ is the product of at most $n$ reflections, there are an even $2 \le k \le n$ and non-isotropic $a_1,\ldots,a_n \in K$ such that $\gamma=\tau_{a_1}\cdots \tau_{a_k}$. Lemma 101:11 of \cite{Ome71} implies that there exists a regular 3-dimensional subspace of $K^n$ and $b_1,\ldots,b_k \in U$ such that for every $1 \le i \le k$, $q(a_i)=q(b_i)$. Denote  $\beta:=\tau_{b_1}\cdots \tau_{b_k}$. Then $\beta \in \Theta_{q\restriction_U}(K)$ and since $\dim U =3$, there exist non-isotropic $c_1,c_2 \in U$ such that $\beta:=\tau_{c_1}\tau_{c_2}$. For every $1 \le i\le k$, denote $d_i=\tau_{a_{i+1}}\cdots \tau_{a_k}(b_i)$ (so $d_k=b_k$). Recall that 
for every $\delta \in \Spin_q(K)$ and every non-isotropic $x \in K^n$, $\delta\tau_x\delta^{-1}=\tau_{\delta(x)}$. By Lemma \ref{lemma:almdone2},
$$
\gamma=\gamma\beta\beta^{-1}=(\tau_{a_1}\cdots \tau_{a_k})(\tau_{b_1}\cdots \tau_{b_k})(\tau_{c_2}\tau_{c_1})=(\tau_{a_1}\tau_{d_1})\cdots (\tau_{a_k}\tau_{d_k})(\tau_{c_2}\tau_{c_1}) \in \left(\gcl_{\Theta_q(K)}(\alpha^\star)\right)^{4(k+1)N^2}.
$$

\end{proof}

\begin{corollary} \label{cor:spin*.perfect} \begin{enumerate}
\item For every $I\in \mathcal{I} ^{K^*}$, the congruence subgroup $\Spin_q(O^*)[I]$ normally generates $\Spin_q(K^*)$.
\item $\Spin_q(K^*)$ is perfect.
\end{enumerate} 
\end{corollary}

	

\begin{remark}
The proofs of Lemma \ref{lemma:almdone2} and  of Proposition \ref{prop:spin_simple2} show that every element of $\Spin_q(K)$ is the product of at most $n+2$ commutators. 
\end{remark}

\end{document}